\pdfoutput=1
\documentclass[twoside,a4paper,reqno]{amsart}

\usepackage{amsmath}
\usepackage{amsthm}
\usepackage{amssymb}
\usepackage{amsfonts}
\usepackage{amsxtra}
\usepackage{stmaryrd}
\usepackage{float}
\usepackage{microtype}
\usepackage{calc}
\usepackage{subfig}
\usepackage{mathtools} 
\usepackage{caption}
\usepackage{booktabs}

\usepackage{colortbl}
\usepackage{tabulary}
\usepackage{diagbox}
\usepackage{xcolor}

\usepackage{amsaddr}

\usepackage{upgreek}

\usepackage{enumitem,ulem}
\setenumerate{label={\rm (\alph{*})}}

\usepackage[rose,frak,operator]{paper_diening}

\usepackage{color}

\usepackage{graphicx}
\usepackage{mathrsfs}
\usepackage{hyperref}

\DeclareMathAlphabet\bscal{OMS}{cmsy}{b}{n}
\DeclareMathAlphabet\mathbfscr{OMS}{mdugm}{b}{n}

\newtheorem{theorem}[equation]
{Theorem}
\newtheorem{proposition}
[equation]{Proposition}
\newtheorem{lemma}
[equation]{Lemma}
[equation]{Corollary}

\newtheorem{asum}[equation]{Assumption}{\bfseries}{\upshape}
\newtheorem{alg}[equation]{Algorithm}{\bfseries}{\upshape}

\newtheorem{remark}[equation]{Remark}

\newtheorem{definition}[equation]{Definition}

\numberwithin{equation}{section}

\providecommand{\meantmp}[2]{#1\langle{#2}#1\rangle}
\providecommand{\mean}[1]{\meantmp{}{#1}}

\providecommand{\jumptmp}[2]{#1\llbracket{#2}#1\rrbracket}
\providecommand{\jump}[1]{\jumptmp{}{#1}}

\providecommand{\avgtmp}[2]{#1\{{#2}#1\}}
\providecommand{\avg}[1]{\avgtmp{}{#1}}

\providecommand{\PiDG}{{\Uppi_{h}^{\ell}}}
\providecommand{\PiDGn}{{\Uppi_{h_n}^{\ell}}}

\providecommand{\Vo}{\smash{\mathaccent23 V}}
\providecommand{\Ho}{\smash{\mathaccent23 H}}
\providecommand{\Qo}{\smash{\mathaccent23 Q}}
\providecommand{\Xhk}{\smash{X_h^\ell}}
\providecommand{\Vhk}{\smash{V_h^\ell}}

\providecommand{\Qhkc}{\smash{Q_{h,c}^\ell}}

\providecommand{\Vhnk}{\smash{V_{h_n}^\ell}}

\providecommand{\Qhnkco}{\smash{{\mathaccent23 Q}_{h_n,c}^\ell}}

\providecommand{\SSS}{\boldsymbol{\mathcal{S}}}

\newcommand{\Ghk}{\boldsymbol{\mathcal{G}}_h^\ell}
\newcommand{\Ghnk}{\boldsymbol{\mathcal{G}}_{h_n}^\ell\! }
\newcommand{\Dhk}{\boldsymbol{\mathcal{D}}_h^\ell}
\newcommand{\Dhnk}{\boldsymbol{\mathcal{D}}_{h_n}^\ell}
\newcommand{\Divhk}{\mathcal{D}\dot{\iota}\nu_h^\ell}
\newcommand{\Divhnk}{\mathcal{D}\dot{\iota}\nu_{h_n}^\ell}
\newcommand{\Rhk}{\boldsymbol{\mathcal{R}}_h^\ell}
\newcommand{\Rhnk}{\boldsymbol{\mathcal{R}}_{h_n}^\ell}
\newcommand{\Rhks}{\boldsymbol{\mathcal{R}}_h^{\smash{\ell,{\textup{sym}}}}}
\newcommand{\Rhnks}{\boldsymbol{\mathcal{R}}_{h_n}^{\smash{\ell,{\textup{sym}}}}}
\newcommand{\WDG}{W^{1,p}(\mathcal{T}_h)}

\newcommand{\WDGn}{W^{1,p}(\mathcal{T}_{h_n})}

\providecommand{\divo}{\mathrm{div}\,}
\providecommand{\tr}{\mathrm{tr}\,}

\begin{document}

	\title[Analysis of a non-conforming approximation of evolution equations]{Analysis of a fully-discrete, non-conforming approximation of evolution equations and applications\vspace{-1mm}}
	
	\author{A.~Kaltenbach and M.~\Ruzicka}
	\email{alex.kaltenbach@mathematik.uni-freiburg.de}
	\address{Institute of Applied Mathematics, Albert--Ludwigs--University Freiburg, Ernst--Zermelo--Straße 1,\\ 79104 Freiburg, Germany\vspace{-1mm}}
	\email{rose@mathematik.uni-freiburg.de}

	\begin{abstract}
		In this paper, we consider a fully-discrete approximation of an abstract
		evolution equation deploying a non-conforming spatial approximation
		and finite differences in time (Rothe--Galerkin method). The main
		result is the convergence of the discrete solutions to a weak
		solution of the continuous problem. Therefore, the result can be
		interpreted either as a justification of the numerical
		method~or~as~an~alternative way of constructing weak solutions. We
		formulate the problem in the very general and abstract setting of
		so-called non-conforming Bochner pseudo-monotone operators, which allows for a unified treatment of several evolution problems. 
		Our abstract results for non-conforming Bochner pseudo-monotone operators allow to establish (weak) convergence just by verifying a few natural assumptions on the operators time-by-time and on the discretization spaces. Hence, applications and extensions to several other evolution problems can be  performed easily. We exemplify the applicability of our approach on several DG schemes for the unsteady $p$-Navier--Stokes problem. The results of some numerical~experiments~are~reported~in~the~final~section.\vspace{-7mm}
	\end{abstract}
	
	
	\date{\small \today }
	\maketitle

	\section{Introduction}
	We consider the numerical approximation of an abstract evolution
	equation
	\begin{align}
		\begin{aligned}
			\frac{d{v}}{dt}(t)+A(t)({v}(t))&= {f}(t)&&\quad\text{
				in } V^*\,,
			\\
			{v}(0)&={v}_0&&\quad\text{ in }H\,,
		\end{aligned} \label{eq:1.1}
	\end{align}
	deploying a non-conforming Rothe--Galerkin scheme. Here, $V\hookrightarrow H
	\cong H^*\hookrightarrow V^*$ is a given evolution~triple, $I\coloneqq \left(0,T\right)$ a
	finite time interval, ${v}_0\in H$ an initial value, ${{f}\in L^{p'}(I,V^*)}$,
	$p \in (1,\infty)$, a right-hand~side,~and $A(t)\colon V\rightarrow V^*$, $t\in I$, a family of
	operators.

	Recently, the existence theory for abstract evolution problems
	with Bochner pseudo-monotone operators in
	\cite{alex-rose-hirano}, based on the convergence of a
	Galerkin approximation, was extended in \cite{br-fully} to
	the~convergence~proof of a fully-discrete Rothe--Galerkin
	approximation (cf.~\cite{AL83} for an early
	contribution). However, the result in \cite{br-fully} is not
	applicable to the treatment of problems describing the flow of
	incompressible fluids. Introducing the notion of quasi
	\mbox{non-conform}ing Bochner pseudo-monotonicity in
	\cite{alex-rose-nonconform}, the methods in \cite{br-fully}
	has been extended to so-called quasi non-conforming
	approximations, which include, e.g., discretely
	divergence-free, conforming finite element approximations,
	and, thus, is applicable to the treatment of problems
	describing the flow of incompressible fluids (cf.~the recent
	results in \cite{tscherpel-phd}, \cite{sueli-tscherpel} using
	a different approach).~On~the~other~hand,~va-rious results proving convergence rates,
	under additional regularity assumptions,~are~available~for~numerical approximations of $p$-Laplace
	evolution equations~and~related~problems
	(cf.~\cite{BarLiu94,rulla,NoSaVe00,FeOePr05,DiEbRu07,P08,CHP10,BaNoSa14,sarah-phd,BDN}).
	
	However, the conformitiy of the finite element approximation in \cite{alex-rose-nonconform} excludes non-conforming~finite element approximations, i.e., the Discontinuous Galerkin (DG) method.
	The main aim of this paper is develop an abstract framework which shows that a
	fully-discrete non-conforming Rothe--Galerkin~\mbox{approximation} converges also for such
	problems (cf.~Theorem \ref{5.17}). Even though the framework is rather abstract it is easily
	applicable to many problems, since we show that it is enough to check
	a few easily verifiable conditions (cf.~conditions
	(\hyperlink{AN.1}{AN.1})--(\hyperlink{AN.4}{AN.4}) in
	Proposition~\ref{3.9},~and~conditions~(\hyperlink{NC.1}{NC.1}),
	\hyperlink{NC.2}{NC.2}) in~Definition~\ref{3.1}).
	
	A prototypical example for the flow of incompressible non-Newtonian
	fluids are the following equations describing the unsteady
	motion of incompressible shear-dependent fluids 
	\begin{align}
		\begin{split}
			\begin{alignedat}{2}
				\partial_t \bfv-\divo \big (
				(\delta
				+|\bfD \bfv|)^{p-2}\bfD\mathbf{v}\big )+[\nabla \bfv]\bfv +\nabla q&=\bfg+\divo \bfG &&\quad\text{ in }I\times\Omega\,,
				\\
				\divo \bfv&=0&&\quad\text{ in }I\times\Omega\,,
				\\
				\bfv&=\mathbf{0}&&\quad\text{ on }I\times \partial\Omega\,,
				\\
				\bfv(0)&=\bfv_0&&\quad\text{ in }\Omega\,.
			\end{alignedat}
		\end{split}\label{eq:p-NS}
	\end{align}
	Here, $\Omega \subseteq \mathbb{R}^d$,
	$d\ge 2$, is a bounded polyhedral
	Lipschitz domain, 
	$I\coloneqq (0,T)$, $T<\infty$, a finite time~interval, and $\delta\ge 0 $ and $p\in (1,\infty)$ material parameters of
	the shear-dependent fluid. Furthermore,
	$\bfv\colon \overline{I\times\Omega}\to \mathbb{R}^d$ denotes~the~velocity vector field,
	$\bfg\colon I\times\Omega\to \mathbb{R}^d$ and $\bfG\colon
	I\times\Omega\to \mathbb{R}^{d\times d}$  jointly describe a given external
	force, $\bfv_0\colon\Omega\to \mathbb{R}^d$ is an initial condition,
	$q\colon I\times\Omega\to \mathbb{R}$ is the pressure, and
	$\bfD\bfv\coloneqq \frac{1}{2}(\nabla \bfv+\nabla
	\bfv^\top)\colon I\times\Omega\to \mathbb{R}^{d\times d}$ denotes the symmetric
	gradient.

	We define for $p\hspace{-0.15em}>\hspace{-0.15em}\frac{3d+2}{d+2}$, the function spaces $\Vo(0)\hspace{-0.15em}\coloneqq  \hspace{-0.15em}W^{1,p}_{0,\divo }(\Omega)$ as the closure of
	${\mathcal{V}\hspace{-0.15em}\coloneqq \hspace{-0.15em}\{\bfv\hspace{-0.15em}\in\hspace{-0.15em}
		C_0^\infty(\Omega) \hspace{-0.1em}\mid \hspace{-0.1em}\divo\bfv\hspace{-0.15em}=\hspace{-0.15em} 0\}}$ in
	$W^{1,p}_0(\Omega)^d$, $\Ho(0)\hspace{-0.1em}\coloneqq\hspace{-0.1em} L^2_{\divo }(\Omega)$ as the closure of $\mathcal{V}$ in $L^2(\Omega)$, operators ${S,B\colon \Vo(0)\hspace{-0.1em}\to \hspace{-0.1em}\Vo(0)^*}$~for~all~${\bfv,\bfw\hspace{-0.1em}\in\hspace{-0.1em} \Vo(0)}$ via
	\begin{gather*}
		\langle S\bfu,\bfv\rangle_V\coloneqq  \int_\Omega (\delta
		+|\bfD\bfv|)^{p-2}{\bfD \bfv}:{\bD \bfw}\,\mathrm{d}x
		\quad \text{ and } \quad \langle
		B\bfv,\bfw\rangle_V\coloneqq \int_\Omega{[\nabla\bfv]\bfv\cdot \bfw\,\mathrm{d}x}\,,
	\end{gather*}
	and a functional $\bff\in L^{p'}(I,\Vo(0)^*)$ for all $\bfw\in
	L^p(I,\Vo(0))$~via
	\begin{gather*}
		\int_I\langle \bff,\bfv\rangle_{\Vo(0)}\,\mathrm{d}t\coloneqq  \int_I\int_\Omega {\mathbf{g}\cdot \bfw+\mathbf{G}:\nabla\bfw\,\mathrm{d}t\,\mathrm{d}x} \,.
	\end{gather*}
	Then, if $\bfv_0\in \Ho(0)$,
	$\bfg \in L^{p'}(I, L^{p'}(\Omega)^d)$, and $\bfG \in L^{p'}(I, L^{p'}(\Omega)^{d\times d})$, \eqref{eq:p-NS}  can be re-written~as~the~abstract 
	evolution equation
	\begin{align}
		\begin{split}
			\begin{alignedat}{2}
				\frac{d\bfv}{dt}(t)+S(\bfv(t))+B(\bfv(t))&=\ff(t)&&\quad\text{
					in }\Vo(0)^*\,,
				\\
				\bfv(0)&=\bfv_0&&\quad\text{ in }\Ho(0)\,.
				\label{eq:p-NS2}
			\end{alignedat}
		\end{split}
	\end{align}
	where, in the notation of~\eqref{eq:1.1}, we set $A(t)\coloneqq S+B\colon \Vo(0)\to \Vo(0)^*$~for~every~${t\in I}$.
	
	In Section \ref{sec:7}, we will discuss several DG schemes for
	\eqref{eq:p-NS2} and show, by using the abstract theory developed in
	Section \ref{sec:3}, its weak convergence to a weak solution of \eqref{eq:p-NS2}. 
	To the best of the authors' knowledge, there are no rigorous
	convergence results for DG Rothe--Galerkin schemes of
	nonlinear evolutions equations in the
	literature (cf.~\cite{HH21,HH22,KR05,SGS13} for results proving
	convergence rates for the incompressible Navier--Stokes
	equations under additional regularity assumptions).\enlargethispage{5mm}

	The paper is organized as follows: In Section \ref{sec:2}, we collect
	some basic facts about evolution equations, Bochner--Lebesgue spaces
	and Bochner--Sobolev spaces. In Section \ref{sec:3}, we introduce the
	concepts~of~Bochner non-conforming approximations and Bochner
	non-conforming pseudo-monotonicity. Further,~we~develop~an abstract
	theory for a non-conforming Rothe--Galerkin scheme and show well-posedness (cf.~Proposition~\ref{5.1}) and stability (cf.~Proposition
	\ref{apriori}) of the scheme as well as weak convergence of a diagonal
	subsequence to a weak solution of the original problem
	(cf.~Theorem \ref{5.17}), which is the main result of the
	paper.~In~Section~\ref{sec:7}, we apply the abstract theory to several
	DG schemes of the $p$-Navier--Stokes equations~\eqref{eq:p-NS}~and~show~well-posedness and stability of the scheme as well as weak convergence
	of a diagonal subsequence to a weak solution of \eqref{eq:p-NS}
	(cf.~Theorem \ref{thm:appl-p-NS}). We conclude the paper by some
	numerical experiments illustrating the theory in Section \ref{sec:8}.
	
	\section{Preliminaries on evolution equations}\label{sec:2}

	We \hspace{-0.1mm}use \hspace{-0.1mm}the \hspace{-0.1mm}usual \hspace{-0.1mm}notation \hspace{-0.1mm}for \hspace{-0.1mm}Bochner--Lebesgue \hspace{-0.1mm}spaces \hspace{-0.1mm}$L^p(I,X)$ \hspace{-0.1mm}and
	\hspace{-0.1mm}Bochner--Sobolev~\hspace{-0.1mm}spaces~\hspace{-0.1mm}$W^{1,p}(I,X)$, where $p \in [1,\infty]$ and $X
	$ is a Banach space (cf.~\cite{droniou}).
	We denote by $(V,H,j)$ an  \textit{evolution triple}, if $V$ is a
	reflexive~Banach~space, $H$ a Hilbert space, and $j\colon V\to H$
	a dense embedding, i.e., linear, injective, and bounded
	with $\overline{j(V)}^{\smash{\|\cdot \|_H}}=H$.  Moreover,
	let $R\colon H\to H^*$ be the Riesz
	isomorphism with respect~to~${(\cdot,\cdot)_H}$. Since $j$ is a dense
	embedding, its adjoint operator \mbox{$j^*\colon H^*\to V^*$} and,
	therefore, also $e\coloneqq j^*Rj\colon V \to V^*$ are embeddings~as~well. We call $e$ the
	\textit {canonical embedding} of $(V,H,j)$. Note that for every~${v,w\in V}$,~it~holds
	\begin{align}\label{eq:2.14}
		\langle ev,w\rangle_V=(jv,jw)_H\,.
	\end{align} 
	For an evolution triple $(V,H,j)$, $I\coloneqq \left(0,T\right)$,
	$T<\infty$, and $1\leq p\leq q\leq \infty$, we define the operators
	$j\colon L^p(I,V)\to L^p(I,H)$ and
	$j^*\colon L^{q'}(I,H^*)\to L^{q'}(I,V^*)$ 
	for every $v\in L^p(I,V)$~and~${h\in L^{q'}(I,H^*)}$~via
	\begin{alignat*}{3}
		(j v)(t)&\coloneqq j(v(t))&&\quad\text{ in }H\quad&&\text{ for a.e. }t\in I\,,\\
		(j^*h)(t)&\coloneqq j^*(h(t))&&\quad\text{ in }V^*&&\text{ for a.e. }t\in I\,.
	\end{alignat*}
	It is shown in \cite[Proposition~2.19]{alex-rose-hirano} that both $j\colon L^p(I,V)\to L^p(I,H)$ and
	$j^*\colon L^{q'}(I,H^*)\to L^{q'}(I,V^*)$  are embeddings, so-called
	\textit{induced embeddings}. 
	We define the  intersection space
	\begin{align*}
		L^p(I,V)\cap_{\smash{j}}L^q(I,H)\coloneqq \{v\in L^p(I,V)\mid jv\in L^q(I,H)\}\,, 
	\end{align*}
	which forms a Banach space (cf.~\cite[Proposition 2.7]{alex-rose-hirano}), if equipped with the canonical sum norm
	\begin{align*}
		\|\cdot\|_{	L^p(I,V)\cap_{\smash{j}}L^q(I,H)}\coloneqq \|\cdot\|_{L^p(I,V)}+\|j(\cdot)\|_{L^q(I,H)}\,.
	\end{align*}
	If $1<p\leq q<\infty$, then  the  intersection space $L^p(I,V)\cap_{\smash{j}}L^q(I,H)$ is reflexive (cf.~\cite[Proposition~2.9]{alex-rose-hirano}).
	For every  $v^*\in (L^p(I,V)\cap_{\smash{j}}L^q(I,H))^*$,
	there exist functions $g\in L^{p'}(I,V^*)$ and ${h\in L^{q'}(I,H^*)}$ 
	such that for every $w\in L^p(I,V)\cap_{\smash{j}}L^q(I,H)$, it holds
	\begin{align}
		\langle v^*,w\rangle_{L^p(I,V)\cap_{\smash{j}}L^q(I,H)}
		=\int_I{\langle g(t)+(j^*h)(t),w(t)\rangle_V\,\mathrm{d}t}\,,\label{eq:dual}
	\end{align}
	and
	$\|v^*\|_{(L^p(I,V)\cap_{\smash{j}}L^q(I,H))^*}
	\coloneqq \|g\|_{L^{p'}(I,V^*)}+\|h\|_{L^{q'}(I,H^*)}$,
	i.e., the dual space $(L^p(I,V)\cap_{\smash{j}}L^q(I,H))^*$ is
	isometrically isomorphic to 
	$\smash{L^{p'}(I,V^*)+j^*(L^{q'}(I,H^*))}$
	(cf.~\cite[Kapitel I, Bemerkung~5.13, Satz~5.13]{GGZ}), 
	which is a Banach space if equipped with the  norm
	\begin{align*}
		\|f\|_{L^{p'}(I,V^*)+j^*(L^{q'}(I,H^*))}\coloneqq \min_{\substack{g\in L^{p'}(I,V^*),\ h\in L^{q'}(I,H^*)\\f=g+j^*h}}{\|g\|_{L^{p'}(I,V^*)}+\|h\|_{L^{q'}(I,H^*)}}\,.
	\end{align*}
	
	\begin{definition}[Generalized time derivative]\label{2.15}
		Let $(V,H,j)$ be an evolution triple,
		$I\coloneqq \left(0,T\right)$, $T<\infty$, and
		$1< p\leq q<\infty$.  A function
		$v\in L^p(I,V)\cap_{\smash{j}}L^q(I,H)$ has a
		\textrm{generalized derivative with respect to the canonical
			embedding $e$ of $(V,H,j)$}, if there exists a function
		$v^*\in L^{p'}(I,V^*)+j^*(L^{q'}(I,H^*))$ such that for
		every $w\in V$ and $\varphi\in C_0^\infty(I)$, it holds
		\begin{align*}
			-\int_I{(j(v(s)),jw)_H\varphi^\prime(s)\,\mathrm{d}s}=
			\int_I{\langle v^*(s),w\rangle_{V}\varphi(s)\,\mathrm{d}s}\,.
		\end{align*}
		Since this function $v^*\in L^{p'}(I,V^*)+j^*(L^{q'}(I,H^*))$
		is unique (cf.~\cite[Proposition
		23.18]{zei-IIA}),~setting~${\frac{d_ev}{dt}\coloneqq v^*}$
		is well-defined. We define the \textrm {Bochner--Sobolev
			space with respect to $e$} via 
		\begin{align*}
			{W}^{1,p,q}_e(I,V,H)
			\coloneqq \Big\{v\in L^p(I,V)\cap_{\smash{j}}L^q(I,H)
			\,\Big|\, \exists\, \frac{d_ev}{dt}\in L^{p'}(I,V^*)+j^*(L^{q'}(I,H^*))\Big\}\,.
		\end{align*}
	\end{definition}

	\begin{proposition}[Integration-by-parts formula]\label{2.16}
		Let $(V,H,j)$ be an evolution triple, $I\coloneqq \left(0,T\right)$,~${T<\infty}$,  
		and $1<p\leq q<\infty$. Then, it holds:
		\begin{itemize}
			\item[(i)] The space ${W}^{1,p,q}_e(I,V,H)$ forms 
			a Banach space if equipped with the norm 
			\begin{align*}
				\|\cdot\|_{{W}^{1,p,q}_e(I,V,H)}
				\coloneqq \|\cdot\|_{L^p(I,V)\cap_{\smash{j}}L^q(I,H)}
				+\left\|\frac{d_e\,\cdot}{dt}\right\|_{L^{p'}(I,V^*)+{j^*}(L^{q'}(I,H^*))}\,.
			\end{align*}
			\item[(ii)] Given $v\in {W}^{1,p,q}_e(I,V,H)$, the
			function $jv\in L^q(I,H)$, defined via  
			$(jv)(t)\coloneqq j(v(t))$ in $H$
			for almost every $t\in I$, possesses a unique representation
			$j_cv\in C^0(\overline{I},H)$. Moreover, the resulting mapping
			${j_c\colon{W}^{1,p,q}_e(I,V,H)\to
				C^0(\overline{I},H)}$ is an embedding.
			\item[(iii)] For every $v,w\in {W}^{1,p,q}_e(I,V,H)$ and
			$t,t'\in \overline{I}$ with $t'\leq t$, it holds
			\begin{align*}
				\int_{t'}^t{\bigg\langle
					\frac{d_ev}{dt}(s),w(s)\bigg\rangle_V\,\mathrm{d}s}
				=\left[((j_cv)(s), (j_c
				w)(s))_H\right]^{s=t}_{s=t'}-\int_{t'}^t{\bigg\langle
					\frac{d_ew}{dt}(s),v(s)\bigg\rangle_V\,\mathrm{d}s}\,.
			\end{align*}
		\end{itemize}  
	\end{proposition}
	
	\begin{proof}
		See \cite[Kapitel IV, Satz 1.16,  Satz 1.17]{GGZ}.
	\end{proof}

	For an evolution triple $(V,H,j)$, $I\coloneqq \left(0,T\right)$, $T<\infty$, and
	$1<p\leq q< \infty$, we~call~an~operator 
	$\mathcal{A}\colon L^p(I,V)\cap_{\smash{j}}L^q(I,H)\to (L^p(I,V)\cap_{\smash{j}}L^q(I,H))^*$ 
	\textit{induced} by a time-dependent family of operators $A(t)\colon V\to V^*$, $t\in I$, 
	if for every $v,w\in L^p(I,V)\cap_{\smash{j}}L^q(I,H)$, it holds
	\begin{align}
		\langle\mathcal{A}v,w\rangle_{L^p(I,V)\cap_{\smash{j}}L^q(I,H)}
		=\int_I{\langle A(t)(v(t)),w(t)\rangle_V\,\mathrm{d}t}\,.\label{eq:induced}
	\end{align}
	
	\begin{definition}[Weak solution]
		\label{2.17}
		Let $(V,H,j)$ be an evolution triple,
		$I\coloneqq \left(0,T\right)$, $T<\infty$, and
		$1<p\leq q< \infty$. Furthermore, let $v_0\in H$, $f\in L^{p'}(I,V^*)$, and
		$\mathcal{A}\colon L^p(I,V)\cap_{\smash{j}}L^q(I,H)\to
		(L^p(I,V)\cap_{\smash{j}}L^q(I,H))^*$ be induced by a family of
		operators $A(t)\colon V\to V^*$, $t\in I$. Then, a function
		$v \in {W}^{1,p,q}_e(I,V,H)$ is called \textrm{weak
			solution} of the initial value problem \eqref{eq:1.1} if
		$(j_cv)(0)=v_0$ in $H $ and for every $\phi\in C^1_0(I,V)$,
		it holds
		\begin{align}\label{eq:2.18}
			\int_I{\bigg\langle\frac{d_ev}{dt}(t),\phi(t)\bigg\rangle_V\,\mathrm{d}t}
			+\int_I{\langle A(t)(v(t)),\phi(t)\rangle_V\,\mathrm{d}t}
			&=\int_I{\langle f(t),\phi(t)\rangle_V\,\mathrm{d}t}\,.
		\end{align}
	\end{definition}

	\section{Fully discrete non-conforming approximation of evolution equations}
	\label{sec:3}
	
	\subsection{Non-conforming pseudo-monotonicity}
	In this section, we adapt the  concepts of non-conforming
	approximations and non-conforming pseudo-monotonicity,
	introduced in \cite{kr-pnse-ldg-1}, to the evolutionary
	setting. Moreover, we lift the notion of non-conforming 
	pseudo-monotonicity to the Bochner--Lebesgue space level, 
	which  leads us to an extended notion of Bochner 
	pseudo-monotonicity~introduced in \cite{alex-rose-hirano},
	\cite{K19}, and \cite{alex-rose-nonconform}. Eventually, we
	formulate the non-conforming Rothe--Galerkin scheme and establish~its~well-posedness (i.e., existence of solutions), stability (i.e., a priori estimates), and  convergence, the main~result~of~the~paper.
	
	We will start by adapting the notion of non-conforming
	pseudo-monotonicity from \cite[Definition 4.2]{kr-pnse-ldg-1} to the
	evolutionary setting. 
	\begin{definition}\label{3.1}
		Let $(V,H,j) $ and $(X_n,Y,j)$, $n\in\mathbb{N}$, be evolution triples 
		such that $H\subseteq Y$ with $(\cdot,\cdot)_H=(\cdot,\cdot)_Y$ 
		in $H\times H$ and for every $n\in \mathbb{N}$, it holds
		$V\subseteq X_n$ with $\|\cdot\|_V=\|\cdot\|_{X_n}$ in $V$  and 
		$\|\cdot\|_Y\leq c_Y\|\cdot\|_{X_n}$~in~$X_n$, where $c_Y>0$ is 
		independent~of~$n\in \mathbb{N}$. Moreover, let $I\coloneqq \left(0,T\right)$, $T<\infty$, 
		and $p\in \left(1,\infty\right)$. A sequence of closed subspaces 
		$V_n\subseteq X_n$, $n\in \mathbb{N}$, is called a
		\textrm{Bochner non-conforming approximation 
			of $V$ with respect to $(X_n)_{n\in\mathbb{N}}$}, if the following properties are satisfied:
		\begin{itemize}
			\item[(NC.1)]\hypertarget{NC.1}{}  There exists a dense subset $D\subseteq V$ 
			such that for every $v\hspace{-0.1em}\in\hspace{-0.1em} D$, there exists a sequence $v_n\hspace{-0.1em}\in\hspace{-0.1em} V_n$,~${n\hspace{-0.1em}\in\hspace{-0.1em}\mathbb{N}}$, with $\|v_n-v\|_{X_n}\to 0 $~$(n\to\infty)$.
			\item[(NC.2)] \hypertarget{NC.2}{} For each sequence 
			$v_n\!\in\! L^p(I,V_{m_n})\cap_{\smash{j}} L^\infty(I,Y)$, $n\in \mathbb{N}$, 
			where ${(m_n)_{n\in\mathbb{N}}\subseteq\mathbb{N}}$~with~${m_n\to\infty}$~${(n\to\infty)}$, from 
			$\sup_{n\in \mathbb{N}}{\|v_n\|_{L^p(I,X_{m_n})\cap_{\smash{j}}L^\infty(I,Y)}}<\infty$, 
			it follows the existence of a cofinal subset $\Lambda\subseteq \mathbb{N}$ 
			and $v\in L^p(I,V)\cap_{\smash{j}} L^\infty(I,H)$ such that 
			$jv_n\ \smash{\overset{\ast}{\rightharpoondown}}\ jv$ in 
			$L^\infty(I,Y)$~${(\Lambda\ni n\to \infty)}$.
		\end{itemize}
	\end{definition}
	
	The following proposition demonstrates that the notion of a Bochner 
	non-conforming approximation is indeed a generalization of the
	notion of a  quasi non-conforming approximation introduced in~\cite{alex-rose-nonconform}.~In~Section~\ref{sec:7}, we will show that our motivating example, i.e,
	the approximation of divergence-free Sobolev functions via
	discretely divergence-free DG finite element spaces, 
	perfectly fits into the framework~of~Bochner~non-conforming approximations.
	
	\begin{proposition}\label{prop:quasi}
		Let $(V,H,j) $ and $(X,Y,j)$ be evolution triples and let
		$V_n\subseteq X$, $n\in \mathbb{N}$, be a quasi
		non-conforming approximation of $V$ in $X$ (cf.~\cite{alex-rose-nonconform}). Then, setting
		$X_n\coloneqq  X$, $n\in\mathbb{N}$,
		$(V_n)_{n\in\mathbb{N}}$ is a Bochner non-conforming approximation
		of $V$ with respect to $(X_n)_{n\in\mathbb{N}}$.
	\end{proposition}
	\begin{proof}
		\textit{ad (NC.1).} Due to \cite[(QNC.1)]{alex-rose-nonconform}, there exists a dense subset $D\subseteq V$ such that for every $v\in D$, there exists a sequence $v_n\in V_n$, $n\in \mathbb{N}$, with $v_n\to v$ in $X$ $(n\to \infty)$. Since $X_n\coloneqq X$ for all $n\in \mathbb{N}$,~we~obtain $\|v_n-v\|_{X_n}\to 0$ $(n\to \infty)$.
		
		\textit{ad (NC.2).}  Let
		$v_n\in L^p(I,V_{m_n})\cap_{\smash{j}} L^\infty(I,Y)$, $n\in \mathbb{N}$, 
		where ${(m_n)_{n\in\mathbb{N}}\subseteq\mathbb{N}}$~with~${m_n\to\infty}$~${(n\to\infty)}$, be a sequence such that
		$\sup_{n\in \mathbb{N}}{\|v_n\|_{L^p(I,X_{m_n})\cap_{\smash{j}}L^\infty(I,Y)}}<\infty$.  Since $X_n\coloneqq X$ for all $n\in \mathbb{N}$,~there exists a cofinal subset $\Lambda\subseteq \mathbb{N}$ and $v\in L^p(I,X)\cap_{\smash{j}} L^\infty(I,Y)$ such that $v_n\rightharpoonup v$ in $L^p(I,X)$ $(\Lambda \ni n\to \infty)$ and $jv_n\ \smash{\overset{\ast}{\rightharpoondown}}\ jv$ in 
		$L^\infty(I,Y)$~${(\Lambda\ni n\to \infty)}$. Due to \cite[(QNC.2)]{alex-rose-nonconform}, we obtain $v\in L^p(I,V)$, which,~in~turn,~yields~that $v\in L^p(I,V)\cap_{\smash{j}} L^\infty(I,H) $.
	\end{proof}

	The following proposition will be of crucial importance in the verification that the
	induced operators $\mathcal A_n\colon L^p(I,X_n)\cap_{\smash{j}}L^q(I,Y)\to (L^p(I,X_n)\cap_{\smash{j}}L^q(I,Y))^*$, $n\in \mathbb{N}$, cf.~\eqref{eq:induced}, 
	of a sequence of families of operators $A_n(t)\colon X_n\to X_n^*$, $t\in I$, $n\in \mathbb{N}$, are non-conforming Bochner 
	pseudo-monotone (cf.~Definition~\ref{3.4}). 
	
	\begin{proposition}\label{3.2}
		Let $(V_n)_{n\in\mathbb{N}}  $ be a Bochner non-conforming
		approximation of $V$ with respect to $(X_n)_{n\in\mathbb{N}}$. Then, the following
		statements apply:
		\begin{itemize}
			\item[(i)] For every sequence $v_n\in V_{m_n}$,
			$n\in \mathbb{N}$, where
			$(m_n)_{n\in\mathbb{N}}\subseteq\mathbb{N}$ with
			$m_n\to \infty$ $(n\to \infty)$, from
			$\sup_{n\in \mathbb{N}}{\|v_n\|_{X_{m_n}}}<\infty$, it
			follows the existence of a cofinal subset
			$\Lambda\subseteq \mathbb{N}$ and $v\in V$ such
			that $jv_n\weakto jv$ in $Y$ $(\Lambda\ni n\to \infty)$.
			\item[(ii)] For every sequence $v_n\in V_{m_n}$, $n\in \mathbb{N}$, 
			with $\sup_{n\in \mathbb{N}}{\|v_n\|_{X_{m_n}}}<\infty$, where 
			$(m_n)_{n\in\mathbb{N}}\subseteq\mathbb{N}$ with $m_n\to \infty$ 
			$(n\to \infty)$, and $v\in V$, the following statements are equivalent:
			\begin{itemize}
				\item[(a)]  $jv_n\weakto jv$ in $Y$ $(n\to\infty)$.
				\item[(b)]  $P_Hjv_n\weakto jv$ in $H$ $(n\to\infty)$, 
				where $P_H\colon Y\to H$ is the orthogonal projection~of~$Y$~into~$H$.
			\end{itemize}
		\end{itemize}
	\end{proposition}
	
	\begin{proof}
		\textit{ad (i).} We define $\tilde v_n(\cdot)\coloneqq v_n\in L^p(I,V_{m_n})\cap_{\smash{j}} L^\infty(I,Y)$ 
		for every $n\in \mathbb{N}$. Then, we have that
		$\sup_{n\in \mathbb{N}}{\|\tilde v_n\|_{L^p(I,X_{m_n})\cap_{\smash{j}}L^\infty(I,Y)}}<\infty$. 
		Therefore, (\hyperlink{NC2}{NC.2}) provides the existence of a cofinal subset 
		$\Lambda\subseteq \mathbb{N}$ and $\tilde v\in L^p(I,V)\cap_{\smash{j}} L^\infty(I,H)$
		such that $j\tilde v_n\,\smash{\overset{\ast}{\rightharpoondown}}\,j\tilde v$
		in $L^\infty(I,Y)$ $(\Lambda\ni n\to \infty)$.\label{key}
		Apart from that, since also $(jv_n)_{n\in \mathbb{N}}\subseteq Y$ is bounded, due to 
		$\|jv_n\|_Y=\|j\tilde v_n\|_{L^\infty(I,Y)}$ for all $n\in \mathbb{N}$, we may assume
		that there exists  $h\in Y$ such that $jv_n\weakto h$ in $Y$ 
		$(\Lambda\ni n\to\infty)$. It is~easy~to~see, that $(j\tilde v)(t)= h$~in~$Y$~for~almost~every~${t\in I}$. As $\tilde v(t)\in V$ for almost every $t\in I$, we find that $h\in R(j)$, i.e.,
		there exits $v\in V$~such~that~${jv=h}$~in~$Y$,~i.e., ${h\in  H}$ and ${jv_n\weakto jv}$ in $Y$ $(\Lambda\ni n\to\infty)$.
		
		\textit{ad (ii).} \textit{(a) $\boldsymbol{\Rightarrow}$ (b).} 
		Follows from the weak continuity of $P_H\colon Y\to H$.
		
		\textit{(b) $\boldsymbol{\Rightarrow}$  (a).} From {(i)} 
		we obtain a subsequence $(v_n)_{n\in \Lambda}$ with cofinal $\Lambda\subseteq \mathbb{N}$ 
		and an~element~${\tilde{v}\in V}$ such that $jv_n\weakto j\tilde{v}$ in $Y$ 
		$(\Lambda\ni n\to \infty)$. From the weak continuity of  $P_H\colon Y\to H$, 
		we conclude that $P_Hjv_n\weakto P_Hj\tilde{v}=j\tilde{v}$ in $H$  $(\Lambda\ni n\to \infty)$. 
		In consequence, we have that $j\tilde{v}=jv$ in $H$, which, by virtue of the injectivity of
		$j\colon V\to H$, implies that $\tilde{v}=v$ in $V$ and, thus, $jv_n\weakto jv$ in $Y$ 
		$(\Lambda \ni n\to  \infty)$.
		Since this argumentation stays valid for each subsequence of 
		$(jv_n)_{n\in \mathbb{N}}\subseteq Y$, $jv\in H$ is  a weak accumulation point of 
		each subsequence of $(jv_n)_{n\in \mathbb{N}}\subseteq Y$. The standard convergence principle (cf.~\cite[Kap. I, Lemma 5.4]{GGZ}) 
		yields $jv_n\weakto jv$ in $Y$ $(n\to  \infty)$. 
	\end{proof}

	\begin{definition}\label{3.4.0}
		Let $(V_n)_{n\in\mathbb{N}}  $ be a Bochner non-conforming
		approximation of $V$ with respect to
		$(X_n)_{n\in\mathbb{N}}$.
		A sequence of operators
		$A_n\colon X_n\to X_n^*$, $n\in \mathbb{N}$, is said to be
		\textrm{non-conforming pseudo-monotone with respect to
			$(V_n)_{n\in\mathbb{N}}$ and
		}$A\colon V\to V^*$,~if~for~every~sequence $v_n\in V_{m_n}$,
		$n\in\mathbb{N}$, where
		$(m_n)_{n\in\mathbb{N}}\subseteq\mathbb{N}$ with
		$m_n\to \infty$ $(n\to\infty)$, from
		\begin{gather}
			\sup_{n\in\mathbb{N}}\;\|v_n\|_{X_{m_n}}<\infty\,,
			\qquad jv_n\;\;\weakto\;\;jv\quad\text{ in }Y\quad(n\to \infty)\,,\label{eq:3.4.a}
			\\
			\limsup_{n\to\infty}{\langle A_{m_n}v_n,v_n-v\rangle_{X_{m_n}}}\leq 0\,,\label{eq:3.4.b}
		\end{gather}
		for every $w\in V$, it follows that 
		\begin{align*}
			\langle Av,v-w\rangle_V\leq	
			\liminf_{n\to\infty}{\langle A_{m_n}v_n,v_n-w\rangle_{X_{m_n}}}\,.
		\end{align*}
		In particular, due to Proposition \ref{3.2} (i), \eqref{eq:3.4.a}  implies that 
		$v\in V$ .
	\end{definition}
	
	\begin{lemma}\label{3.4.a} 
		Let $(V_n)_{n\in\mathbb{N}} $ be a Bochner non-conforming
		approximation of $V$ with respect to
		$(X_n)_{n\in\mathbb{N}}$.  Moreover, let
		$A_n,B_{n}\colon X_n\to X_n^*$, $n\in \mathbb{N}$, be
		non-conforming pseudo-monotone with respect to
		$(V_n)_{n\in\mathbb{N}}$~and $A,B\colon V\to V^*$,
		respectively. Then, ${A_n+B_{n}\colon X_n\to
			X_n^*}$,~${n\in \mathbb{N}}$, is non-conforming
		pseudo-monotone with respect to $(V_n)_{n\in\mathbb{N}}$ and
		$A+B\colon V\to V^*$.
	\end{lemma}
	
	\begin{proof}
		In view of Proposition \ref{3.2} (i), the assertion is proved in
		\cite[Lemma 4.3]{kr-pnse-ldg-1}.
	\end{proof}
	
	Now, we can lift the adapted notion of non-conforming
	pseudo-monotonicity to the notion of non-conforming Bochner pseudo-monotonicity.
	
	\begin{definition}\label{3.4}
		Let $(V_n)_{n\in\mathbb{N}} $ be a Bochner non-conforming
		approximation of $V$ with respect to
		$(X_n)_{n\in\mathbb{N}}$, and let $q \in [p,\infty)$.
		Then, a sequence of operators $\mathcal{A}_n\colon L^p(I,X_n)\cap_{\smash{j}} L^q(I,Y)
		\to (L^p(I,X_n)\cap_{\smash{j}} L^q(I,Y))^*$, $n\in \mathbb{N}$, is called
		\textrm{non-conforming~Bochner pseudo-monotone with respect to 
			$(V_n)_{n\in\mathbb{N}}$ and} $\mathcal{A}\colon L^p(I,V)\cap_{\smash{j}} L^q(I,H)
		\to (L^p(I,V)\cap_{\smash{j}}L^q(I,H))^*$ if for a sequence 
		$v_n\in L^p(I,V_{m_n})\cap L^\infty(I,Y)$, $n\in\mathbb{N}$, 
		where $(m_n)_{n\in\mathbb{N}}\subseteq\mathbb{N}$ with $m_n\to \infty$~$(n\to\infty)$,~from
		\begin{gather}
		\sup_{n\in\mathbb{N}}\;\|v_n\|_{L^p(I,X_{m_n})}<\infty\,,\qquad
		jv_n\;\;\overset{\ast}{\rightharpoondown}\;\;
		jv\quad\text{ in } L^\infty(I,Y)\quad (n\to\infty)\,,
		\label{eq:3.6}
		\\[-1mm]
		P_H(jv_n)(t)\;\;\weakto\;\;
		(jv)(t) \quad\text{ in }H\quad(n\to \infty)\quad\text{for a.e. }t\in I\,,\label{eq:3.7}
	\end{gather}
and
\begin{align}
	\limsup_{n\to\infty}{\langle \mathcal{A}_{m_n}v_n,
		v_n-v\rangle_{L^p(I,X_{m_n})\cap_{\smash{j}} L^q(I,Y)}}\leq 0\,,\label{eq:3.8}
\end{align}
for every $w\in L^p(I,V)\cap_{\smash{j}} L^q(I,H)$, it follows that $$\langle \mathcal{A}v,v-w\rangle_{L^p(I,V)\cap_{\smash{j}} L^q(I,H)}
\leq	\liminf_{n\to\infty}{\langle \mathcal{A}_{m_n}v_n,
	v_n-w\rangle_{L^p(I,X_{m_n})\cap_{\smash{j}} L^q(I,Y)}}\,.$$
In particular,
due to Definition~\ref{3.1},~\eqref{eq:3.6}~implies~that~${v\in L^p(I,V)\cap_{\smash{j}} L^\infty(I,H)}$.
\end{definition}
The following theorem gives sufficient conditions on a family of
non-conforming pseudo-monotone operators such that the  induced
operators are non-conforming Bochner pseudo-monotone.
\begin{theorem}\label{3.9}
Let $(V_n)_{n\in\mathbb{N}} $ be a Bochner non-conforming
approximation of $V$ with respect to
$(X_n)_{n\in\mathbb{N}}$. 
Moreover, let
${A(t)\colon\hspace*{-0.05em}V\hspace*{-0.05em}\to\hspace*{-0.05em}
	V^*}$,~${t\hspace*{-0.05em}\in\hspace*{-0.05em} I}$,~be a
family that possesses for some $q \in [p,\infty)$
the~following properties:
\begin{itemize}
	\item[(A.1)] \hypertarget{A.1}  $A(t)\colon V\to V^*$ is for almost every $t\in I$ demi-continuous.
	\item[(A.2)] \hypertarget{A.2}  $(t\mapsto A(t)v)\colon I\to V^*$ is Bochner measurable  for every $v\in V$.
	\item[(A.3)] \hypertarget{A.3}  There exist a constant   $\gamma\ge 0$  such that 	
	for almost every $t\in I$ and every $v,w\in V$, it holds
	\begin{align*}
		\vert\langle A(t)v,w\rangle_V\vert \leq \gamma\,\big[1+\|jv\|_H^q+\|v\|_V^p+\|jw\|_H^q+\|w\|_V^p\big]\,.
	\end{align*}
\end{itemize}
Moreover, let $A_n(t)\colon X_n\to X_n^*$, $t\in I$, $n\in \mathbb{N}$, be families of operators with the following properties:
\begin{itemize}
	\item[(AN.1)] \hypertarget{AN.1}  $A_n(t)\colon X_n\to X_n^*$, $n\in \mathbb{N}$, 
	are for almost every $t\in I$ demi-continuous 
	and non-conforming pseudo-monotone 
	with respect to $(V_n)_{n\in \mathbb{N}}$ and $A(t)\colon V\to V^*$.
	\item[(AN.2)] \hypertarget{AN.2} $(t\mapsto A_n(t)x_n)\colon I\to X_n^*$ is 
	Bochner measurable for every $x_n\in X_n$ and $n\in \mathbb{N}$.
	\item[(AN.3)] \hypertarget{AN.3} There exist constants  $c_0>0$ 
	and $c_1,c_2\ge 0$, independent of $n\in \mathbb{N}$, such that 
	for almost every $t\in I$ and every $v_n\in V_n$, it holds
	\begin{align*}
		\langle A_n(t)v_n,v_n\rangle_{X_n}\ge c_0\,\|v_n\|_{X_n}^p-c_1\,\|jv_n\|_Y^2-c_2\,.
	\end{align*}
	\item[(AN.4)] \hypertarget{AN.4} There exist constants   $\gamma\ge 0$ 
	and $\lambda\in \left(0,c_0\right)$, independent of $n\in \mathbb{N}$, such that 	
	for almost every $t\in I$ and every $x_n,y_n\in X_n$, it holds
	\begin{align*}
		\vert\langle A_n(t)x_n,y_n\rangle_{X_n}\vert \leq \lambda\,\|x_n\|_{X_n}^p
		+\gamma\,\big[1+\|jx_n\|_Y^q+\|jy_n\|_Y^q+\|y_n\|_{X_n}^p\big]\,.
	\end{align*}
	
\end{itemize}
Then, the through the families $A_n(t)\colon X_n\to X_n^*$, $t\in I$, $n\in \mathbb{N}$, 
induced sequence of operators 
$\mathcal{A}_n\colon L^p(I,X_n)$ ${\cap_{\smash{j}} L^q(I,Y)\to (L^p(I,X_n)\cap_{\smash{j}} L^q(I,Y))^*}$, $n\in \mathbb{N}$, 
is well-defined, bounded, and non-conforming Bochner 
pseudo-monotone with respect to $(V_n)_{n\in \mathbb{N}}$ and the through $ A(t)\colon V\to V^* $, $ t\in I$, induced operator 
$\mathcal{A}\colon L^p(I,V)\cap_{\smash{j}} L^q(I,H)\to (L^p(I,V)\cap_{\smash{j}} L^q(I,H))^*$.
\end{theorem}

\begin{proof}
For the well-definedness and boundedness of $\mathcal A$, and
$\mathcal A_n$, $n\in \mathbb{N}$, we refer to
\cite[Proposition~4.2]{alex-rose-nonconform}, since, in the
proof of these points there, only the demi-continuity is used. We align the proof~of~the~\mbox{non-con}-forming Bochner pseudo-monotonicity into four main steps:

\textit{1. Collecting information.} Let 
$v_n\in L^p(I,V_{m_n})\cap_{\smash{j}} L^\infty(I,Y)$, $n\in\mathbb{N}$, 
where $(m_n)_{n\in\mathbb{N}}\subseteq\mathbb{N}$ with $m_n\to \infty$
$(n\to\infty)$, be a sequence satisfying
\eqref{eq:3.6}--\eqref{eq:3.8}, in particular, let 
$K\coloneqq \sup_{n\in\mathbb{N}}\;\|j v_n\|_{L^\infty(I,Y)}<\infty$. 
For sake of simplicity, we assume that ${m_n=n}$ for all $n\in \mathbb{N}$.
Then, we fix an arbitrary $w\in L^p(I,V)\cap_{\smash{j}} L^q(I,H)$, and 
choose~a~subsequence~$(v_n)_{n\in\Lambda}$, with
$\Lambda\subseteq\mathbb{N}$, such that
\begin{align}
	\lim_{\Lambda\ni n\to\infty}{
		\langle \mathcal{A}_nv_n,v_n-w\rangle_{L^p(I,X_n)\cap_{\smash{j}} L^q(I,Y)}}=
	\liminf_{n\to\infty}{\langle\mathcal{A}_n
		v_n,v_n-w
		\rangle_{L^p(I,X_n)\cap_{\smash{j}} L^q(I,Y)}}\,.\label{eq:3.10}
\end{align}
Due to \eqref{eq:3.7}, there exists a subset
$E_0\subseteq I$, with $I\setminus E_0$ a null set, 
such that for every $t\in E_0$, it holds
\begin{align}
	P_H(jv_n)(t)\weakto(jv)(t)\quad\text{ in }H\quad(n\to \infty)\,.\label{eq:3.11}
\end{align}
Using (\hyperlink{AN.3}{AN.3}) and
(\hyperlink{AN.4}{AN.4}), for every $w\in L^p(I,V)\cap_{\smash{j}} L^q(I,H)$
and almost every~$t\in I$, we obtain
\begin{align}
	\begin{aligned}
		&\langle A_n(t)(v_n(t)),v_n(t)
		-w(t)\rangle_{X_n}
		\\
		&\ge c_0\,\|v_n(t)\|_{X_n}^p-
		c_1\,\|j(v_n(t))\|_Y^2-c_2
		-\langle
		A_n(t)(v_n(t)),w(t)\rangle_{X_n}
		\\
		&\ge 	(c_0-\lambda)\,\|v_n(t)\|_{X_n}^p
		-c_1\,K^2-c_2-\gamma\,\big[1+K^q+\|(jw)(t)\|_Y^q+\|w(t)\|_V^p\big]\,.
	\end{aligned}\label{eq:3.12}
\end{align}
We set $\mu_{w}\coloneqq c_1K^2+c_2
+\gamma\big[1+K^q+\|(jw)(\cdot)\|_Y^q+\|w(\cdot)\|_V^p\big]\in L^1(I)$ for all $w\in L^p(I,V)\cap_{\smash{j}} L^q(I,H)$. Then,  \eqref{eq:3.12}  for every $w\in L^p(I,V)\cap_{\smash{j}} L^q(I,H)$, for almost every $t\in I$, and all $n\in \Lambda$ reads
\begin{align}
	\langle
	A_n(t)(v_n(t)),v_n(t)-
	w(t)\rangle_{X_n}\ge
	(c_0-\lambda)\,\|v_n(t)\|_{X_n}^p
	-\mu_{w}(t)\,. \tag*{$(\ast)_{w,n,t}$}
\end{align}
Next, we define
\begin{align*}
	E_1\coloneqq 
	\big \{t\in E_0 \mid& \;A_n(t)\colon X_n\to X_n^*\text{ is
		non-conforming pseudo-monotone with respect to }(V_n)_{n\in \mathbb{N}}
	\\&\text{ and }A(t)\colon V\to V^*,
	\vert\mu_{v}(t)\vert<\infty\text{
		and }(\ast)_{v,n,t}\text{ holds for all }n\in\Lambda\big \}\,. 
\end{align*}
From the defining properties of $E_1$, it follows
immediately that $I\setminus E_1$ is a null set. \\[-2mm] 

\textit{2. Intermediate objective.} Our next objective is 
to verify that for every $t\in E_1$,  we have that
\begin{align}
	\liminf_{\Lambda\ni n\to\infty}
	{\langle A_n(t)(v_n(t)),v_n(t)
		-v(t)\rangle_{X_n}}\ge 0\,. \tag*{$(\ast\ast )_{n,t}$}
\end{align}
To this end, let us fix an arbitrary
$t\in E_1$ and introduce the set
\begin{align*}
	\Lambda_t\coloneqq \big\{n\in\Lambda\mid
	\langle A_n(t)(v_n(t)),v_n(t)
	-v(t)\rangle_{X_n}< 0\big\}\,.
\end{align*}
We assume, without loss of generality, that $\Lambda_t$
is not finite. Otherwise, $(\ast\ast )_{n,t}$ would already hold~true~for this specific $t\in E_1$ and nothing would be left to
do. But if $\Lambda_t$ is not finite, then, we have that
\begin{align}
	\limsup_{\Lambda_t\ni n\to\infty}
	{\langle A_n(t)(v_n(t)),v_n(t)-
		v(t)\rangle_{X_n}}\leq 0\,.\label{eq:3.13}
\end{align}
Then, the definitions of $\Lambda_t$ and $(\ast)_{v,n,t}$, for every $n\in\Lambda_t$, imply that
\begin{align}
	(c_0-\lambda)\,\|v_n(t)\|_{X_n}^p\leq\langle A(t)
	(v_n(t)),v_n(t)-
	v(t)\rangle_{X_n}+
	\vert\mu_{v}(t)\vert<\vert\mu_{v}(t)\vert
	<\infty\,.\label{eq:3.14}
\end{align}
Due to $\lambda\! <\! c_0$, from   \eqref{eq:3.14}, it follows that  $\sup_{n\in \Lambda_t}{\!\|v_n(t)\|_{X_n}}\!\!<\!\infty$. 
Thus, also using~that~${v(t)\!\in\! V}$~and~\eqref{eq:3.11}, 
Proposition \ref{3.2} (ii) implies that $j(v_n(t))\weakto j(v(t))
$ in $Y$ $(\Lambda_t\ni n\to \infty)$. 
Then, the non-conforming pseudo-monotonicity of $A_n(t)\colon\hspace*{-0.1em} X_n\hspace*{-0.1em}\to\hspace*{-0.1em} X_n^*$, $n\hspace*{-0.1em} \in\hspace*{-0.1em} \mathbb{N}$,
with~respect~to~$(V_n)_{n\in \mathbb{N}}$ and $A(t)\colon V\to V^*$, eventually, yields that
\begin{align*}
	\liminf_{\Lambda_t\ni n\to\infty}
	{\langle A_n(t)(v_n(t)),v_n(t)
		-v(t)\rangle_{X_n}}\ge \langle A(t)(v(t)),v(t)
	-v(t)\rangle_V=0\,.
\end{align*}
Eventually, owing to
$\langle A_n(t)(v_n(t)), v_n(t) -v(t)\rangle_{X_n}\ge
0$ for all $n\in\Lambda\setminus\Lambda_t$,
$(\ast\ast)_t$ holds for all $t\in E_1$. \\[-2mm] 

\textit{3. Switching to the image space level.} In this passage, we
verify the existence of a cofinal subset $\Lambda_0\subseteq\Lambda$ 
such that for almost every $t\in I$, it holds
\begin{align}
	\begin{gathered}
		\sup_{n\in \Lambda_0}{\|v_n(t)\|_{X_n}}<\infty\,,
		\qquad j(v_n(t))\weakto
		j(v(t))\quad\text{ in }Y\quad(\Lambda_0\ni n\to \infty)\,,\\
		\limsup_{\Lambda_0\ni n\to\infty}
		{\langle A_n(t)(v_n(t)),v_n(t)
			-v(t)\rangle_{X_n}}\leq 0\,,
	\end{gathered}
	\label{eq:3.15}
\end{align}
which enables us to exploit almost everywhere
the non-conforming pseudo-monotonicity of the operator
family.  For \eqref{eq:3.15}, we use that
$\langle   A_n(t)(v_n(t)),v_n(t)-v(t)\rangle_{X_n}\hspace*{-0.1em}\ge\hspace*{-0.1em}-\mu_{v}(t)$
for all $t\in E_1$and~$n\in\Lambda$ (cf.~$(\ast)_{v,n,t}$), which yields
that Fatou's lemma (cf.~\cite[Theorem 1.18]{Roub}) is
applicable. It yields, also~using~\eqref{eq:3.8},
that\enlargethispage{3mm}
\begin{align}\label{eq:3.16}
	\begin{aligned}
		0&\leq
		\int_I{\liminf_{\Lambda\ni n\to\infty}
			{\langle A_n(s)(v_n(s)),v_n(s)-
				v(s)\rangle_{X_n}}\,\mathrm{d}s}
		\\
		&\leq
		\liminf_{\Lambda\ni n\to\infty}{\int_I{\langle
				A_n(s)(v_n(s)),v_n(s)-v(s)\rangle_{X_n}\,\mathrm{d}s}}
		\leq  0\,.
	\end{aligned}
\end{align}
Thus, proceeding as in the proof of \cite[Proposition~4.2]{alex-rose-nonconform}, we find  a further subsequence
$v_n\in L^p(I,V_n)\cap_{\smash{j}} L^\infty(I,Y)$, $n\hspace{-0.15em}\in \hspace{-0.15em}\Lambda_0$,
with $\Lambda_0\hspace{-0.15em}\subseteq\hspace{-0.15em}\Lambda$ and a subset $E_2\hspace{-0.15em}\subseteq\hspace{-0.15em} E_1$, 
with $ I\setminus E_2$ a null set, such that~for~every~${t\hspace{-0.15em}\in \hspace{-0.15em}E_2}$,~it~holds
\begin{align}
	\lim_{\Lambda_0\ni n\to\infty}
	{\langle A_n(t)(v_n(t)),v_n(t)-v(t)\rangle_{X_n}}= 0\,.\label{eq:3.19}
\end{align}
Then,  \eqref{eq:3.14} and \eqref{eq:3.19}  imply for all $t\in E_2$ 
\begin{align*}
	\limsup_{\Lambda_0\ni n\to\infty}
	{(c_0-\lambda)\,\|v_n(t)\|_{X_n}^p\leq
		\limsup_{\Lambda_0\ni n\to\infty}
		{\langle A_n(t)(v_n(t)),v_n(t)
			-v(t)\rangle_{X_n}}+
		\vert\mu_{v}(t)\vert}
	=\vert\mu_{v}(t)\vert<\infty\,,
\end{align*}
i.e., $\sup_{n\in \Lambda_0}{\|v_n(t)\|_{X_n}}<\infty$ 
for all $t\in E_2$. Thus, 
Proposition \ref{3.2} (ii) and \eqref{eq:3.11} yield~for~all~${t\in E_2}$
\begin{align}
	j(v_n(t))\;\;\weakto\;\; j(v(t))\quad\text{ in }
	Y\quad(\Lambda_0\ni n\to\infty)\,.\label{eq:3.20}
\end{align} 
The relations 
\eqref{eq:3.19} and \eqref{eq:3.20} are just \eqref{eq:3.15}. \\[-2mm] 

\textit{4. Switching to the Bochner--Lebesgue level.} 
The non-conforming
pseudo-monotonicity of the operators $A_n(t)\colon X_n\to X_n^*$  
with respect to $(V_n)_{n\in \mathbb{N}}$ and $A(t)\colon V\to V^*$ for
all $t\in E_2$ implies for~almost~every~${t\in I}$ 
\begin{align*}
	\langle A(t)(v(t)),v(t)-
	w(t)\rangle_V\leq
	\liminf_{\Lambda_0\ni n\to\infty}
	{\langle A_n(t)(v_n(t)),v_n(t)
		-w(t)\rangle_{X_n}}\,.
\end{align*}
Due to $(\ast)_{w,n,t}$, we have that
$\langle A_n(t)(v_n(t)),v_n(t)
-w(t)\rangle_{X_n}\ge-\mu_{w}(t)$
for almost every $t\in I$~and~all~${n\in\Lambda_0}$. Using the definition of the induced operator \eqref{eq:induced},~Fatou's~lemma~and~\eqref{eq:3.10},~we~find~that
\begin{align*}
	\langle \mathcal{A}v,v
	-w\rangle_{L^p(I,V)\cap_{\smash{j}} L^q(I,H)}
	&\leq
	\int_I{\liminf_{\Lambda_0\ni n\to\infty}
		{\langle A_n(s)(v_n(s)),v_n(s)
			-w(s)\rangle_{X_n}}\,\mathrm{d}s}
	\\
	&\leq \liminf_{\Lambda_0\ni n\to\infty}
	{\int_I{\langle A_n(s)(v_n(s)),v_n(s)
			-w(s)\rangle_{X_n}\,\mathrm{d}s}}
	\\
	&=\lim_{\Lambda_0\ni n\to\infty}
	{\langle \mathcal{A}_nv_n,v_n
		-w\rangle_{L^p(I,X_n)\cap_{\smash{j}} L^q(I,Y)}}
	\\
	&\leq \liminf_{n\to\infty}{\langle \mathcal{A}_n
		v_n,v_n-w\rangle_{L^p(I,X_n)\cap_{\smash{j}} L^q(I,Y)}}\,.
\end{align*}
Since $w\in L^p(I,V)\cap_{\smash{j}} L^q(I,H)$ was arbitrary, 
this completes the proof
of Proposition~\ref{3.9}.~
\end{proof}

\subsection{Time discretization}
\label{sec:5}

We now prepare our time discretization and prove some relevant
properties. Let $X$ be a Banach space, and $I\coloneqq \left(0,T\right)$, ${T<\infty}$. 
For $K\in\mathbb{N}$, we define ${\tau\coloneqq \tau_K\coloneqq \frac{T}{K}}$, ${I_k^\tau\coloneqq \left((k-1)\tau,k\tau\right]}$, $k=1,\dots,K$, and 
$\mathcal{I}_\tau \coloneqq \{I_k^\tau\}_{k=1,\dots,K}$. Moreover, we denote by 
\begin{align*}
\mathcal{P}_0(\mathcal{I}_\tau,X)\coloneqq \big\{v\colon I\to X\mid v(s)
=v(t)\text{ in }X\text{ for all }t,s\in I_k^\tau,k=1,\dots,K\big\}
\subseteq L^\infty(I,X)\,,
\end{align*}
the \textit{space of piece-wise constant functions
with respect to $\mathcal{I}_\tau$}. For a 
finite~sequence~$(v^k)_{k=0,\dots,K}\subseteq  X$,~${K\in \mathbb{N}}$, the
\textit{backward difference quotient} operator is
defined via
\begin{align*}
d_\tau v^k\coloneqq \frac{1}{\tau}(v^k-v^{k-1})\quad\text{ in }X\,,\quad k=1,\dots,K\,.
\end{align*}
In addition, we denote by
$\overline{v}^\tau\in \mathcal{P}_0(\mathcal{I}_\tau,X)$, 
the \textit{piece-wise~constant interpolant}, and by 
$\hat{v}^\tau\in W^{1,\infty}(I,X)$, the \textit{piece-wise affine interpolant},  
for every $t\in I_k^\tau$ and $k=1,\dots,K$, $K\in \mathbb{N}$, defined via
\begin{align}\label{eq:polant}
\overline{v}^\tau(t)\coloneqq v^k\,,\qquad\hat{v}^\tau(t)
\coloneqq \Big(\frac{t}{\tau}-(k-1)\Big)v^k+\Big(k-\frac{t}{\tau}\Big)v^{k-1}\quad\text{ in }X\,.
\end{align}
If $(X,Y,j)$ is an evolution triple and $(v^k)_{k=0,\dots,K}\hspace{-0.1em}\subseteq\hspace{-0.1em} X$, $K\hspace{-0.1em}\in\hspace{-0.1em} \mathbb{N}$,
a finite sequence, then, for~any~${k,l\hspace{-0.1em}=\hspace{-0.1em}0,\dots,K}$, it holds the 
\textit{discrete integration-by-parts formula} 
\begin{align}
\int_{k\tau}^{l\tau}{\bigg\langle \frac{d_e\hat{v}^\tau}{dt}(t)\,,
\overline{v}^\tau(t)\bigg\rangle_X\,\mathrm{d}t}
\ge \frac{1}{2}\|jv^l\|_Y^2-\frac{1}{2}\|jv^k\|_Y^2\,,\label{eq:4.2}
\end{align}
which follows from the identity
$\langle d_\tau ev^k,v^k\rangle_X
=\frac{1}{2}d_\tau\|jv^k\|_Y^2+\frac{\tau}{2}\|d_\tau jv^k\|_Y^2$
for every $k=1,\dots,K$.

For the discretization of the right-hand side in \eqref{eq:1.1}, 
we 
employ the \textit{Clem\'ent~$0$-order quasi-interpolation operator}
$\mathscr{J}_\tau\colon L^p(I,X)\to
\mathcal{P}_0(\mathcal{I}_\tau,X)$,
for~every~${v\in L^p(I,X)}$, defined via
\begin{align*}
\mathscr{J}_\tau[v]\coloneqq \sum_{k=0}^K{[v]_k^\tau\chi_{I_k^\tau}}\in
\mathcal{P}_0(\mathcal{I}_\tau,X)\,,\qquad\text{where }\quad[v]_k^\tau\coloneqq \dashint\nolimits_{\hspace*{-1.5mm}I_k^\tau}{v(s)\,\mathrm{d}s}\in X\,.
\end{align*}
\begin{proposition}\label{4.4}
For every $v\in L^p(I,X)$, it holds:\enlargethispage{3mm}
\begin{itemize}
\item[(i)] $\mathscr{J}_\tau[v]\to v$ in $L^p(I,X)$ $(\tau\to 0)$, 
i.e., $\bigcup_{\tau>0}\mathcal{P}_0(\mathcal{I}_\tau,X)$ is dense in $L^p(I,X)$.
\item[(ii)] $\sup_{\tau>0}{\|\mathscr{J}_\tau[v]\|_{L^p(I,X)}}\leq \|v\|_{L^p(I,X)}$.
\end{itemize}
\end{proposition}

\begin{proof}
See \cite[Remark 8.15]{Roub}.
\end{proof}

Since we treat non-autonomous evolution equations, we need
to discretize the time-dependent~family of operators in
\eqref{eq:1.1}.  To this end, we assume that the assumptions
of Theorem \ref{3.9} are satisfied and that $n \hspace*{-0.1em}\in\hspace*{-0.1em} \setN$ is
fixed but arbitrary. 
The \textit{$k$-th temporal mean}
$[A_n]^\tau_k\colon X_n\hspace*{-0.1em}\to\hspace*{-0.1em} X_n^*$,
$k=1,\dots,K$,~of~${A_n(t)\colon X_n\to X^*_n}$,~${t\in  I}$, for every
$w\in X_n$, is defined via
\begin{align*}
[A_n]^\tau_k w\coloneqq \dashint\nolimits_{\hspace*{-1.5mm}I_k^\tau}{A_n(s)w\,\mathrm{d}s}\in X_n^*\,.
\end{align*} 
Then, the \textit{Clement $0$-order quasi-interpolant}
$\mathscr{J}_\tau[A_n](t)\colon X_n\to X_n^*$, $t\in I$, of
$A_n(t)\colon X_n\to X_n^*$, $t\in I$, for every $w\in X_n$, is
defined via  
\begin{align*}
\mathscr{J}_\tau[A_n](t)w\coloneqq \sum_{k=1}^{K}{\chi_{I_k^\tau}(t)[A_n]^\tau_kw}\in X_n^*\quad\text{ for a.e. }t\in I\,.
\end{align*} 
The \textit{Clement $0$-order quasi-interpolant}
$\mathscr{J}_\tau[\mathcal{A}_n]\colon L^p(I,X_n)\cap_{\smash{j}}L^q(I,Y)\to 
(L^p(I,X_n)\cap_{\smash{j}}L^q(I,Y))^*$ of
$\mathcal{A}_n\colon L^p(I,X_n)\cap_{\smash{j}}L^q(I,Y)\to
(L^p(I,X_n)\cap_{\smash{j}}L^q(I,Y))^*$, for every
$v,w\in L^p(I,X_n)\cap_{\smash{j}}L^q(I,Y)$,~is~defined~via
\begin{align*}
\langle \mathscr{J}_\tau[\mathcal{A}_n]v,w\rangle_{L^p(I,X_n)\cap_{\smash{j}}L^q(I,Y)}
\coloneqq \int_I{\langle \mathscr{J}_\tau[A_n] (t)(v(t)),w(t)\rangle_{X_n}\,\mathrm{d}t}\,.
\end{align*} 
Note that $\mathscr{J}_\tau[\mathcal{A}_n]$ is just the induced operator of
the family of operators ${\mathscr{J}_\tau[A_n](t)\colon X_n\to X_n^*}$,~${t \in I}$.
\begin{proposition}
\label{4.6}
Let the assumptions
of Theorem \ref{3.9} be satisfied and let $n \in \setN$ be
fixed but arbitrary.  
\begin{itemize}
\item[(i)] $[A_n]^\tau_k\colon X_n\to X_n^*$ is well-defined, bounded,  
pseudo-monotone, and satisfies: \\[-3mm]
\begin{itemize}
	\item[(i.a)] $\langle [A_n]^\tau_kx_n,y_n\rangle_{X_n}
	\leq \lambda\,\|x_n\|_{X_n}^p+\gamma\,[1+\|jx_n\|_Y^q+\|jy_n\|_Y^q+\|y_n\|_{X_n}^p]$
	for every $x_n,y_n\in X_n$. \\[-3mm]
	\item[(i.b)] $\langle [A_n]^\tau_kv_n,v_n\rangle_{X_n}
	\ge c_0\,\|v_n\|_{X_n}^p-c_1\,\|jv_n\|_Y^2-c_2$ for  every $v_n\in V_n$. \\[-3mm]
\end{itemize}
\item[(ii)]  $\mathscr{J}_\tau[A_n](t)\colon X_n\to X_n^*$, $t\in I$, 
satisfies the conditions (\hyperlink{AN.1}{AN.1})--(\hyperlink{AN.4}{AN.4}). \\[-3mm]

\item[(iii)] $\mathscr{J}_\tau[\mathcal{A}_n]\!\colon\!L^p(I,X_n)\cap_{\smash{j}}\!L^q(I,Y)
\!\to\!(L^p(I,X_n)\cap_{\smash{j}}\!L^q(I,Y))^*$~is~well-defined,~bounded~and~satisfies: \\[-3mm]
\begin{itemize}
	\item[(iii.a)] There holds	$\langle \mathscr{J}_\tau[\mathcal{A}_n]v_\tau,w\rangle_{L^p(I,X_n)\cap_{\smash{j}}L^q(I,Y)}=\langle \mathcal{A}_nv_\tau, \mathscr{J}_\tau[w]\rangle_{L^p(I,X_n)\cap_{\smash{j}}L^q(I,Y)}$
	for every \linebreak${v_\tau\in \mathcal{P}_0(\mathcal{I}_\tau,X_n)}$ and
	$w\in L^p(I,X_n)\cap_{\smash{j}}L^q(I,Y)$.
	\item[(iii.b)] If the family $v_\tau\in\mathcal{P}_0(\mathcal{I}_\tau,X_n)$,
	$\tau>0$, satisfies $\sup_{\tau>0}{\|v_\tau\|_{L^p(I,X_n)\cap_{\smash{j}}L^q(I,Y)}}<\infty$, then $	\mathcal{A}_nv_\tau-\mathscr{J}_\tau[\mathcal{A}_n]v_\tau\weakto{0}$ in $(L^p(I,X_n)\cap_{\smash{j}}L^q(I,Y))^*$ $(\tau\to 0)$.
	\item[(iii.c)] If  $v_\tau\in\mathcal{P}_0(\mathcal{I}_\tau,X_n)$, 
	then $\|\mathscr{J}_\tau[\mathcal{A}_n]v_\tau\|_{(L^p(I,X_n)\cap_{\smash{j}}L^q(I,Y))^*}
	\leq \|\mathcal{A}_nv_\tau\|_{(L^p(I,X_n)\cap_{\smash{j}}L^q(I,Y))^*}$.
\end{itemize}
\end{itemize}
\end{proposition}

\begin{proof}
See \cite[Proposition 5.2]{alex-rose-nonconform}.
\end{proof}

\subsection{Analysis of the fully-discrete non-conforming approximation}

\label{sec:6}

In this section, we formulate the exact framework of a non-conforming Rothe--Galerkin approximation, establish its well-posedness (i.e., the existence of iterates), and
its stability (i.e., the boundedness of the corresponding double sequence of piece-wise
constant and piece-wise affin interpolants). Furthermore, we prove the main result of this paper, Theorem \ref{5.17}, which shows the weak convergence of a diagonal subsequence toward~a~weak~solution~of~\eqref{eq:1.1}.

\begin{asum}\label{asum}
{\upshape
Let  $I\coloneqq \left(0,T\right)$, $T<\infty$, and $1<p<\infty$. 
We make the following assumptions:\enlargethispage{5mm}
\begin{itemize}
\item[(i)] \textit{Spaces:} $(V,H,j)$ and $(X_n,Y,j)$, $n\in \mathbb{N}$, 
are as in Definition \ref{3.1} and the  sequence $(V_n)_{n\in\mathbb{N}}$  is a non-conforming approximation of $V$ with respect to $(X_n)_{n\in\mathbb{N}}$.
\item[(ii)] \textit{Initial data:} $v_0\in H$, $v_n^0\in V_n$, $n\in\mathbb{N}$, 
with ${v_n^0\to v_0}$~in~$Y$~${(n\to \infty)}$ and 
${\sup_{n\in \mathbb{N}}{\|jv_n^0\|_Y}\leq \|v_0\|_H}$.
\item[(iii)] \textit{Right-hand side:} Let $f\in L^{p'}(I,V^*)$ and 
$f_n\in L^{p'}(I,X_n^*)$, $n\in \mathbb{N}$,  be a sequence such that
\begin{itemize}
	\item[] (BN.1) \hypertarget{BN.1}{} $\sup_{n\in \mathbb{N}}{\|f_n\|_{L^{p'}(I,X_n^*)}}<\infty$.
	\item[] (BN.2)\hypertarget{BN.2}{} For  every sequence $v_n\in L^p(I,V_{m_n})\cap_{\smash{j}}L^\infty(I,Y)$, $n\in \mathbb{N}$, where $(m_n)_{n\in\setN}\subseteq\setN$ with $ m_n\to \infty $ $(n\to\infty)$, from both $\sup_{n\in \mathbb{N}}\|v_n\|_{L^p(I,X_{m_n})\cap_{\smash{j}}L^\infty(I,Y)}<\infty$ and $v_n\,\smash{\overset{\ast}{\rightharpoondown}}\, v$ in $L^\infty(I,Y)$ $(n\to \infty)$, it follows that $\int_I{\langle f_{m_n}(s),v_n(s)\rangle_{X_{m_n}}\,\mathrm{d}s}\to \int_I{\langle f(s),v(s)\rangle_V\,\mathrm{d}s}$ $ {(n\to\infty)}$.
\end{itemize}
\item[(iv)] \textit{Operators:} $A_n(t)\colon X_n\to X_n^*$, $n\in \mathbb{N}$,
and $A(t)\colon V\to V^*$, $t\in I$, are families operators satisfying 
(\hyperlink{AN.1}{AN.1})--(\hyperlink{AN.4}{AN.4}) and 
(\hyperlink{A.1}{A.1})--(\hyperlink{A.3}{A.3}) with respect to 
$(V_n)_{n\in \mathbb{N}}$, respectively. Moreover, let denote
${\mathcal{A}_n\colon L^p(I,X_n)\cap_{\smash{j}}L^\infty(I,Y)\to (L^p(I,X_n)\cap_{\smash{j}}L^q(I,Y))^*}$,
$n\in \mathbb{N}$,  and $\mathcal{A}\colon L^p(I,V)\cap_{\smash{j}}L^\infty(I,H)\to (L^p(I,V)\cap_{\smash{j}}L^q(I,H))^*$
 the corresponding induced operators. In addition, we assume that the operators 
$A_n(t)\colon X_n\to X_n^*$, $n\in \mathbb{N}$, $t\in I$, satisfy the following condition:\vspace{1mm}
\begin{itemize}
	\item[] (AN.5) \hypertarget{AN.5}{} $(\mathrm{id}_{V_n})^*\circ A_n(t)\circ\mathrm{id}_{V_n}\colon V_n\to V_n^*$, $n\in \mathbb{N}$, are for almost every $t\in I$ pseudo-monotone.
\end{itemize}
\end{itemize}
}
\end{asum}

\begin{remark}
Note that if $\textup{dim}(V_n)<\infty$ for every $n\in \mathbb{N}$, then (\hyperlink{AN.5}{AN.5}) is satisfied  if and only if $(\mathrm{id}_{V_n})^*\circ A_n(t)\circ\mathrm{id}_{V_n}\colon V_n\to V_n^*$, $n\in \mathbb{N}$, are for almost every $t\in I$ continuous.
\end{remark}

Furthermore, for every $n\in \mathbb{N}$, we define $H_n\coloneqq j(V_n)\subseteq Y$ equipped with $(\cdot,\cdot)_Y$, denote the restriction of $j$ to $V_n$ by $j_n\colon V_n\hspace*{-0.1em}\to \hspace*{-0.1em} H_n$  and the corresponding Riesz isomorphism~with~respect~to~$(\cdot,\cdot)_Y$~by~${R_n\colon H_n\hspace*{-0.1em}\to \hspace*{-0.1em} H_n^*}$. As $j_n\colon V_n\to H_n$ is an isomorphism, the triple $(V_n,H_n,j_n)$ is an evolution triple with canonical~embedding $e_n\coloneqq j_n^*R_nj_n\colon V_n\to V_n^*$, which satisfies
\begin{align}
\langle e_nv_n,w_n\rangle_{V_n}=(jv_n,jw_n)_Y\quad\text{ for all }v_n,w_n\in V_n\,.\label{eq:iden}
\end{align}

Putting all together, leads us to the following algorithm:

\begin{alg}[Non-conforming Rothe--Galerkin scheme]
{\upshape
Let Assumptions \ref{asum} be satisfied. For given
$K,n\in \mathbb{N}$, the sequence of iterates
$(v_n^k)_{k=1,\dots,K}\subseteq V_n$ is given solving the
implicit Rothe--Galerkin scheme for $\tau=\frac{T}{K}$ and
$k=1,\dots,K$
\begin{align}
(d_\tau jv_n^k,jw_n)_Y+\langle [A_n]^\tau_k v_n^k,w_n\rangle_{X_n}= \langle [f_n]_k^\tau,w_n\rangle_{X_n}\quad\text{ for all }w_n\in V_n\,.\label{eq:4.15}
\end{align}
}
\end{alg}

\begin{proposition}[Well-posedness of \eqref{eq:4.15}]\label{5.1}
Let Assumption \ref{asum} be satisfied and set
${\tau_0\coloneqq \smash{\frac{1}{4c_1}}}$. Then, for every $K,n\in \mathbb{N}$ with $\tau=\smash{\frac{T}{K}}<\tau_0$, there exist iterates~${(v_n^k)_{k=1,\dots,K}\subseteq V_n}$~solving~\eqref{eq:4.15}.
\end{proposition}

\begin{proof}
Using \eqref{eq:iden} and the identity mapping $\textup{id}_{V_n}\colon V_n
\to X$, we observe that \eqref{eq:4.15} is for all $k=1,\dots, K$ equivalent to 
\begin{align}
(\textup{id}_{V_n})^*\big ([f_n]^\tau_k\big
)+\frac{1}{\tau}e_nv_n^{k-1} \in  R\Big(\frac{1}{\tau}e_n+(\textup{id}_{V_n})^*\circ [A_n]^\tau_k \circ
\textup{id}_{V_n}\Big)\,.\label{eq:5.2}
\end{align}
We fix an arbitrary $k=1,\dots,K$. Apparently,
$\frac{1}{\tau}e_n\colon V_n\to V_n^*$ is linear and
continuous.~Using~\eqref{eq:iden}, we deduce that
$\langle
\frac{1}{\tau}e_{n}v_n,v_n\rangle_{V_n}=\frac{1}{\tau}\|j_nv_n\|_Y^2\ge
0$ for all $v_n\in V_n$, i.e.,
$\frac{1}{\tau}e_n\colon V_n\to V_n^*$ is positive definite~and, thus, monotone. Hence,
$\frac{1}{\tau}e_n\colon V_n\to V_n^*$ is pseudo-monotone. Since
the conditions
(\hyperlink{AN.1}{AN.1})--(\hyperlink{AN.4}{AN.4})~are~inhe-rited from $A_n(t)\colon X_n\to X_n^*$ to
$(\textup{id}_{V_n})^*\circ A_n(t) \circ
\textup{id}_{V_n}\colon V_n \to V_n^*$,
$(\textup{id}_{V_n})^*\circ [A_n]^\tau_k\circ
\textup{id}_{V_n}=[(\textup{id}_{V_n})^*\circ A_n \circ
\textup{id}_{V_n}]^\tau_k$, and due to (\hyperlink{AN.5}{AN.5}), Proposition \ref{4.6} (i)
guarantees that the operator
$(\textup{id}_{V_n})^*\circ [A_n]_\tau^k \circ
\textup{id}_{V_n}\colon V_n\to V_n^*$ is bounded and
pseudo-monotone. Altogether, we conclude that the sum
$\frac{1}{\tau}e_n+(\textup{id}_{V_n})^*\circ
[A_n]_\tau^k\circ \textup{id}_{V_n}\colon V_n\to V_n^*$ is
bounded and pseudo-monotone.  In addition, since
$\tau<\smash{\frac{1}{2c_1}}$, combining \eqref{eq:iden} and
Proposition~\ref{4.6}~(i.b), provides for all
$v_n\in V_n$
\begin{align*}
\Big\langle \Big(\frac{1}{\tau}e_n+(\textup{id}_{V_n})^* \circ
[A_n]^\tau_k\circ \textup{id}_{V_n}\Big)v_n,v_n\Big\rangle_{V_n}\ge
3\,c_1\,\|j_nv_n\|_Y^2+c_0\,\|v_n\|_{V_n}^p-c_2\,,
\end{align*}
i.e.,  $\frac{1}{\tau}e_n+(\textup{id}_{V_n})^*\circ
[A_n]^\tau_k\circ \textup{id}_{V_n}\colon V_n\to V_n^*$ is coercive. 
Hence, the main theorem on pseudo-monotone operators, cf. \cite[Theorem 27.A]{zei-IIB},  proves \eqref{eq:5.2}.
\end{proof}

\begin{proposition}[Stability of \eqref{eq:4.15}]\label{apriori}
Let Assumption \ref{asum} be satisfied~and~set~${\tau_0\coloneqq \frac{1}{4c_1}}$. Then, 
there exists a constant $M>0$ (not depending on $K,n\in \mathbb{N}$) 
such that the piece-wise constant interpolants 
$\overline{v}_n^\tau\in \mathcal{P}_0(\mathcal{I}_\tau,V_n)$, $K,n\in \mathbb{N}$ 
with $\tau=\frac{T}{K}\in (0,\tau_0)$, and piece-wise affine interpolants 
$\hat{v}_n^\tau\in W^{1,\infty}(I,V_n)$, $K,n\in \mathbb{N}$ with 
$\tau=\frac{T}{K}\in (0,\tau_0)$, generated by iterates $(v_n^k)_{k=1,\dots,K}\subseteq V_n$, 
$K,n\in \mathbb{N}$ with $\tau=\frac{T}{K}\in (0,\tau_0)$, 
solving \eqref{eq:4.15}, satisfy the following estimates:
\begin{align}
\|\overline{v}_n^\tau\|_{L^p(I,X_n)\cap_{\smash{j}}L^\infty(I,Y)}&\leq M\,,\label{eq:5.4}\\
\|j\hat{v}_n^\tau\|_{L^\infty(I,Y)}&\leq M\,,\label{eq:5.5}\\
\|\mathcal{A}_n\overline{v}_n^\tau\|_{(L^p(I,X_n)\cap_{\smash{j}}L^q(I,Y))^*}&\leq M\,,\label{eq:5.6}\\
\|e_n(\hat{v}_n^\tau-\overline{v}_n^\tau)\|_{L^{q'}(I,V_n^*)}
&\leq\tau\,\smash{\big(\sup_{n\in \mathbb{N}}{\|f_n\|_{L^{p'}(I,X_n^*)}}+M\big)}\,.\label{eq:5.7}
\end{align}
\end{proposition}

\begin{proof}
We employ $w_n=v_n^k\in V_n$,  $k=1,\dots,l$, for arbitrary ${l=1,\dots,K}$~in~\eqref{eq:4.15},~multiply~by~${\tau\in \left(0,\tau_0\right)}$, sum with respect to $k\!=\!1,\dots,l$, 
use \eqref{eq:4.2} and $\sup_{n\in \mathbb{N}}{\|jv^0_n\|_Y}\!\leq\!\|v_0\|_H$, 
to~obtain~for~all~${l\!=\!1,\dots,K}$
\begin{align}\begin{aligned}
\frac{1}{2}\|j v_n^l\|_Y^2
+\sum_{k=1}^l{\tau\,\langle[A_n]^k_\tau v_n^k,v_n^k\rangle_{X_n}}
\leq \frac{1}{2}\|v_0\|_H^2+
\sum_{k=1}^l{\tau\,\langle [f_n]^k_\tau ,v_n^k\rangle_{X_n}}\,.
\end{aligned}
\label{eq:5.8}
\end{align}
Using $\sum_{k=1}^l{\tau\langle[f_n]^k_\tau
,v_n^k\rangle_{X_n}}
=\langle\mathscr{J}_\tau[f_n],\overline{v}_n^\tau\chi_{\left[0,l\tau\right]}\rangle_{L^p(I,X_n)}$,
the $\varepsilon$-Young inequality  
and $\|\mathscr{J}_\tau[f_n]\|_{L^{p'}(I,X_n^*)}
\leq\|f_n\|_{L^{p'}(I,X_n^*)}$ for all $n\in\mathbb{N}$ (cf.~Proposition~\ref{4.4}~(ii)), 
for every $l=1,\dots,K$, we deduce that
\begin{align}
\begin{aligned}
\sum_{k=1}^l{\tau\,\langle[f_n]^k_\tau ,v_n^k\rangle_{X_n}}
\leq c_{\varepsilon}\,\sup_{n\in \mathbb{N}}{\|f_n\|_{L^{p'}(I,X_n^*)}^{p'}}
+\varepsilon\, \int _0^{l\tau}{\|\overline{v}^\tau_n(s)\|_{X_n}^p\,\mathrm{d}s}\,.
\end{aligned}\label{eq:5.9}
\end{align}
In addition, using Proposition \ref{4.6} (i.b), for every $l=1,\dots,K$, we obtain 
\begin{align}
\sum_{k=1}^l{\tau\,\langle[A_n]^k_\tau v_n^k,v_n^k\rangle_{X_n}}\ge
c_0\,\int_0^{l\tau}{\|\overline{v}^\tau_n(s)\|_{X_n}^p\,\mathrm{d}s}
-\tau\, c_1\,\|jv_n^l\|_Y^2-\sum_{k=1}^{l-1}{\tau\, c_1\,\|jv_n^k\|_Y^2}
-c_2\,T\,.\label{eq:5.10}
\end{align}
We set $\varepsilon\hspace{-0.12em}\coloneqq \hspace{-0.12em} \frac{c_0}{2}$, 
$\alpha\hspace{-0.12em}\coloneqq \hspace{-0.12em}\frac{1}{2}\|v_0\|_H^2+c_{\varepsilon}\sup_{n\in \mathbb{N}}{\|f_n\|_{L^{p'}(I,X_n^*)}^{p'}}+c_2T$, 
$\beta\hspace{-0.12em}\coloneqq \hspace{-0.12em}4\tau c_1$ and $y^k_n\coloneqq \frac{1}{4}\|jv^k_n\|_Y^2$~for~all~${k\hspace{-0.12em}=\hspace{-0.12em}1,\dots,K}$. 
Then,  for every $l=1,\dots,K$, from \eqref{eq:5.8}--\eqref{eq:5.10} we infer that
\begin{align}
y^l_n+\frac{c_0}{2}\int_0^{l\tau}{\|\overline{v}^\tau_n(s)\|_{X_n}^p\,\mathrm{d}s}
\leq \alpha+\beta\,\sum_{k=1}^{l-1}{y^k_n}\,.\label{eq:5.11}
\end{align}
The discrete Gronwall inequality, cf. \cite[Lemma 2.2]{Ba15}, applied to \eqref{eq:5.11} yields
\begin{align}
\frac{1}{4}\|j\overline{v}_n^\tau\|_{L^\infty(I,Y)}^2
+\frac{c_0}{2}\|\overline{v}_n^\tau\|_{L^p(I,X_n)}^p\leq \alpha\exp(K\beta)
=\alpha\,\exp(4\,T\,c_1)=\vcentcolon C_0\,,\label{eq:5.12}
\end{align}
i.e., \eqref{eq:5.4}. From \eqref{eq:5.4} and  (\hyperlink{AN.4}{AN.4}), 
we obtain a constant ${C_1>0}$, such that for every ${n\in \mathbb{N}}$~and~${\tau\in \left(0,\tau_0\right)}$, it holds $$\|\mathcal{A}_n\overline{v}_n^\tau\|_{(L^p(I,X_n)\cap_{\smash{j}}L^q(I,Y))^*}\hspace{-0.1em}\leq\hspace{-0.1em} C_1\,,$$
i.e.,  \eqref{eq:5.6}.  In addition, from~\eqref{eq:5.12}, for every $n\in \mathbb{N}$ and $\tau\in \left(0,\tau_0\right)$, we~deduce~that 
$$\|j\hat{v}_n^\tau\|_{L^\infty(I,Y)}^2\leq \|j\overline{v}_n^\tau\|_{L^\infty(I,Y)}^2\leq 4 C_0\,,$$ 
i.e.,  \eqref{eq:5.5}.
Since $\smash{e_n
(\hat{v}_n^\tau(t)-\overline{v}_n^\tau(t))
=(t-k\tau) d_\tau e_n \hat{v}_n^\tau(t)=(t-k\tau)\frac{d_{e_n}\hat{v}_n^\tau}{dt}(t)}$~in~$V_n^*$ and $\vert t-k\tau\vert \leq \tau$ for all $t\in I_k^\tau$, ${k=1,\dots,K}$, 
$K,n\in\mathbb{N}$, for every $K,n\in\mathbb{N}$ with $\tau\in (0,\tau_0)$, it holds
\begin{align*}
\begin{aligned}
\big\|e_n\big(\hat{v}_n^\tau-\overline{v}_n^\tau\big)\big\|_{L^{q'}(I,V_n^*)}&
\leq\tau\left\|\frac{d_{e_n}\hat{v}_n^\tau}{dt}\right\|_{L^{q'}(I,V_n^*)}
\\&= \tau\big\|(\textup{id}_{L^q(I,V_n)})^*\big(\mathscr{J}_\tau[f_n]
-\mathscr{J}_\tau[\mathcal{A}_n]\overline{v}_n^\tau\big)\big\|_{L^{q'}(I,V_n^*)}
\\&\leq \tau\,\big(\sup_{n\in\mathbb{N}}{\|f\|_{L^{p'}(I,X_n^*)}}+C_1\big)\,,
\end{aligned}
\end{align*}
i.e.,\eqref{eq:5.7}, where we used Proposition \ref{4.4} 
and Proposition \ref{4.6} (iii.c).
\end{proof}
\enlargethispage{3mm}

We can now prove the abstract convergence result, which is the main result of this paper.

\begin{theorem}[Weak convergence of \eqref{eq:4.15}]\label{5.17}
Let Assumption \ref{asum} be satisfied~and set ${\tau_0\coloneqq \frac{1}{4c_1}}$.
If $\overline{v}_n\coloneqq \overline{v}^{\tau_n}_{m_n}\in L^\infty(I,V_{m_n})$, $n\in \mathbb{N}$, 
where $\smash{\tau_n=\frac{T}{K_n}}$ and $K_n,m_n\to\infty$  $(n\to\infty)$, 
is a diagonal sequence of the  piece-wise constant interpolants 
$\overline{v}_n^\tau\in \mathcal{P}_0(\mathcal{I}_\tau,V_n)$, ${K,n\in \mathbb{N}}$~with~$\smash{\tau=\frac{T}{K}\in (0,\tau_0)}$, from Proposition~\ref{apriori}, then there 
exists a not relabeled subsequence and an element 
${\overline{v}\in L^p(I,V)\cap_{\smash{j}} L^\infty(I,H)}$ such that
\begin{align*}
\sup_{n\in\mathbb{N}}{\|\overline{v}_n\|_{L^p(I,X_{m_n})}}<\infty\,,\qquad
j\overline{v}_n\;\;\overset{\ast}{\rightharpoondown}\;\;\;j\overline{v}
\quad\text{ in }L^\infty(I,Y)\quad(n\to\infty)\,.
\end{align*}
Furthermore, it follows that $\overline{v}\hspace*{-0.1em}\in\hspace*{-0.1em} {W}^{1,p,q}_e(I,V,H)$
is a weak solution of the initial value problem~\eqref{eq:1.1}.
\end{theorem}

\pagebreak
\begin{proof} We align the proof into four steps:

\textit{1. Convergences.} From \eqref{eq:5.4}--\eqref{eq:5.7}, (\hyperlink{NC.2}{NC.2}) 
and the reflexivity of  ${(L^p(I,V)\cap_{\smash{j}} L^q(I,H))^*}$, we obtain not relabeled
subsequences $\overline{v}_n,\hat{v}_n\in 
L^p(I,V_{m_n})\cap_{\smash{j}} L^\infty(I,Y)$, $n\in \mathbb{N}$, where $\hat{v}_n\coloneqq \hat{v}_{m_n}^{\tau_n}$ 
for all $n\in \mathbb{N}$, as well as elements 
$\overline{v}\in L^p(I,V)\cap_{\smash{j}} L^\infty(I,H)$, $\hat{v}\in  L^\infty(I,Y)$ and
${\chi}\in (L^p(I,V)\cap_{\smash{j}} L^q(I,H))^*$ such that
\begin{align}
\begin{alignedat}{3}
j\overline{v}_n&\;\;\overset{\ast}{\rightharpoondown}\;\;j\overline{v}&&\quad\text{
	in }L^\infty(I,Y)&&\quad (n\to\infty)\,,\\[-1mm]
j\hat{v}_n&\;\;\overset{\ast}{\rightharpoondown}\;\;\hat{v}&&\quad\text{
	in }L^\infty(I,Y)&&\quad (n\to\infty)\,,\\
(\mathrm{id}_{L^p(I,V)\cap_{\smash{j}} L^q(I,H)})^*\mathcal{A}_n\overline{v}_n&
\;\;\weakto\;\;{\chi}&&\quad\text{
	in }(L^p(I,V)\cap_{\smash{j}} L^q(I,H))^*&&\quad (n\to\infty)\,.
\end{alignedat}\label{eq:5.18}
\end{align}
In particular, cf. \eqref{eq:dual}, there exist $g\in L^{p'}(I,V^*)$ and $h\in L^{q'}(I,H^*)$, 
such that for every $w\in L^p(I,V)\cap_{\smash{j}} L^q(I,H)$, it~holds 
\begin{align}
\langle{\chi},{w}\rangle_{L^p(I,V)\cap_{\smash{j}} L^q(I,H)}
=\int_I{\langle g(t)+(j^*h)(t),w(t)\rangle_V\,\mathrm{d}t}\,.\label{eq:rep}
\end{align}
Due to $\eqref{eq:5.7}$, there exists a subset $E\subseteq I$, with 
$I\setminus E$ a null set, such that for every $t\in E$, it holds
\begin{align}
\big	\|e_{m_n}\big(\hat{v}_n(t)-\overline{v}_n(t)\big)\big\|_{V_{m_n}^*}\;\;\to\;\; 0\quad(n\to \infty)\,.\label{eq:5.19}
\end{align}
On the basis of (\hyperlink{NC.1}{NC.1}), we can choose for every element $w\in D$, 
a sequence $w_n\in V_{m_n}$, $n\in\mathbb{N}$, such that 
$\|w_n- w\|_{X_{m_n}}\to 0$ $(n\to\infty)$. Then, using \eqref{eq:iden}, 
\eqref{eq:5.4}, \eqref{eq:5.5}, and \eqref{eq:5.19},~for~every~${t\in E}$,~we~infer~that
\begin{align}
\begin{aligned}
\big\vert\big(P_H\big((j\hat{v}_n)(t)-(j\overline{v}_n)(t)\big),jw\big)_H\big\vert
&=\big\vert\big((j\hat{v}_n)(t)-(j\overline{v}_n)(t),jw\big)_Y\big\vert
\\
&\leq \big\vert\big\langle
e_{m_n}\big((j\hat{v}_n)(t)-(j\overline{v}_n)(t)\big),w_n\big\rangle_{V_{m_n}}\big\vert
\\
&\quad+ \big\vert\big((j\hat{v}_n)(t)-(j\overline{v}_n)(t),jw-jw_n\big)_Y\big\vert
\\
&\leq
\big\|e_{m_n}\big(\hat{v}_n(t)-\overline{v}_n(t)\big)\big\|_{V_{m_n}^*}\|w_n\|_{X_{m_n}}
\\
&\quad+ \big\|(j\hat{v}_n)(t)-(j\overline{v}_n)(t)\big\|_Y\|jw-jw_n\|_Y
\\
&\leq
\big\|e_{m_n}\big(\hat{v}_n(t)-\overline{v}_n(t)\big)\big\|_{V_{m_n}^*}\|w_n\|_{X_{m_n}}
\\
&\quad+ 2M\|jw-jw_n\|_{X_{m_n}}\;\;\to\;\; 0\quad(n\to \infty)\,.
\end{aligned}\label{eq:5.20}
\end{align}
Since $D$ is dense in $V$ and $R(j)$ is dense in $H$, for every $t\in E$, from \eqref{eq:5.20} we conclude that
\begin{align}
P_H\big((j\hat{v}_n)(t)-(j\overline{v}_n)(t)\big )\;\;\weakto\;\; 0\quad\text{ in }H\quad(n\to \infty)\,.\label{eq:5.21}
\end{align}
Since $(P_Hj\overline{v}_n)_{n\in\mathbb{N}},(P_Hj\hat{v})_{n\in\mathbb{N}}\subseteq L^\infty(I,H)$ are bounded (cf. \eqref{eq:5.4} and \eqref{eq:5.5}), \cite[Proposition 2.15]{alex-rose-hirano}~yields, owing to \eqref{eq:5.21}, that $P_H(j\overline{v}_n-j\hat{v}_n)\weakto {0}$ in $L^q(I,H)$ $(n\to\infty)$. From $\eqref{eq:5.18}_{1,2}$ we easily deduce that $P_H(j\overline{v}_n-j\hat{v}_n)\weakto P_H(j\overline{v}-\hat{v})$ in $L^q(I,H)$ $(n\to\infty)$. As a consequence,  $P_H\hat{v}=P_Hj\overline{v}=j\overline{v}$ in $L^\infty(I,H)$, where we used that $\overline{v}\in L^p(I,V)\cap_{\smash{j}}L^\infty(I,H)$. \\[-2mm]

\textit{2. Regularity and trace of the weak limit.}
\hypertarget{3.2}{}
Let $w\in D$ and $w_n\in V_{m_n}$, $n\in\mathbb{N}$, be a sequence such that $\smash{\|w_n- w\|_{X_{m_n}}\to 0}$ $(n\to\infty)$. Testing
\eqref{eq:4.15} for $n\in \mathbb{N}$ by
$\smash{w_n\in V_{m_n}}$, multiplication by $\varphi\in C^\infty(\overline{I})$ with $\varphi(T)=0$, integration over $I$, and integration-by-parts,  for every $n\in \mathbb{N}$, yield
\begin{align}\label{eq:5.22}
\begin{aligned}
&\langle  \mathscr{J}_{\tau_n}[\mathcal{A}_{m_n}]\overline{v}_n,
w_n\varphi\rangle_{L^p(I,X_{m_n})\cap_{\smash{j}}L^q(I,Y)}
-\int_{I}{\langle\mathscr{J}_{\tau_n}[f_{m_n}](s),w_n\rangle_{X_{m_n}}\varphi(s)\,\mathrm{d}s} 
\\
&=\int_I{((j\hat{v}_n)(s),jw_n)_Y\varphi^\prime(s)\,\mathrm{d}s}
+(jv_{m_n}^0,jw_n)_Y\varphi(0)\,.
\end{aligned}
\end{align}
Using Proposition \ref{4.4}, Proposition \ref{4.6} and $\eqref{eq:5.18}_3$, we obtain
\begin{align}
&\big\vert\langle \mathscr{J}_{\tau_n}[\mathcal{A}_{m_n}]\overline{v}_n,w_n\varphi\rangle_{L^p(I,X_{m_n})\cap_{\smash{j}}L^q(I,Y)}-\langle{\chi},w\varphi\rangle_{L^p(I,V)\cap_{\smash{j}}L^q(I,H)}\big\vert\label{eq:conv1}
\\
&=\big\vert\langle \mathcal{A}_{m_n}\overline{v}_n,w_n\mathscr{J}_{\tau_n}[\varphi]\rangle_{L^p(I,X_{m_n})\cap_{\smash{j}}L^q(I,Y)}-\langle{\chi},w\varphi\rangle_{L^p(I,V)\cap_{\smash{j}}L^q(I,H)}\big\vert\notag
\\
&\le\big\vert\langle
\mathcal{A}_{m_n}\overline{v}_n,(w_n-w)\mathscr{J}_{\tau_n}[\varphi]\rangle_{L^p(I,X_{m_n})\cap_{\smash{j}}L^q(I,Y)}\big\vert\notag
\\
&\quad+\big\vert\langle \mathcal{A}_{m_n}\overline{v}_n,w\mathscr{J}_{\tau_n}[\varphi]\rangle_{L^p(I,X_{m_n})\cap_{\smash{j}}L^q(I,Y)}-\langle{\chi},w\varphi\rangle_{L^p(I,V)\cap_{\smash{j}}L^q(I,H)}\big\vert\notag
\\
&\leq\|
\mathcal{A}_{m_n}\overline{v}_n\|_{(L^p(I,X_{m_n})\cap_{\smash{j}}L^q(I,Y))^*}\big(\|w_n-w\|_{X_{m_n}}\|\mathscr{J}_{\tau_n}[\varphi]\|_{L^p(I)}+\|w_n-w\|_Y\|\mathscr{J}_{\tau_n}[\varphi]\|_{L^q(I)}\big)\notag
\\
&\quad+\big\vert	\langle(\mathrm{id}_{L^p(I,V)\cap_{\smash{j}} L^q(I,H)})^*\mathcal{A}_{m_n}\overline{v}_n,w\mathscr{J}_{\tau_n}[\varphi]\rangle_{L^p(I,V)\cap_{\smash{j}}L^q(I,H)}-\langle{\chi},w\varphi\rangle_{L^p(I,V)\cap_{\smash{j}}L^q(I,H)}\big\vert\notag
\\
&\to 0\quad(n\to \infty)\,.\notag
\end{align}
Inasmuch as $\sup_{n\in
\mathbb{N}}{\|w_n\mathscr{J}_{\tau_n}[\varphi]\|_{L^p(I,X_{m_n})\cap
L^\infty(I,Y)}}<\infty$ and
$w_n\mathscr{J}_{\tau_n}[\varphi]\overset{\ast}{\rightharpoondown}w\varphi$
in $L^\infty(I,Y)$ $(n\to \infty)$
(cf.~Proposition~\ref{4.4}), using Proposition~\ref{4.6}
(iii.a),  we infer, for $n\to \infty$, that
\begin{align}
\int_I{\langle \mathscr{J}_{\tau_n}[f_{m_n}](s),w_n\varphi(s)\rangle_{X_{m_n}}\mathrm{d}s}=
\int_I{\langle f_{m_n}(s),w_n\mathscr{J}_{\tau_n}[\varphi](s)\rangle_{X_{m_n}}\mathrm{d}s}\to\int_I{\langle f,w\varphi(s)\rangle_V\,\mathrm{d}s}\,.\label{eq:conv2}
\end{align}
By passing in \eqref{eq:5.22} for $n\to \infty$, using \eqref{eq:5.18}, \eqref{eq:conv1}, \eqref{eq:conv2}, \eqref{eq:rep}, $P_H(\hat{v})=j\overline{v}$ in $L^{\infty}(I,H)$,
$v_{m_n}^0\to v_0$ in
$Y$ $(n\to\infty)$, and the density of $D$ in $V$, for every $w\in V$ and
$\varphi\in C^\infty(\overline{I})$ with $\varphi(T)=0$, we obtain 
\begin{align}
\begin{aligned}
\int_{I}{\langle g(s)+(j^*h)(s)-f(s),w\rangle_V\varphi(s)\,\mathrm{d}s}&=\int_I{(\hat{v}(s),jw)_Y\varphi^\prime(s)\,\mathrm{d}s}
+(v_0,jw)_Y\varphi(0)
\\&=\int_I{((j\overline{v})(s),jw)_H\varphi^\prime(s)\,\mathrm{d}s}
+(v_0,jw)_H\varphi(0)\,.
\end{aligned}\label{eq:5.23} 
\end{align}
In the case
$\varphi\in C_0^\infty(I)$ in \eqref{eq:5.23}, recalling Definition \ref{2.15}, we conclude that $\overline{v}\in {W}_e^{1,p,q}(I,V,H)$ with continuous representation $j_c\overline{v}\in C^0(\overline{I},H)$ and
\begin{align}
\frac{d_e\overline{v}}{dt}=f-g-j^*h\quad\text{ in }
L^{p'}(I,V^*)+ j^*(L^{q'}(I,H^*))\,.\label{eq:5.24}
\end{align}
Therefore, we can apply the generalized integration-by-parts formula in ${W}_e^{1,p,q}(I,V,H)$ (cf.~Proposition~\ref{2.16}) in
\eqref{eq:5.22} in the case $\varphi\in C^\infty(\overline{I})$ with
$\varphi(T)=0$ and $\varphi(0)=1$, which yields for all $w\in
V$ 
\begin{align}
((j_c\overline{v})(0)-v_0,jw)_H=0\,. \label{eq:5.25}
\end{align}
since $R(j)$ is dense in $H$ and $(j_c\overline{v})(0)\in H$, we deduce from \eqref{eq:5.24} that $(j_c\overline{v})(0)=v_0$ in $H$. \\[-2mm]

\textit{3. Point-wise weak convergence.}
Next, we show~that~${P_H(j\hat{v}_n)(t)\weakto (j_c
\overline{v})(t)}$~in~$H$~${(n\to\infty)}$~for~all~${t \in \overline{I}}$, which owing to \eqref{eq:5.21}, in turn, yields
that $ P_H(j\overline{v}_n)(t)
\weakto (j_c
\overline{v})(t)$ in $H$ $(n\to\infty)$ for almost~every~${t \in \overline{I}}$. To this end, let us fix an arbitrary
$t\hspace{-0.15em}\in\hspace{-0.15em} I$. From the a priori estimate
${\sup_{n\in \mathbb{N}}{\|(j\hat{v}_n)(t)\|_{Y}}\hspace{-0.15em}\leq\hspace{-0.15em} M}$~for~all~${t\hspace{-0.15em}\in\hspace{-0.15em} \overline{I}}$ (cf.~\eqref{eq:5.5}) we
obtain a subsequence
$((j\hat{v}_n)(t))_{n\in\Lambda_t}\subseteq
Y$ with $\Lambda_t\subseteq\mathbb{N}$, initially
depending on~this~fixed~$t$,~and an element ${\hat{v}_{\Lambda_t}\in Y}$ such that
\begin{align}
(j\hat{v}_n)(t)
\;\;\weakto\;\;
\hat{v}_{\Lambda_t}\quad\text{ in }Y\quad(\Lambda_t\ni n\to
\infty)\,.\label{eq:5.27}  
\end{align}
Let $w\in D$ be arbitrary and let $w_n\in V_{m_n}$, $n\in \mathbb{N}$, be a sequence~such~that~${\|w_n- w\|_{X_{m_n}}\to 0}$~${(n\to\infty)}$. We test \eqref{eq:4.15} for any $n\in \Lambda_t$ by
$w_n\in V_{m_n}$, multiply by $\varphi\in C^\infty(\overline{I})$ with $\varphi(0)=0$ and
$\varphi(t)=1$, integrate with respect to $t$ over $\left[0,t\right]$, and integrate-by-parts, to obtain for all $n\in \Lambda_t$ 
\begin{align}
\begin{aligned}
&\langle
\mathscr{J}_{\tau_n}[\mathcal{A}_{m_n}]\overline{v}_n,w_n\varphi\chi_{\left[0,t\right]}\rangle_{L^p(I,X_{m_n})\cap_{\smash{j}}L^q(I,Y)}-\int_0^t{\langle\mathscr{J}_{\tau_n}[f_{m_n}](s),w_n\rangle_{X_{m_n}}\varphi(s)\,\mathrm{d}s}\label{eq:5.28}
\\&=\int_0^t{((j\hat{v}_n)(s),jw_n)_Y\varphi^\prime(s)\,\mathrm{d}s}
-((j\hat{v}_n)(t),jw_n)_Y\,.
\end{aligned}
\end{align}
By passing in \eqref{eq:5.28} for  $n\in \Lambda_t$ to
infinity, using the same argumentation as for \eqref{eq:conv1}
and \eqref{eq:conv2} as~well~as \eqref{eq:5.18},
\eqref{eq:5.27}, $P_H(\hat{v})=j\overline{v}$ in $L^{\infty}(I,H)$, and the density of $D$ in $V$,  for every $w\in V$, we obtain
\begin{align}
\begin{aligned}	\int_0^t{\langle g(s)+(j^*h)(s)-f(s),w\rangle_V\varphi(s)\,\mathrm{d}s}
=\int_0^t{((j\overline{v})(s),jw)_H\varphi^\prime(s)\,\mathrm{d}s}
-(\hat{v}_{\Lambda_t},jw)_Y\,.
\end{aligned} \label{eq:5.29b}
\end{align}
From \eqref{eq:5.24},
\eqref{eq:5.29b} and the integration-by-parts formula in ${W}_e^{1,p,q}(I,V,H)$ (cf. Proposition \ref{2.16}) we also
obtain
\begin{align}
((j_c\overline{v})(t)-
P_H\hat{v}_{\Lambda_t},jw)_H=((j_c\overline{v})(t)-
\hat{v}_{\Lambda_t},jw)_Y=0\label{eq:5.29} 
\end{align}
for every $w\in V$. Thanks to the density of $R(j)$ in $H$, \eqref{eq:5.29} yields
$(j_c\overline{v})(t)=P_H\hat{v}_{\Lambda_t}$
in $H$, i.e.,
\begin{align}
P_H(j\hat{v}_n)(t)\;\;\weakto\;\;(j_c\overline{v})(t)\quad\text{
in }H\quad(\Lambda_t\ni n \to\infty)\,.\label{eq:5.30} 
\end{align}
As this argumentation stays valid for each subsequence of
$(P_H(j\hat{v}_n)(t) )_{n\in\mathbb{N}}\subseteq H$,
${(j_c\overline{v})(t)\in H}$ is a weak accumulation point of
each subsequence of
$(P_H(j\hat{v}_n)(t) )_{n\in\mathbb{N}}\subseteq
H$. The standard convergence principle (cf.~\cite[Kap. I, Lem. 5.4]{GGZ}) yields that
$\Lambda_t=\mathbb{N}$ in \eqref{eq:5.30}. Therefore, using \eqref{eq:5.21} and that $(j_c\overline{v})(t)=(j\overline{v})(t)$ in $H$ for almost every $t\in I$, for almost every $t\in I$, we obtain
\begin{align}\label{eq:5.31}
P_H(j\overline{v}_n)(t)\;\;\weakto\;\;(j\overline{v})(t)\quad\text{
in }H\quad(n\to \infty)\,.
\end{align}

\textit{4. Identification of $\mathcal{A}\overline{v}$ and ${\chi}$.}
Inequality \eqref{eq:5.8} in the case $\tau=\tau_n$, $n=m_n$ and $l=K_n$, resorting to Proposition \ref{4.6} (iii.a),  $(j_c\overline{v})(0)=v_0$ in $H$, 
$\|P_H(j\hat{v}_n)(T)\|_H\leq
\|(j\hat{v}_n)(T)\|_Y=\|(j\overline{v}_n)(T)\|_Y$ and
$\int_I{\langle\mathscr{J}_{\tau_n}[f_{m_n}](s),\overline{v}_n(s)\rangle_{X_{m_n}}\mathrm{d}s}=\int_I{\langle 
f_{m_n}(s),\overline{v}_n(s)\rangle_{X_{m_n}}\mathrm{d}s}$, yields for all $n\in \mathbb{N}$
\begin{align}
\begin{aligned}
\langle
\mathcal{A}_{m_n}\overline{v}_n,\overline{v}_n\rangle_{L^p(I,X_{m_n})\cap_{\smash{j}}L^q(I,Y)}&\leq 
-\frac{1}{2}\|P_H(j\hat{v}_n)(T)\|_H^2+\frac{1}{2}\|(j_c\overline{v})(0)\|_H^2\\&\quad+\int_I{\langle 
	f_{m_n}(s),\overline{v}_n(s)\rangle_{X_{m_n}}\mathrm{d}s}\,.
\end{aligned}\label{eq:5.32}
\end{align}
Thus, the limit superior with respect to $n\to\infty$ on both sides in \eqref{eq:5.32},
\eqref{eq:5.18}, \eqref{eq:5.30} with
$\Lambda_t=\mathbb{N}$ in the case $t=T$, the weak lower
semi-continuity of $\|\cdot\|_H$, the integration-by-parts
formula in ${W}_e^{1,p,q}(I,V,H)$ (cf. Proposition \ref{2.16}), and \eqref{eq:5.24} yield that
\begin{align}\label{eq:5.33}
\begin{aligned}
&\limsup_{n\to\infty}{ \langle
	\mathcal{A}_{m_n}\overline{v}_n,\overline{v}_n-\overline{v}\rangle_{L^p(I,X_{m_n})\cap_{\smash{j}}L^q(I,Y)}}
\\
&\leq
-\frac{1}{2}\|(j_c\overline{v})(T)\|_H^2+\frac{1}{2}\|(j_c\overline{v})(0)\|_H^2+
\int_I{\langle
	f(s)-g(s)-(j^*h)(s),\overline{v}(s)\rangle_V\,\mathrm{d}s}
\\&=-\int_I{\left\langle
	\frac{d_e\overline{v}}{dt}(s)+f(s)-g(s)-(j^*h)(s),\overline{v}(s)\right\rangle_V\,\mathrm{d}s}=0\,.
\end{aligned}
\end{align}
As a result of \eqref{eq:5.24}, \eqref{eq:5.31},
\eqref{eq:5.33} and the non-conforming Bochner pseudo-monotonicity of the sequence
$\mathcal{A}_n\colon L^p(I,X_n)\cap_{\smash{j}}L^q(I,Y)\!\to\!
({L^p(I,X_n)\cap_{\smash{j}}L^q(I,Y)})^*$, $n\!\in\! \mathbb{N}$, (cf.~Proposition~\ref{3.9}) with respect~to~$(V_n)_{n\in \mathbb{N}}$ and $\mathcal{A}\colon L^p(I,V)\cap_{\smash{j}}L^q(I,H)\to
({L^p(I,V)\cap_{\smash{j}}L^q(I,H)})^*$,  for every $w\in {L^p(I,V)\cap_{\smash{j}}L^q(I,H)}$,~it~holds
\begin{align*}
\langle \mathcal{A}\overline{v},\overline{v}-w\rangle_{L^p(I,V)\cap_{\smash{j}}L^q(I,H)}&\leq\liminf_{n\to\infty}{\langle \mathcal{A}_{m_n}\overline{v}_n,\overline{v}_n-w\rangle_{{L^p(I,X_{m_n})\cap_{\smash{j}}L^q(I,Y)}}}\\&\leq\langle{\chi},\overline{v}-w\rangle_{{L^p(I,V)\cap_{\smash{j}}L^q(I,H)}}\,,
\end{align*}
which implies that $\mathcal{A}\overline{v}={\chi}$ in $({L^p(I,V)\cap_{\smash{j}}L^q(I,H)})^*$.~
\end{proof}

\section{Application}\label{sec:7}
We now apply the abstract theory developed in Section
\ref{sec:3} to our motivating example, i.e., $p$-Navier--Stokes
equations \eqref{eq:p-NS2}.  To do so, we introduce some more
notation. We employ the same one~as~in~\cite{kr-pnse-ldg-1},~\cite{kr-pnse-ldg-2},~and~\cite{kr-pnse-ldg-3}.
\subsection{Function spaces}

We write $f\sim g$ if and only if there~exist~constants $c,C>0$ such~that~${c\, f \le g\le C\, f}$.

For $\ell\in \setN$ and $p\in [1,\infty]$, we employ the customary
Lebesgue spaces $(L^p(\Omega), \smash{\norm{\cdot}_p}) $ and Sobolev
spaces $(W^{\ell,p}(\Omega), \smash{\norm{\cdot}_{k,p}})$, where $\Omega
\subseteq \setR^d$, $d\in \{2,3\}$, is a bounded,~polyhedral Lipschitz domain. The space $\smash{W^{1,p}_0(\Omega)}$
is defined as those functions from $W^{1,p}(\Omega)$ whose trace vanishes on $\partial\Omega$. We equip
$\smash{W^{1,p}_0(\Omega)}$~with~the~norm $\smash{\norm{\nabla\,\cdot\,}_p}$. 

With a few exceptions, we do not distinguish between function spaces for scalar,
vector-~or~{tensor-valued} functions. However, we denote
vector-valued functions by boldface letters~and~tensor-valued
functions by capital boldface letters. The standard scalar product
between~two vectors is denoted by $\bfa \cdot\bfb$, while the
Frobenius scalar product between~two~tensors is denoted by
$\bfA: \bfB$.  The mean value of a locally integrable function $f$
over a measurable set $M\subseteq \Omega$ is denoted by
${\mean{f}_M\coloneqq \smash{\dashint_M f
\,\textup{d}x}\coloneqq \smash{\frac 1 {|M|}\int_M f
\,\textup{d}x}}$. Furthermore, we employ the notation
$\hskp{f}{g}\coloneqq \int_\Omega f g\,\textup{d}x$, whenever the
right-hand~side~is~\mbox{well-defined}.

We will also use Orlicz and Sobolev--Orlicz spaces
(cf.~\cite{ren-rao}).~A~real~convex~\mbox{function}
$\psi \colon \mathbb{R}^{\geq 0} \to \mathbb{R}^{\geq 0}$ is said to be
an \textit{N-function}, if it holds $\psi(0)=0$,
$\psi(t)>0$~for~all~${t>0}$, $\lim_{t\rightarrow0} \psi(t)/t=0$, and
$\lim_{t\rightarrow\infty} \psi(t)/t=\infty$.
We~define the \textit{(convex) conjugate N-function} $\psi^*$ by
${\psi^*(t)\coloneqq \sup_{s \geq
0} (st -\psi(s))}$~for all
${t \geq 0}$, which satisfies
$(\psi^*)' =  (\psi')^{-1}$. A given N-function
$\psi$~satisfies~the~\textit{\mbox{$\Delta_2$-condition}} (in short,
$\psi \in\Delta_2$), if there exists
$K> 2$ such that ${\psi(2\,t) \leq K\,
\psi(t)}$ for all
$t \ge
0$. We denote the smallest such constant by
$\Delta_2(\psi)>0$. We need the
following version of the \textit{$\varepsilon$-Young inequality}: for every
${\varepsilon> 0}$,~there~exists a constant $c_\epsilon>0 $,
depending  on $\Delta_2(\psi),\Delta_2( \psi ^*)<\infty$, such
that~for~every~${s,t\geq 0}$,~it~holds
\begin{align}
\label{ineq:young}
t\,s&\leq \epsilon \, \psi(t)+ c_\epsilon \,\psi^*(s)\,.
\end{align}

For $p\in (1,\infty)$,
we frequently use the following function spaces
\begin{align*}
\begin{aligned}
\Vo&\coloneqq W^{1,p}_0(\Omega)^d\,,\quad&
\Vo(0)&\coloneqq \{\bfz\in \Vo\mid \divo \bfz=0\}\,,
\\
\Ho(0)&\coloneqq \overline {\mathaccent23 V(0)}^{\norm{\cdot}_2}\,,\quad&
\Qo&\coloneqq L_0^{p'}(\Omega)\coloneqq \big\{f\in
L^{p'}(\Omega)\;|\;\mean{f}_\Omega=0\big\}\,.
\end{aligned}
\end{align*}

\subsection{Basic properties of the extra stress tensor}

We assume that the extra stress tensor $\SSS$ has $(p,\delta)$-structure, which will be~defined~now. A detailed discussion and thorough proofs can be found in \cite{die-ett,dr-nafsa}. For a tensor $\bfA\in \mathbb{R}^{d\times d}$, we denote its symmetric part by ${\bfA^{\textup{sym}}\coloneqq \frac{1}{2}(\bfA+\bfA^\top)\in \mathbb{R}^{d\times d}_{\textup{sym}}\coloneqq \{\bfA\in \mathbb{R}^{d\times d}\mid \bfA=\bfA^\top\}}$.

For $p \in (1,\infty)$ and~$\delta\ge 0$, we define a special N-function
$\phi=\phi_{p,\delta}\colon \smash{\mathbb{R}^{\ge 0}\to \mathbb{R}^{\ge 0}}$~via
\begin{align} 
\label{eq:5} 
\varphi(t)\coloneqq  \int _0^t \varphi'(s)\, \mathrm ds\,,\quad\text{where}\quad
\varphi'(t) \coloneqq  (\delta +t)^{p-2} t\,,\quad\textup{ for all }t\ge 0\,.
\end{align}
It is well-known that $\phi$ is balanced
(cf.~\cite{dr-nafsa,br-multiple-approx}), since $ {\min\set{1,p-1}\,( \delta+t)^{p-2} \le \varphi''(t)}$ $\leq
\max\set{1,p-1}$ $( \delta+t)^{p-2}$ is valid for~all~${t,\delta\ge 0}$. Moreover, $\phi$ satisfies, independent~of~${\delta\ge  0}$, the
$\Delta_2$-condition~with ${\Delta_2(\phi) \leq c\, 2^{\max \set{2,p}}}$. 
The conjugate function $\phi^*$ satisfies, independent of ${\delta\ge  0}$, the
$\Delta_2$-condition with $\Delta_2(\phi^*) \leq c\,2^{\max \set{2,p'}}$~and ${\phi^*(t) \sim
(\delta^{p-1} + t)^{p'-2} t^2}$ uniformly in $t\ge 0$.

An important tool in our analysis are shifted N-functions
$\set{\psi_a}_{\smash{a \ge 0}}$  (cf.~\cite{DK08,dr-nafsa}). For a given N-function $\psi\colon \mathbb{R}^{\ge 0}\to \mathbb{R}^{\ge
0}$, we define the family  of shifted N-functions ${\psi_a\colon \mathbb{R}^{\ge
0}\to \mathbb{R}^{\ge 0}}$,~${a \ge 0}$, for every $a\ge 0$, via
\begin{align}
\label{eq:phi_shifted}
\psi_a(t)\coloneqq  \int _0^t \psi_a'(s)\, \mathrm ds\,,\quad\text{where }\quad
\psi'_a(t)\coloneqq \psi'(a+t)\frac {t}{a+t}\,,\quad\textup{ for all }t\ge 0\,.
\end{align}

\begin{remark} \label{rem:phi_a}
{\rm 
For the above defined N-function $\varphi $, uniformly in $a,t\ge 0$, it holds
${\phi_a(t) \sim (\delta+a+t)^{p-2} t^2}$ and $(\phi_a)^*(t)
\sim ((\delta+a)^{p-1} + t)^{p'-2} t^2$.  The families
$\set{\phi_a}_{\smash{a \ge 0}}$~and $\set{(\phi_a)^*}_{\smash{a \ge 0}}$ satisfy, uniformly in $a \ge 0$,
the $\Delta_2$-condition~with~$\Delta_2(\phi_a) \leq c\, 2^{\smash{\max \set{2,p}}}$ and
$\Delta_2((\phi_a)^*) \leq c\, 2^{\smash{\max \set{2,p'}}}$,
respectively. Moreover, note that $(\phi_{p,\delta})_a(t)=\phi_{p,
\delta+a}(t)$ for all $t,a,\delta\ge 0$, and that the function
$(a \mapsto \phi_a(t))\colon \mathbb{R}^{\ge 0}\to\mathbb{R}^{\ge 0} $,~for~every~${t\ge 0}$, is~non-decreasing for $p\ge 2$ and non-increasing for $p\le 2$.
}
\end{remark}

\begin{asum}[Extra stress tensor]\label{assum:extra_stress}
We assume that the extra stress tensor $ \SSS\colon \mathbb{R}^{d \times d}
\to \mathbb{R}^{d \times d}_{\textup{sym}} $ belongs to $C^0(\mathbb{R}^{d \times
d};\mathbb{R}^{d \times d}_{\textup{sym}} ) $, satisfies $\SSS (\bfA) = \SSS 
(\bfA^{\textup{sym}})$ for all $ \bfA\in \mathbb{R}^{d \times d}$, and $\SSS (\mathbf 0)=\mathbf 0$.~Moreover,~we~assume~that the tensor $\SSS$ has {\rm $(p, \delta)$-structure}, i.e.,
for some $p \hspace{-0.1em}\in \hspace{-0.1em}(1, \infty)$, $ \delta\hspace{-0.1em}\in\hspace{-0.1em} [0,\infty)$, and the
N-function~${\varphi\hspace{-0.1em}=\hspace{-0.1em}\varphi_{p,\delta}}$~(cf.~\eqref{eq:5}), there
exist constants $C_0, C_1 >0$ such that
\begin{subequations}
\label{eq:ass_S}
\begin{align}
({\SSS}(\bfA) - {\SSS}(\bfB)) : (\bfA-\bfB
) &\ge C_0 \,\phi_{\abs{\bfA^{\textup{sym}}}}(\abs{\bfA^{\textup{sym}} -
	\bfB^{\textup{sym}}}) \,,\label{1.4b}
\\
\abs{\SSS(\bfA) - \SSS(\bfB)} &\le C_1 \,
\phi'_{\abs{\bfA^{\textup{sym}}}}(\abs{\bfA^{\textup{sym}} -
	\bfB^{\textup{sym}}})\label{1.5b}
\end{align}
\end{subequations}
are satisfied for all $\bfA,\bfB \hspace{-0.1em}\in\hspace{-0.1em} \mathbb{R}^{d \times d} $.  The
constants $C_0,C_1\hspace{-0.1em}>\hspace{-0.1em}0$ and $p\hspace{-0.1em}\in\hspace{-0.1em} (1,\infty)$ are called the 
characteristics~of~$\SSS$.
\end{asum}

\begin{remark}
{\rm Let $\phi$ be defined in \eqref{eq:5} and 
$\{\phi_a\}_{a\ge 0}$ the corresponding family of the~shifted~\mbox{N-functions}. Then, the operators 
$\SSS_a\colon\mathbb{R}^{d\times d}\to \smash{\mathbb{R}_{\textup{sym}}^{d\times
	d}}$, $a \ge 0$, for every $a \ge 0$
and~$\bfA \in \mathbb{R}^{d\times d}$, defined via
\begin{align}
\label{eq:flux}
\SSS_a(\bfA) \coloneqq 
\frac{\phi_a'(\abs{\bfA^{\textup{sym}}})}{\abs{\bfA^{\textup{sym}}}}\,
\bfA^{\textup{sym}}\,, 
\end{align}
have $(p, \delta +a)$-structure (cf.~Remark~\ref{rem:phi_a}).  In this case, the characteristics of
$\SSS_a$ depend only on $p\in (1,\infty)$ and are independent of
$\delta \geq 0$ and $a\ge 0$.
}
\end{remark}

\pagebreak
\subsection{DG spaces, jumps and averages}\label{sec:dg-space}

\subsubsection{Triangulations}\label{sec:triang}

Throughout the remaining article, we will always denote by $\mathcal{T}_h$, $h>0$, a family of regular, i.e., uniformly shape regular and conforming, triangulations of $\Omega\subseteq  \mathbb{R}^d$, ${d\in \set{2,3}}$, cf.~\cite{ciarlet,Bs08}, each consisting of \mbox{$d$-dimensional} simplices~$K$. Here,
$h>0$ refers to the maximal mesh-size~of~$\mathcal{T}_h$,~i.e., if we~define~${h_K\coloneqq \textup{diam}(K)}$~for~all~${K\in \mathcal{T}_h}$, then~${h\coloneqq \max_{K\in \mathcal{T}_h}{h_K}}$.
For simplicity, we assume~that~${h \le 1}$. For~every~simplex~${K \in \mathcal{T}_h}$,~we~denote~by $\rho_K>0$, the supremum of diameters~of~all~inscribed~balls.  We assume that there is a constant $\omega_0>0$, which is independent of $h> 0$, such that ${h_K}{\rho_K^{-1}}\leq 
\omega_0$~for~all~${K \in \mathcal{T}_h}$. The smallest such constant~is~called~the~chunkiness~of~$(\mathcal{T}_h)_{h>0}$. 
By $\Gamma_h^{i}$, we denote the interior faces~and put $\Gamma_h\coloneqq  \Gamma_h^{i}\cup \partial\Omega$.
We assume that each simplex ${K \in \mathcal{T}_h}$ has at most one face from the boundary~$\partial\Omega$.  We introduce the following scalar products~on~$\Gamma_h$:
\begin{align*}
\skp{f}{g}_{\Gamma_h} \coloneqq  \smash{\sum_{\gamma \in \Gamma_h} {\langle f, g\rangle_\gamma}}\,,\quad\text{ where }\quad\langle f, g\rangle_\gamma\coloneqq \int_\gamma f g \,\mathrm{d}s\quad\text{ for all }\gamma\in \Gamma_h\,,
\end{align*}
whenever all integrals are well-defined. Analogously, we define the products 
$\skp{\cdot}{\cdot}_{\partial\Omega}$ and~$\skp{\cdot}{\cdot}_{\Gamma_h^{i}}$. 

We  need the following construction: Let $\widehat{\Omega} \supsetneq \Omega$ be a polyhedral, bounded
Lipschitz domain. Let
$\widehat{\mathcal{T}}_h$ denote an extension of the
triangulation~$\mathcal{T}_h$ to $\widehat{\Omega}$, having~the~same~properties~as~$\mathcal{T}_h$ (in particular, with a similar chunkiness).  We
extend our notation to this setting by adding a superposed `hat'~to~it.~In~particular, we denote by $\widehat{\Gamma}_h$ and $\widehat{\Gamma}_h^{i}$, 
the faces and interior faces, respectively,~of~$\widehat{\mathcal{T}}_h$. 

\subsubsection{Broken function spaces and projectors}

For every $m \in\mathbb{N}_0$ and $K\in \mathcal{T}_h$,
we denote by ${\mathcal P}_m(K)$, the space of
polynomials of degree at most $m$ on $K$. Then, for
given~$p\in (1,\infty)$ and $\ell \in \mathbb{N}_0$,~we~define~the~spaces
\begin{align}
\begin{aligned}
Q_h^\ell&\coloneqq \big\{ q_h\in L^1(\Omega)\fdg  q_h|_K\in \mathcal{P}_\ell(K)\text{ for all }K\in \mathcal{T}_h\big\}\,,\\
V_h^\ell&\coloneqq \big\{\bfz_h\in L^1(\Omega)^d\fdg \bfz_h|_K\in \mathcal{P}_\ell(K)^d\text{ for all }K\in \mathcal{T}_h\big\}\,,\\
X_h^\ell&\coloneqq \big\{\bfX_h\in L^1(\Omega)^{d\times d}\fdg \bfX_h|_K\in \mathcal{P}_\ell(K)^{d\times d}\text{ for all }K\in \mathcal{T}_h\big\}\,,\\
W^{1,p}(\mathcal T_h)&\coloneqq \big\{\bfw_h\in L^1(\Omega)^d\fdg \bfw_h|_K\in W^{1,p}(K)^d\text{ for all }K\in \mathcal{T}_h\big\}\,.
\end{aligned}\label{eq:2.19}
\end{align}
Apart from that, for given $\ell \in \mathbb{N}_0$, we define $\Qhkc\coloneqq  Q_h^\ell\cap C^0(\overline{\Omega})$ and $V_{h,c}^\ell\coloneqq  V_h^\ell\cap C^0(\overline{\Omega})$. 
We~denote~by $\PiDG\colon L^1(\Omega)^d\hspace{-0.1em}\to \hspace{-0.1em}V_h^\ell$, the \textit{(local)
$L^2$-projection}~into~$V_h^\ell$, which, for every $\bfz\hspace{-0.1em} \in\hspace{-0.1em}
\smash{L^1(\Omega)^d}$ and $\bfz_h\hspace{-0.1em}\in \hspace{-0.1em}V_h^\ell$,~is~defined~via
\begin{align}
\label{eq:PiDG}
\bighskp{\PiDG \bfz}{\bfz_h}=\hskp{\bfz}{\bfz_h}\,.
\end{align}
Analogously, we define the (local)
$L^2$-projection into $X_h^\ell$, i.e., ${\PiDG\colon L^1(\Omega)^{d\times d} \to \Xhk}$.

For  every $\bfw_h\in \WDG$, we denote by $\nabla_h \bfw_h\in L^p(\Omega)$,
the \textit{local gradient},~defined~via~$(\nabla_h \bfw_h)|_K\coloneqq \nabla(\bfw_h|_K)$ for~all~${K\in\mathcal{T}_h}$.
For every $K\in \mathcal{T}_h$, ${\bfw_h\in \WDG}$~admits~an
interior~trace~${\textrm{tr}^K(\bfw_h)\in L^p(\partial K)}$. For each face
$\gamma\in \Gamma_h$ of a simplex $K\in \mathcal{T}_h$, we denote the respective
interior trace by
$\smash{\textup{tr}^K_\gamma(\bfw_h)\in L^p(\gamma)}$.~Then, for
every $\bfw_h\in \WDG$ and interior faces $\gamma\in \Gamma_h^{i}$ shared by
adjacent elements $K^-_\gamma, K^+_\gamma\in \mathcal{T}_h$,~we~\mbox{denote}~by
\begin{align}
\{\bfw_h\}_\gamma&\coloneqq \smash{\tfrac{1}{2}}\big(\textup{tr}_\gamma^{K^+}(\bfw_h)+
\textup{tr}_\gamma^{K^-}(\bfw_h)\big)\in
L^p(\gamma)\,, \label{2.20}\\
\jump{\bfw_h\otimes\bfn}_\gamma
&\coloneqq \textup{tr}_\gamma^{K^+}(\bfw_h)\otimes\bfn^+_\gamma+
\textup{tr}_\gamma^{K^-}(\bfw_h)\otimes\bfn_\gamma^- 
\in L^p(\gamma)\,,\label{eq:2.21}
\end{align}
the \textit{average} and \textit{normal jump}, respectively, of $\bfw_h$ on $\gamma$.
Moreover,  for every $\bfw_h\in \WDG$ and boundary faces $\gamma\in \partial\Omega$, we define boundary averages and 
boundary~jumps,~respectively,~via
\begin{align}
\{\bfw_h\}_\gamma&\coloneqq \textup{tr}^\Omega_\gamma(\bfw_h) \in L^p(\gamma)\,,\label{eq:2.23a} \\
\jump{ \bfw_h\otimes\bfn}_\gamma&\coloneqq 
\textup{tr}^\Omega_\gamma(\bfw_h)\otimes\bfn \in L^p(\gamma)\,,\label{eq:2.23} 
\end{align}
where $\bfn\colon\partial\Omega\to \mathbb{S}^{d-1}$ denotes the unit normal vector field to $\Omega$ pointing outward. 
Analogously, we
define $\{\bfX_h\}_\gamma$ and $ \jump{\bfX_h\bfn}_\gamma
$~for all $\bfX_h \in \Xhk$ and $\gamma\in \Gamma_h$. If there is no
danger~of~confusion, we omit~the~index~${\gamma\in \Gamma_h}$, in particular,  when we interpret jumps and averages as global functions defined on the whole of  $\Gamma_h$.

\subsubsection{DG gradient and jump operators}

For every $\ell \in \mathbb{N}_0$ and every face $\gamma\in \Gamma_h$, we define the
\textit{(local)~jump \mbox{operator}}
$\boldsymbol{\mathcal{R}}_{h,\gamma}^\ell \colon\WDG \to X_h^\ell $, using
Riesz representation, for every $ \bfw_h\in \smash{\WDG}$ via
\begin{align}
\big(\boldsymbol{\mathcal{R}}_{h,\gamma}^\ell\bfw_h,\bfX_h\big)
\coloneqq \big\langle \llbracket\bfw_h\otimes\bfn\rrbracket_\gamma,\{\bfX_h\}_\gamma\big\rangle_\gamma
\quad\text{ for all }\bfX_h\in X_h^\ell\,.\label{eq:2.25}
\end{align}
Then, for every $\ell \in \mathbb{N}_0$, the \textit{(global) jump operator} $\smash{\Rhk\coloneqq \sum_{\gamma\in \Gamma_h}{\boldsymbol{\mathcal{R}}_{\gamma,h}^\ell}\colon\WDG \to X_h^\ell}$, 
by definition, for every $\bfw_h\in \smash{\WDG}$ and  $\bfX_h\in X_h^\ell$ satisfies
\begin{align}
\big(\Rhk\bfw_h,\bfX_h\big)=\big\langle
\llbracket\bfw_h\otimes\bfn\rrbracket,\{\bfX_h\}\big\rangle_{\Gamma_h}\,.\label{eq:2.25.1}
\end{align}
In addition, for every $\ell \in \mathbb{N}_0$, the \textit{DG gradient operator} 
$  \Ghk\colon\WDG\to L^p(\Omega)$, for every $\bfw_h\in \smash{\WDG}$, is  defined via
\begin{align}
\boldsymbol{\mathcal G}^\ell_{h}\bfw_h\coloneqq 
\nabla_h\bfw_h-\boldsymbol{\mathcal R}^\ell_h\bfw_h
\quad\text{ in }L^p(\Omega)\,.\label{eq:DGnablaR} 
\end{align}
In particular, for every $\bfw_h\in \smash{\WDG}$ and $\bfX_h\in X_h^\ell$, we have that 
\begin{align}
\big(\Ghk\bfw_h,\bfX_h\big)=(\nabla_h\bfw_h,\bfX_h)
-\big\langle \llbracket
\bfw_h\otimes\bfn\rrbracket,\{\bfX_h\}\big\rangle_{\Gamma_h}
\,.  \label{eq:DGnablaR1}
\end{align}
Note that for every $\bfz\in \Vo$, we have that $\Ghk \bfz=\nabla\bfz
$ in $L^p(\Omega)$.
For every $\bfw_h\in \WDG$, we  introduce the \textit{DG norm} via
\begin{align}
\|\bfw_h\|_{\nabla,p,h}\coloneqq \|\nabla_h\bfw_h\|_p+h^{\frac{1}{p}}\big\|h^{-1}\jump{\bfw_h\otimes \bfn}\big\|_{p,\Gamma_h}\,,
\end{align}
which turns $\WDG$ into a Banach space\footnote{The completeness of
$\WDG$ equipped with $\|\cdot\|_{\nabla,p,h}$, for every fixed
$h>0$, follows from ${\|\bfw_h\|_p\leq c\,\|\bfw_h\|_{\nabla,p,h}}$
for all $\bfw_h\in \smash{\WDG}$ (cf.~\cite[Lemma A.9]{dkrt-ldg})
and an element-wise application of the trace
theorem.\vspace{-1cm}}. Owing to \cite[{(A.26)--(A.28)}]{dkrt-ldg},
there exists a constant $c>0$, depending on $p\in (1,\infty)$, $\ell \in \setN_0$, and the chunkiness $\omega_0>0$, such that $\bfw_h\in \smash{\WDG}$, it holds
\begin{align}\label{eq:eqiv0}
c^{-1}\,\|\bfw_h\|_{\nabla,p,h}\leq \big\|\Ghnk\bfw_h\big\|_p+h^{\frac{1}{p}}\big\|h^{-1}\jump{\bfw_h\otimes \bfn}\big\|_{p,\Gamma_h}\leq c\,\|\bfw_h\|_{\nabla,p,h}\,.
\end{align}

\subsubsection{Symmetric DG gradient and symmetric jump operators}
Following \cite[Sec. \!4.1.1]{BCPH20}, we define~a~symme-trized
version of the DG gradient. \hspace{-0.17em}For every $\bfw_h\hspace{-0.17em}\in\hspace{-0.17em}
\WDG$, we denote by
${\bfD_h\bfw_h\hspace{-0.17em}\coloneqq \hspace{-0.17em}[\nabla_h\bfw_h]^{\textup{sym}}\hspace{-0.17em}\in\hspace{-0.17em}
L^p(\Omega;\mathbb{R}^{d\times
d}_{\textup{sym}})}$, the \textit{local symmetric gradient}.
In~addition, for every $\ell \in \setN_0$, we define the
\textit{symmetric DG gradient
operator} $ \smash{\Dhk}\colon\WDG\to L^p(\Omega;\mathbb{R}^{d\times
d}_{\textup{sym}})$,~for~every~${\bfw_h\in \WDG}$
via
$\smash{\Dhk\bfw_h\coloneqq [\Ghk\bfw_h]^{\textup{sym}}}
\in L^p(\Omega;\mathbb{R}^{d\times d}_{\textup{sym}})
$, i.e., if we define the \textit{symmetric jump {operator}}
$\smash{\Rhks}\colon\WDG\to X_h^{\ell, \textup{sym}}:=X_h^{\ell}\cap L^p(\Omega;\mathbb{R}^{d\times d}_{\textup{sym}})$ via
${\smash{\Rhks}\bfw_h\coloneqq [\Rhk\bfw_h]^{\textup{sym}}\in
X_h^{\ell, \textup{sym}}}$ for every ${\bfw_h\in
\WDG}$, then for every ${\bfw_h\in
\WDG}$,~we~have~that
\begin{align}\label{eq:defD}
\smash{\Dhk\bfw_h =\bfD_h\bfw_h
-\Rhks\bfw_h
\quad\text{ in }L^p(\Omega;\mathbb{R}^{d\times d}_{\textup{sym}})\,.}
\end{align}
In particular, for every $\bfw_h\in \WDG$ and $\bfX_h\in X_h^{\ell,
  \textup{sym}}$, we have that
\begin{align}
\smash{\big(\Dhk\bfw_h,\bfX_h\big)
=(\bfD_h\bfw_h,\bfX_h)
-\big\langle \llbracket \bfw_h\otimes\bfn\rrbracket,\{\bfX_h\}\big\rangle_{\Gamma_h}\,.}\label{eq:2.24}
\end{align}
For every $\bfw_h\in \WDG$, we  introduce the \textit{symmetric DG norm}~via 
\begin{align}
\smash{\|\bfw_h\|_{\bfD,p,h}\coloneqq \|\bfD_h\bfw_h\|_p
+\smash{h^{\frac{1}{p}}\big\|  h^{-1} \llbracket\bfw_h\otimes\bfn\rrbracket\big\|_{p,\Gamma_h}}}\,.\label{eq:2.29}
\end{align}
Owing to \cite[Proposition 2.5]{kr-pnse-ldg-1}, there exists a
constant~${c>0}$, depending on $p\in (1,\infty)$, $\ell \in \setN$, and
the chunkiness $\omega_0>0$, such that for every $\bfw_h\in \WDG$, it
holds
\begin{align}
\begin{aligned}
\smash{c^{-1}\,\|\bfw_h\|_{\bfD,p,h} \leq \big\|\Dhk\bfw_h\big\|_p
+h^{\frac{1}{p}}\big\|
h^{-1}\llbracket\bfw_h\otimes\bfn\rrbracket\big\|_{p,\Gamma_h}
\leq c\,\|\bfw_h\|_{\bfD,p,h}\,.}
\end{aligned}\label{eq:equi2}
\end{align}

The following discrete Korn inequality on $V_h^\ell$ demonstrates that $\|\cdot\|_{\bfD,p,h}\sim\|\cdot\|_{\nabla,p,h}$ on $V_h^\ell$ 
and, thus, forms a cornerstone of the numerical  analysis of the $p$-Navier--Stokes~system~\eqref{eq:p-NS}.
\begin{proposition}[Discrete Korn inequality]\label{korn}
For every $p\in (1,\infty)$ and $\ell \in \setN$, there
exists a constant ${c>0}$,  depending on $p\in
(1,\infty)$, $\ell$ and the chunkiness $\omega_0>0$, such that 
for every $\bfz_h\in V_h^\ell$, it holds
\begin{align*}
\smash{\|\bfz_h\|_{\nabla,p,h}\leq
c\,\|\bfz_h\|_{\bfD,p,h}}\,. 
\end{align*}
\end{proposition}

\begin{proof}
See \cite[Proposition 2.4]{kr-pnse-ldg-1}.
\end{proof}

\subsubsection{DG divergence operator}
The \textit{local divergence}, for every $\bfw_h\in\WDG$,~is~defined~via~$\textup{div}_h\bfw_h\coloneqq \text{tr}(\nabla_h \bfw_h)\in  L^p(\Omega)$. 
In addition, for every $\ell \in\setN_0$, the \textit{DG divergence operator} ${\Divhk\colon\WDG\to L^p(\Omega)}$, for every $\bfw_h\in  \WDG$, is defined via $\smash{\Divhk\bfw_h\coloneqq \text{tr}(\Ghk\bfw_h)
=\text{tr}(\Dhk\bfw_h)\in L^p(\Omega)}$,~i.e., 
\begin{align*}    
\Divhk\bfw_h=\textup{div}_h\bfw_h-\textup{tr}(\Rhk\bfw_h)\quad\text{ in }L^p(\Omega)\,.
\end{align*}
In particular, 
for every $\bfw_h\in \WDG$ and $z_h  \in Q_h^\ell$, we have that
\begin{align*}
\big(\Divhk\bfw_h,z_h\big)&=(\textup{div}_h\bfw_h,z_h)
-\big\langle \llbracket \bfw_h\cdot\bfn\rrbracket,\{z_h\}\big\rangle_{\Gamma_h}\\
&=-(\bfw_h,\nabla_h z_h)
+\big\langle \{\bfw_h\cdot\bfn\}, \llbracket z_h\rrbracket\big\rangle_{\Gamma_h^{i}}\,,
\end{align*}
i.e., $(\Divhk \bfw_h,z_h)=-( \bfw_h,\nabla z_h)$ if $z_h  \in \Qhkc$. As a result,  for every $\bfz\in W^{1,p}_0(\Omega)$~and~$z_h\in \Qhkc$,~it~holds
\begin{align}
\label{eq:div-dg}\big(\Divhk \PiDG \bfz,z_h\big)=-( \bfz,\nabla z_h)=(\divo\bfz, z_h)\,.
\end{align}

\subsection{The $p$-Navier--Stokes equations}

We now show that different DG formulations of the $p$-Navier--Stokes
equations fit into the abstract framework of Section \ref{sec:3}. We
start proving that the approximation of divergence-free Sobolev
functions via discretely divergence-free DG finite element spaces is a
Bochner non-conforming approximation. From now on, we restrict
ourselves to the case $\ell \in \setN$ for the polynomial degree in the
broken function spaces.
\begin{proposition}\label{3.3}
For given $\ell \in\mathbb{N}$ and $h>0$, we define 
\begin{align*}
\Vhk(0)&\coloneqq \big\{\bfz_h\in \Vhk
\fdg \big(\Divhk\bfz_h,z_h\big)=0\text{ for all }z_h\in \Qhkc\big\}\,.
\end{align*}
If $p>\frac{2d}{d+2}$ and $(h_n)_{n\in \mathbb{N}}\subseteq\left(0,\infty\right)$ is a null sequence,
$( \Vhnk(0))_{n\in \mathbb{N}}$ 
is a Bochner non-conforming approximation of $\Vo(0)$ with respect to $(\WDGn)_{n\in \mathbb{N}}$.
\end{proposition}
\begin{proof}
Apparently, $(\Vo(0),\Ho(0), \mathrm{id})$ and $(\WDGn,L^2(\Omega), \mathrm{id} )$, $n\in \mathbb{N}$, form evolution triples 
such that $\Ho(0)\subseteq L^2(\Omega)$ with~$(\cdot,\cdot)_{\Ho(0)}=(\cdot,\cdot)_{L^2(\Omega)}$ in 
$\Ho(0)\times \Ho(0)$ and for every $n\in \mathbb{N}$, it holds $\Vo(0)\subseteq\WDGn$ with 
$\|\nabla \cdot\|_p=\|\cdot\|_{\nabla,p,h_n}$ in~$\Vo(0)$~and~${\|\cdot\|_2\leq c\,\|\cdot\|_{\nabla,p,h_n}}$ in $\WDGn$
for some constant $c>0$, which does not depend on $n\in \mathbb{N}$, by virtue of 
\cite[Proposition 2.6]{kr-pnse-ldg-3}.
So, let us verify  (\hyperlink{NC.1}{NC.1}) 
and (\hyperlink{NC.2}{NC.2}):

\textit{ad (NC.1).} See \cite[Lemma 5.1]{kr-pnse-ldg-1}.

\textit{ad (NC.2).} Without loss of
generality, we assume~that $m_n=n$ for all $n\in \mathbb{N}$. Let
${\bfv_n\hspace{-0.1em}\in \hspace{-0.1em} L^p(I,V_{h_n}^\ell(0))}\cap L^\infty(I,L^2(\Omega))$,
$n\hspace{-0.1em}\in\hspace{-0.1em} \mathbb{N}$, 
be such that
$\sup_{n\in \mathbb{N}}{\|\bfv_n\|_{{L^p(I,W^{1,p}(\mathcal{T}_{h_{n}}))\cap
	L^\infty(I,L^2(\Omega))}}}<\infty$.  
Since
$\bfv_n(t)\hspace{-0.1em}\in \hspace{-0.1em}\Vhnk(0)\hspace{-0.1em}\subseteq\hspace{-0.1em}
\smash{W^{1,p}(\mathcal{T}_{h_n})}$~for~almost~every~$t\hspace{-0.1em}\in\hspace{-0.1em}
I$~and~all~$n\hspace{-0.1em}\in \hspace{-0.1em}\mathbb{N}$,~we~can~extend this function
for almost every $t\in I$ by zero to $\widehat{\Omega}\setminus
\Omega$, where $\widehat {\Omega}$ is introduced in Section
\ref{sec:triang}, and denote this extension by 
$\overline{\bfv}_n(t)$. Thus, 
$\overline{\bfv}_n(t)$ belongs to $\smash{W^{1,p}(\widehat{\mathcal{T}}_{h_n})}$
for almost every $t\in I$ and all $n\in \mathbb{N}$.  In particular, it holds
$\Ghnk\overline{\bfv}_n(t)=\Ghnk \bfv_n(t)$ in $\Omega$ and
$\Ghnk\overline{\bfv}_n(t)=\mathbf{0}$ in
$\widehat{\Omega}\setminus \Omega$ for almost every $t\in I$ and all
$n\in \mathbb{N}$.  Then, owing to \eqref{eq:eqiv0}, we have
that
\begin{align}
\begin{aligned}
\sup_{n\in \mathbb{N}}{\big\|\Ghnk\overline{\bfv}_n\big\|_{L^p(I,L^p( \widehat{\Omega}))}}
=\sup_{n\in \mathbb{N}}{\big\|\Ghnk \bfv_n\big\|_{L^p(I,L^p(\Omega))}}
\leq \sup_{n\in \mathbb{N}}{\|\bfv_n\|_{L^p(I,\WDGn)}}<\infty\,.
\end{aligned}\label{eq:3.3.1}
\end{align}
Thus, we obtain a cofinal subset $\Lambda\subseteq \mathbb{N}$  and functions
$\bfv\in L^\infty(I,L^2(\widehat{\Omega}))$ and 
$\bfG\in L^p(I,L^p(\widehat{\Omega}))$ such that
\begin{alignat}{3}
\overline{\bfv}_n&\;\;\overset{\ast}{\rightharpoondown}\;\;\bfv
&&\quad\text{ in }L^\infty(I,L^2(\widehat{\Omega}))&&\quad(\Lambda\ni n\to\infty)\,,\label{eq:3.3.2}\\
\Ghnk\overline{\bfv}_n&\;\;\weakto\;\;\bfG 
&&\quad\text{ in }L^p(I,L^p(\widehat{\Omega}))&&\quad(\Lambda\ni n\to \infty)\,.\label{eq:3.3.3}
\end{alignat}
We next prove that $\bfv \hspace{-0.1em}\in \hspace{-0.1em} L^p(I,\Vo(0))$ with 
$\nabla \bfv \hspace{-0.1em}=\hspace{-0.1em}\bfG$ in $L^p(I,L^p(\widehat{\Omega}))$. To this
end, let $\bfX \hspace{-0.1em}\in \hspace{-0.1em}
C_0^\infty(\widehat{\Omega})$,~${\varphi \hspace{-0.1em}\in \hspace{-0.1em} C_0^\infty(I)}$, and observe, using
\eqref{eq:DGnablaR1} on $\widehat{\Omega}$ and element-wise integration-by-parts, that for every $n\in \mathbb{N}$, it holds
\begin{align}
&\int_I(\Ghnk\overline{\bfv}_n(t),\bfX)_{\widehat{\Omega}}\varphi(t)\,\mathrm{d}t
=\int_I{(\nabla_{h_n}\overline{\bfv}_n(t),\bfX)_{\widehat{\Omega}}\varphi(t)\,\mathrm{d}t}
-\int_I{\langle \jump{\overline{\bfv}_n(t)\otimes\bfn},
\{\PiDGn\bfX\}\rangle_{\widehat{\Gamma}_{h_n}}\varphi(t)\,\mathrm{d}t}\notag\\&
=-\int_I{(\overline{\bfv}_n(t),\divo\bfX)_{\widehat{\Omega}}\varphi(t)\,\mathrm{d}t}
+\int_I{\langle \jump{\overline{\bfv}_n(t)\otimes\bfn},
\{\bfX-\PiDGn\bfX\}\rangle_{\widehat{\Gamma}_{h_n}}\varphi(t)\,\mathrm{d}t}\,.\label{eq:3.3.4}
\end{align}
Applying the H\"older inequality and
$\|\bfX-\PiDG\bfX\|_{p',\widehat{\Gamma}_h}
\leq c\,h^{\smash{1-\frac{1}{p'}}}\|\bfX\|_{1,p',\widehat{\Omega}}$  
(cf.~\cite[Corollary~A.19]{kr-phi-ldg}) we find, for every ${n\in \mathbb{N}}$,  that
\begin{align*}
\vert\langle \jump{\overline{\bfv}_n(t)\otimes\bfn},
\{\bfX-\PiDGn\bfX\}\rangle_{\smash{\widehat{\Gamma}_{h_n}}}\vert
&\leq c\, h_n^{\smash{1-\frac{1}{p}}}\|\{\bfX-\PiDGn\bfX\}\|_{p',\widehat{\Gamma}_{h_n}}\|\bfv_n(t)\|_{\nabla,p,h_n}
\\&\leq 
c\,h_n\,\|\bfX\|_{1,p',\widehat{\Omega}}\|\bfv_n(t)\|_{\nabla,p,h_n}\,,
\end{align*}
and, therefore, that
\begin{align}
\begin{aligned}
\bigg\vert\int_I{\langle \jump{\overline{\bfv}_n(t)\otimes\bfn},
	\{\bfX-\PiDGn\bfX\}\rangle_{\smash{\widehat{\Gamma}_{h_n}}}\varphi(t)\,\mathrm{d}t}\bigg\vert
\leq c\,h_n\,\|\bfX\|_{1,p',\widehat{\Omega}}
\|{\bfv}_n\|_{L^p(I,\WDGn)}\overset{n\to\infty}{\to}0\,.
\end{aligned}\label{eq:3.3.5}
\end{align}
If we pass in \eqref{eq:3.3.4} for $n\to \infty$, taking into account \eqref{eq:3.3.2}, 
\eqref{eq:3.3.3} and \eqref{eq:3.3.5} in doing so, then,   for 
every $\bfX\in C_0^\infty(\widehat{\Omega})$ 
and $\varphi\in C_0^\infty(I)$, we obtain
\begin{align*}
\int_I{(\bfG(t),\bfX)_{\widehat{\Omega}}\varphi(t)\,\mathrm{d}t}=-\int_I{(\bfv(t),\divo\bfX)_{\widehat{\Omega}}\varphi(t)\,\mathrm{d}t}\,,
\end{align*}
i.e., it holds $\bfv\in  L^\infty(I,L^2(\widehat{\Omega}))$ 
with $\nabla \bfv=\bfG$ in $L^p(I,L^p(\widehat{\Omega}))$. On the other hand,  looking~back~to~\eqref{eq:3.3.2}, it is readily seen that $\bfv(t)=\mathbf{0}$ in 
$\widehat{\Omega}\setminus\Omega$ for almost~every~${t\in I}$. Consequently,  we, eventually, find that
$\bfv\in L^p(I,\Vo)\cap L^\infty(I,L^2(\Omega))$.

Next, we prove that $\bfv\hspace{-0.1em}\in \hspace{-0.1em}
L^p(I,\Vo(0))\cap L^\infty(I,\Ho(0))$. To this end, let
${z\hspace{-0.1em}\in
\hspace{-0.1em}C_0^\infty(\Omega)}$~and~${\varphi\hspace{-0.1em}\in\hspace{-0.1em}
C_0^\infty(I)}$ be arbitrary. Then, 
$z_{h_n}\coloneqq \mathcal{I}_{h_n}^\ell z\in \Qhkc$, $n\in \mathbb{N}$, where $\mathcal{I}_{h_n}^\ell\colon C^0(\overline{\Omega})\to \Qhkc$ is the 
nodal interpolation operator, satisfies $z_{h_n}\to z$ in $W^{1,2}(\Omega)$ $(n\to \infty)$. Thus,~since $(\bfv_n(t),\nabla z_{h_n})_\Omega=-(\Divhnk\!\bfv_n(t), z_{h_n})_\Omega=0$ for a.e. $t\in I$ and all $n\in \mathbb{N}$, due to $\bfv_{n}(t)\in \Vhnk(0)$ for almost every $t\in I$ and all $n\in \mathbb{N}$, it~holds
\begin{align*}
\int_I{(\divo \bfv(t),z)_\Omega\varphi(t)\,\mathrm{d}t}=\int_I{(\bfv(t),\nabla z)_\Omega\varphi(t)\,\mathrm{d}t}
=\lim_{n\to\infty}{\int_I{(\bfv_{n}(t),\nabla z_{h_n})_\Omega\varphi(t)\,\mathrm{d}t}}=0\,,
\end{align*}
i.e., $\bfv\in L^p(I,\Vo(0))\cap L^\infty(I,\Ho(0))$.
\end{proof}

Next, we show that the LDG and the SIP approximation of the extra
stress tensor $\SSS$ are non-conforming pseudo-monotone. This enables
us to use Theorem \ref{3.9}. 
\begin{proposition}[Stress tensor, LDG]\label{ldg}
Let $\SSS$ satisfy Assumption \ref{assum:extra_stress}, let $\ell
\in \setN$, and let
$\alpha>0$. Moreover, for every $n\in \mathbb{N}$, let
$S_n^{\textup{\textbf{\textsf{\tiny LDG}}}}\colon\WDGn\to
\WDGn^*$,  for every $\bfw_{h_n},\bfz_{h_n}\in \WDGn$, be
defined~via  
\begin{align*}
\langle S_n^{\textup{\textbf{\textsf{\tiny LDG}}}}\bfw_{h_n},\bfz_{h_n}\rangle_{\smash{\WDGn}}
\coloneqq \big(\SSS(\Dhnk\!\bfw_{h_n}),
\Dhnk\!\bfz_{h_n}\big)
+\alpha\big\langle \SSS_{\avg{\abs{\Dhnk\!\bfw_{h_n}}}}(h_n^{-1}\jump{\bfw_{h_n}\otimes\bfn}),
\jump{\bfz_{h_n}\otimes\bfn}\big\rangle_{\Gamma_{h_n}}\,.
\end{align*}
Then, it holds:
\begin{itemize}
\item[(i)] For every $n\in \mathbb{N}$, the operator
$S_n^{\textup{\textbf{\textsf{\tiny LDG}}}}\colon\WDGn\to
\WDGn^*$ is well-defined, bounded, continuous, monotone,
and, thus, pseudo-monotone. In addition, for every
$\varepsilon> 0$, there exists a constant~$c_\vep>0$, 
depending only on $\vep^{-1}$, $\ell$, $\alpha$, the characteristics of
$\SSS$, $\delta$, $\abs{\Omega}$, and $\omega_0$, 
such that, for every $n\in \mathbb{N}$ and
$\bfw_{h_n},\bfz_{h_n}\in \WDGn$, it holds
\begin{align*}
\vert\langle S_n^{\textup{\textbf{\textsf{\tiny LDG}}}}\bfw_{h_n},\bfz_{h_n}\rangle_{\smash{\WDGn}}\vert
\leq \varepsilon\,\|\bfw_{h_n}\|_{\nabla,p,h_n}^p+c_{\varepsilon}\,(1+\|\bfz_{h_n}\|_{\nabla,p,h_n}^p)\,.
\end{align*}
\item[(ii)] There exist
constants $c_0,c_1> 0$, depending only on $\ell$, 
$\alpha$, the characteristics of $\SSS$, $\delta$, $\abs{\Omega}$, and $\omega_0$, such that for  every $n\in \mathbb{N}$ and $\bfv_{h_n}\in \Vhnk(0)$, it holds 
\begin{align*}
\langle S_n^{\textup{\textbf{\textsf{\tiny LDG}}}}\bfv_{h_n},\bfv_{h_n}\rangle_{\smash{\WDGn}}
\ge c_0\,\|\bfv_{h_n}\|_{\nabla,p,h_n}^p-c_1\,.
\end{align*}
\item[(iii)] The operators $S_n^{\textup{\textbf{\textsf{\tiny LDG}}}}\colon\WDGn\to \WDGn^*$, $n\in\mathbb{N}$, 
are non-conforming pseudo-monotone with respect 
to $(\Vhnk(0))_{n\in \mathbb{N}}$ and
$S\colon\Vo(0)\to \Vo(0)^*$,   for every $\bfv,\bfz\in
\Vo(0)$, defined via 
$\langle S\bfv,\bfz\rangle_{\Vo(0)}\coloneqq (\SSS(\bfD\bfv),
\bfD\bfz)$.
\end{itemize}
\end{proposition}

\begin{proof}
\textit{ad (i).} Due to \cite[Lemma  5.2]{kr-pnse-ldg-1}, for every $n\hspace*{-0.1em}\in\hspace*{-0.1em} \mathbb{N}$, the operator 
$S_n^{\textup{\textbf{\textsf{\tiny LDG}}}}\colon\WDGn\hspace*{-0.1em}\to\hspace*{-0.1em} \WDGn^*$~is~well-defined, bounded, continuous, 
monotone,~and,~thus,~pseudo-monotone. Moreover, in
\cite[Lemma~5.2,~(5.8)]{kr-pnse-ldg-1} it is proved that there exists a constant $c>0$,  depending on $\alpha$, the characteristics of $\SSS$, and $\omega_0$,~such~that for every $\bfw_{h_n}\in \WDGn$,  it holds
\begin{align}\label{eq:ldg.1}
\|S_n^{\textup{\textbf{\textsf{\tiny LDG}}}}\bfw_{h_n}\|_{\WDGn^*}^{p'}\leq c\, \|\bfw_{h_n}\|_{\nabla,p,h_n}^p+c\,\delta^p\vert \Omega\vert\,.
\end{align}
Therefore, applying the $\varepsilon$-Young inequality  with $\psi=\vert \cdot\vert^p$ and using \eqref{eq:ldg.1}, we conclude that
\begin{align*}
\vert\langle S_n^{\textup{\textbf{\textsf{\tiny LDG}}}}\bfw_{h_n},\bfz_{h_n}\rangle_{\smash{\WDGn}}\vert &
\leq \varepsilon\,	\|S_n^{\textup{\textbf{\textsf{\tiny LDG}}}}\bfw_{h_n}\|_{\WDGn^*}^{p'}+c_{\varepsilon}\,\|\bfz_{h_n}\|_{\nabla,p,h_n}^p
\\&\leq \varepsilon\,c\, \|\bfw_{h_n}\|_{\nabla,p,h_n}^p+\varepsilon\,c\,\delta^p\vert \Omega\vert+c_{\varepsilon}\,\|\bfz_{h_n}\|_{\nabla,p,h_n}^p\,,
\end{align*}
which is the assertion.\\[-2mm]

\textit{ad (ii).} Due to \cite[Lemma  5.2, (5.4)]{kr-pnse-ldg-1}, there exists a constant $c>0$, depending on $\alpha$, the characteristics of $\SSS$, and $\omega_0$,~such~that for every  $\bfv_{h_n}\in \Vhnk(0)$, $n\in \mathbb{N}$,~it~holds
\begin{align*}
\langle S_n^{\textup{\textbf{\textsf{\tiny LDG}}}}\bfv_{h_n},\bfv_{h_n}\rangle_{\smash{\WDGn}}\ge c\,\|\bfv_{h_n}\|_{\nabla,p,h_n}^p-c\,\delta^p\,\vert \Omega\vert\,.
\end{align*}

\textit{ad (iii).} See \cite[Lemma 5.2]{kr-pnse-ldg-1}.
\end{proof}

\begin{proposition}[Stress tensor, SIP (lifting)]\label{SIP}
Let  $\SSS$ satisfy Assumption
\ref{assum:extra_stress}, let $\ell \in \setN$, and let
${\alpha >0}$. Moreover, for every $n\!\in\! \mathbb{N}$, let
$S_n^{\textup{\textsf{\textbf{\tiny
		SIP}}}}\colon\!\WDGn\!\to\! \WDGn^*$, for every
$\bfw_{h_n},\bfz_{h_n}\!\in \!\WDGn$,~\mbox{be~defined~via}
\begin{align*}
\langle S^{\textup{\textsf{\textbf{\tiny SIP}}}}_n\bfw_{h_n},\bfz_{h_n}\rangle_{\smash{\WDGn}}&
\coloneqq \big(\SSS(\Dhnk\!\bfw_{h_n}),\Dhnk\!\bfz_{h_n}\big)
+\alpha\big\langle \SSS(h_n^{-1}\jump{\bfw_{h_n}\otimes\bfn}),
\jump{\bfz_{h_n}\otimes\bfn}\big\rangle_{\Gamma_{h_n}}
\\&\quad-\big(\SSS(\Rhnks\bfw_{h_n}),\Rhnks\bfz_{h_n}\big)\,.
\end{align*}
Then, it holds:
\begin{itemize}
\item[(i)] For every $n\in \mathbb{N}$, the operator
$S^{\textup{\textsf{\textbf{\tiny SIP}}}}_n\colon\WDGn\to \WDGn^*$ is well-defined,~bounded~and~continuous,
and, thus, pseudo-monotone. In addition, for every
$\varepsilon> 0$, there exists a constant~$c_\vep>0$,
depending only on $\vep^{-1}$, $\ell$, $\alpha$, the characteristics of
$\SSS$, $\delta$, $\abs{\Omega}$, and $\omega_0$, 
such that, for every $n\in \mathbb{N}$ and
$\bfw_{h_n},\bfz_{h_n}\in \WDGn$, it holds
\begin{align*}
\vert \langle S^{\textup{\textsf{\textbf{\tiny SIP}}}}_n\bfw_{h_n},\bfz_{h_n}\rangle_{\smash{\WDGn}}\vert 
\leq \varepsilon\, \|\bfw_{h_n}\|_{\nabla,p,h_n}^p+c_{\varepsilon}\,(1+\|\bfz_{h_n}\|_{\nabla,p,h_n}^p)\,.
\end{align*}
\item[(ii)] There exist a constant
  $\alpha_0>0$, depending only on $\ell$, the characteristics of
$\SSS$, and $\omega_0$ such that for every $\alpha >\alpha_0$,
there exists  constants $c_0, c_1$, depending only on $\ell$, the characteristics of
$\SSS$, $\delta$, $\abs{\Omega}$,  $\omega_0$, and $\alpha$, such that for every $n\in \mathbb{N}$~and~${\bfv_{h_n}\in \Vhnk(0)}$, it holds 
\begin{align*}
\smash{\langle S^{\textup{\textsf{\textbf{\tiny SIP}}}}_n\bfv_{h_n},\bfv_{h_n}\rangle_{\smash{\WDGn}}
	\ge c_0\,\|\bfv_{h_n}\|_{\nabla,p,h_n}^p-c_1\,.}
\end{align*}
\item[(iii)] The operators $S^{\textup{\textsf{\textbf{\tiny SIP}}}}_n\colon\WDGn\to \WDGn^*$, $n\in\mathbb{N}$, 
are non-conforming pseudo-monotone with 
respect to $(\Vhnk(0))_{n\in \mathbb{N}}$ and $S\colon\Vo(0)\to \Vo(0)^*$ 
from Proposition \ref{ldg}, if
$\alpha>\alpha_0$ with $\alpha_0$ from (ii).
\end{itemize}
\end{proposition}

\begin{proof}
\textit{ad (i).} Well-definedness, boundedness and
continuity follow from the theory of Nemytski\u{\i}
operators. For the desired inequality, we proceed for the
first two terms as in Proposition~\ref{ldg}.  For the third
term, we use the $\vep$-Young inequality \eqref{ineq:young} with $\psi=\varphi$,
\eqref{1.5b}, $\varphi^*(\varphi '(t)) \sim \varphi (t)$
uniformly in $t\ge 0$, the estimate
\begin{align}
\int_{\Omega}\phi (\Rhk\bfw_h)\, \textup{d}x\leq
c_2\,h\,\int_{\Gamma_h}\hspace*{-2mm}\phi
(h^{-1}\jump{\bfw_h\otimes\bfn})\, \textup{d}s\,,\label{eq:rh}
\end{align}
for all
$\bfw_{h_n}\!\!\in\!
\WDGn$~and~${n\!\in\!
\mathbb{N}}$, where $c_2\!>\!0$ depends~on~$\ell$,~the characteristics
of $\SSS$ and $\omega_0$~(cf.~\mbox{\cite[\hspace*{-1mm}(A.25)]{dkrt-ldg}}), that
${\phi(t)\le c\,t^p +c\,\delta^p}$
for~every~${t,\delta\ge 0}$, where $c>0$ depends on
$p\in (1,\infty)$, and that ${h_n\mathscr{H}^{d-1}(\Gamma_{h_n})\leq c\,\vert
\Omega\vert }$, where $c>0$ depends on
$\omega_0>0$, to obtain
\begin{align*}
\bigabs{\big(\SSS(\Rhnks\bfw_{h_n}),\Rhnks\bfz_{h_n}\big)}
&\le \vep \, \int_\Omega \varphi^*(\abs{\SSS(\Rhnks\bfw_{h_n})})\,
\textup{d}x+ c_\vep \, \int_\Omega \varphi(\abs{\Rhnks\bfz_{h_n}})\,
\textup{d}x
\\
&\le \vep \,\int_\Omega \varphi(\abs{\Rhnk\bfw_{h_n})})\,
\textup{d}x+c_\vep \, \int_\Omega \varphi(\abs{\Rhnks\bfz_{h_n}})\,
\textup{d}x
\\
&\le \vep \,  c_2\,h\,\int_{\Gamma_{h_n}}\hspace*{-2mm}\phi
(h^{-1}\jump{\bfw_{h_n}\otimes\bfn})\, \textup{d}s +
c_\vep \,  c_2\,h\,\int_{\Gamma_{h_n}}\hspace*{-2mm}\phi
(h^{-1}\jump{\bfz_{h_n}\otimes\bfn})\, \textup{d}s 
\\
& \le \vep \, c\,
\|\bfw_{h_n}\|_{\nabla,p,h_n}^p+c_{\varepsilon}\,(\delta^p
\, \abs{\Omega}+\|\bfz_{h_n}\|_{\nabla,p,h_n}^p)\,.
\end{align*}

\textit{ad (ii).} Combining \eqref{eq:ass_S},
$\varphi^*(\varphi '(t)) \le \Delta_2(\varphi)\, \varphi (t)$
uniformly in $t\ge 0$, 
\eqref{eq:rh}, 
and that $\phi(t)\ge c\,t^p-c\,\delta^p$
for~every~${t,\delta\ge 0}$, we obtain for almost
every $t\in I$ and every $\bfv_{h_n}\in \Vhnk(0)$,
$n\in \mathbb{N}$, also using that
$h_n\mathscr{H}^{d-1}(\Gamma_{h_n})\leq c\,\vert
\Omega\vert $, where $c>0$~depends~on~$\omega_0>0$, that
\begin{align}
\begin{aligned}
&\alpha\big\langle
\SSS(h_n^{-1}\jump{\bfv_{h_n}\otimes\bfn}),
\jump{\bfv_{h_n}\otimes\bfn}\big\rangle_{\Gamma_{h_n}}
-\big(\SSS(\Rhnks\bfv_{h_n}),\Rhnks\bfv_{h_n}\big)
\\
&\ge \alpha\,
C_0\,h_n\,\rho_{\varphi,\Gamma_{h_n}}(h_n^{-1}\jump{\bfv_{h_n}\otimes\bfn})-C_1\, \Delta_2(\varphi)\,\rho_{\varphi,\Omega}(\Rhnks\bfv_{h_n})
\\
&\ge
(\alpha\,C_0-C_1\, \Delta_2(\varphi) \,c_2)\,h_n\,\rho_{\varphi,\Gamma_{h_n}}(h_n^{-1}\jump{\bfv_{h_n}\otimes\bfn})
\\
&\ge c\, (\alpha\,C_0-C_1\, \Delta_2(\varphi) \,c_2)\,h_n\,\|h_n^{-1}\jump{\bfv_{h_n}\otimes\bfn}\|_{p,\Gamma_{h_n}}^p- c\, (\alpha\,C_0-C_1\, \Delta_2(\varphi) \,c_2)\,\delta^p\,h_n\,\mathscr{H}^{d-1}(\Gamma_{h_n})
\\
&\ge  c\, (\alpha\,C_0-C_1\, \Delta_2(\varphi) \,c_2)\,h_n\,\|h_n^{-1}\jump{\bfv_{h_n}\otimes\bfn}\|_{p,\Gamma_{h_n}}^p-c\,\delta^p\,\vert\Omega\vert
\,.
\end{aligned}\label{eq:SIP1}
\end{align}
Similarly, we obtain that
$(\SSS(\Dhnk\!\bfv_{h_n}),\Dhnk\!\bfv_{h_n}) \ge
c\,{C_0}\,\|\Dhnk\!\bfv_{h_n}\|_p^p -c\,\delta^p\vert
\Omega\vert$. Thus, using the Korn~\mbox{inequality}
(cf.~Proposition \ref {korn}), the assertion follows,
if $\alpha > \frac {C_1\, \Delta_2(\varphi)\, c_2}{C_0}=\vcentcolon\alpha_0$.\\[-3mm]

\textit{ad (iii).} 		Without loss of generality, we assume that $m_n=n$ for
every $n\in \mathbb{N}$ in Definition
\ref{3.4.0}. Therefore, let $\bfv_{h_n}\in V_{n}^\ell(0)$,
$n\in \mathbb{N}$,
be a sequence~such~that
\begin{gather}
\sup_{n\in \mathbb{N}}{\|\bfv_{h_n}\|_{\nabla,p,h_{n}}}<\infty\,,
\qquad\bfv_{h_n}\;\;\weakto\;\;\bfv\quad\text{ in }L^2(\Omega)\quad(n\to \infty)\,,\label{eq:SIP.3}
\\
\limsup_{n\to\infty}{\langle S^{\textup{\textsf{\textbf{\tiny SIP}}}}_{n}\bfv_{h_n}\,,
\bfv_{h_n}-\bfv\rangle_{W^{1,p}(\mathcal{T}_{h_{n}})}}\leq 0\,.\label{eq:SIP.4}
\end{gather}
Proposition \ref{3.2} and \eqref{eq:SIP.3} immediately
yield $\bfv\in \Vo(0)$.  In addition, on the
basis of \cite[Lemma~A.37]{kr-phi-ldg},
\eqref{eq:SIP.3} implies that
\begin{align}\label{eq:weak}
\Dhnk\!\bfv_{h_n}\rightharpoonup\bfD\bfv\quad\text{ in }L^p(\Omega)\quad(n\to \infty)\,.
\end{align}
For arbitrary $\bfz\hspace{-0.05em}\in\hspace{-0.05em} \Vo(0)$ and $s\hspace{-0.05em}\in\hspace{-0.05em}
\left(0,1\right)$, we set ${\bfz_s\hspace{-0.05em}\coloneqq \hspace{-0.05em}(1-s)\bfv+s\bfz\hspace{-0.05em}\in\hspace{-0.05em} \Vo(0)}$.
Observing~that~${\jump{\bfz_s\otimes \bfn}\hspace{-0.05em}=\hspace{-0.05em}\mathbf{0}}$~on~$\Gamma_{h_n}$ and, thus, $\Rhnks\bfz_s=\mathbf{0}$ in 
$L^p(\Omega)$, inasmuch as $\bfz_s\in \Vo(0)$, and that the third line in \eqref{eq:SIP1} 
is non-negative, for $\alpha >\alpha_0$, we deduce that
\begin{align*}
&\alpha\big\langle \SSS(h_n^{-1}\jump{\bfv_{h_n}\otimes\bfn})
-\SSS(h_n^{-1}\jump{\bfz_s\otimes\bfn}),
\jump{(\bfv_{h_n}-\bfz_s)\otimes\bfn}\big\rangle_{\Gamma_{h_n}}
\\&\quad-\big(\SSS(\Rhnks\bfv_{h_n})-\SSS(\Rhnks\bfz_s),
\Rhnks(\bfv_{h_n}-\bfz_s)\big)
\\&= \alpha\big\langle \SSS(h_n^{-1}\jump{\bfv_{h_n}\otimes\bfn}),
\jump{\bfv_{h_n}\otimes\bfn}\big\rangle_{\Gamma_{h_n}}
-\big(\SSS(\Rhnks\bfv_{h_n}),\Rhnks\bfv_{h_n}\big)
\ge 0\,.
\end{align*}
This and the the monotonicity condition \eqref{1.4b} imply that
\begin{align*}
\langle S^{\textup{\textsf{\textbf{\tiny
			SIP}}}}_n\bfv_{h_n}- S^{\textup{\textsf{\textbf{\tiny
			SIP}}}}_n\bfz_s,\bfv_{h_n}-\bfz_s\rangle_{\smash{\WDGn}}\ge0\,.
\end{align*}
Using standard arguments (cf.~\cite{zei-IIB}), we infer from this
\begin{align}
&s\langle S^{\textup{\textsf{\textbf{\tiny
			SIP}}}}_n\bfv_{h_n},\bfv_{h_n}-\bfz\rangle_{\smash{\WDGn}}			\label{eq:SIP.7}
\\&\ge -(1-s)\langle
S^{\textup{\textsf{\textbf{\tiny
			SIP}}}}_n\bfv_{h_n},\bfv_{h_n}-\bfv\rangle_{\smash{\WDGn}}
+(1-s)\langle
S^{\textup{\textsf{\textbf{\tiny SIP}}}}_n\bfz_s,
\bfv_{h_n}-\bfv\rangle_{\smash{\WDGn}}
+s\langle S^{\textup{\textsf{\textbf{\tiny
			SIP}}}}_n\bfz_s,\bfv_{h_n}-\bfz\rangle_{\smash{\WDGn}} \notag
\\[-0.5mm]&\ge -(1-s)\langle
S^{\textup{\textsf{\textbf{\tiny
			SIP}}}}_n\bfv_{h_n},\bfv_{h_n}-\bfv\rangle_{\smash{\WDGn}}
+(1-s)\big(\SSS(\bfD\bfz_s),
\Dhnk\!\bfv_{h_n}-\bfD\bfv\big)
+s\big(\SSS(\bfD\bfz_s),
\Dhnk\!\bfv_{h_n}-\bfD\bfz\big)\,,\notag
\end{align}
where we also used $\langle \SSS(h_n^{-1}\jump{\bfz_s\otimes \bfn}),
\jump{(\bfv_{h_n}-\bfz)\otimes\bfn}\rangle_{\Gamma_{h_n}}
\hspace{-0.1em}=\hspace{-0.1em}(
\SSS(\Rhnks\bfz_s,\Rhnks(\bfv_{h_n}-\bfz))\hspace{-0.1em}=\hspace{-0.1em}0$
for ${\bfz \hspace{-0.1em}\in\hspace{-0.1em} \Vo(0)}$.
Dividing by $s>0$, taking the limit inferior on both sides 
in \eqref{eq:SIP.7} with respect to $n\in \mathbb{N}$ and using 
\eqref{eq:SIP.4}, \eqref{eq:weak}, eventually, yields for every $\bfz\in \Vo(0)$
\begin{align*}
\begin{aligned}
\liminf_{n\to \infty}{\langle S^{\textup{\textsf{\textbf{\tiny SIP}}}}_n\bfv_{h_n},
	\bfv_{h_n}-\bfz\rangle_{\smash{\WDGn}}}\ge\langle S\bfz_s,\bfv-\bfz\rangle_{\Vo(0)}
\overset{s\to 0}{\to} \langle S\bfv,\bfv-\bfz\rangle_{\Vo(0)}\,.
\end{aligned}
\end{align*}
\end{proof}
\begin{proposition}[Convective term I, cf.~\cite{EP12}]\label{3.4.1}
Let $d\in \{2,3\}$, $\ell \in \setN$ and  $p>\frac{3d+2}{d+2}$.
Moreover, for every $n\in \mathbb{N}$, let ${B_{n}^{\textup{I}} \colon\WDGn\to \WDGn^*}$, for every $\bfw_{h_n},\bfz_{h_n}\in \WDGn$, be defined via
\begin{align*}
\langle B_{n}^{\textup{I}}\bfw_{h_n},\bfz_{h_n}\rangle_{\smash{\WDGn}}
&\coloneqq ([\nabla_{h_n}\!\bfw_{h_n}]\bfw_{h_n},\PiDGn\!\bfz_{h_n})+\tfrac{1}{2}((\mathrm{div}_{h_n}\!\bfw_{h_n})\bfw_{h_n},\PiDGn\!\bfz_{h_n})
\\&\quad-\langle\jump{\PiDGn\bfw_{h_n}\otimes\bfn},
\{\PiDGn\bfw_{h_n}\}\otimes\{\PiDGn\!\bfz_{h_n}\}\rangle_{\smash{\Gamma_{h_n}^{i}}}
\\&\quad-\tfrac{1}{2}\langle \tr \jump{ \PiDGn\bfw_{h_n}\otimes\bfn},
\{\PiDGn\bfw_{h_n}\cdot\PiDGn\!\bfz_{h_n}\}\rangle_{\Gamma_{h_n}}\,.
\end{align*}
Then, it holds:
\begin{itemize}
\item[(i)] For every $n\in \mathbb{N}$, the operator $B_{n}^{\textup{I}} \colon\WDGn\to \WDGn^*$ is well-defined, 
bounded and continuous. In addition, for every
$p$ there exist $q\in \left(1,\infty\right)$
such that for every $\varepsilon>0$ there
exists $c_\vep>0$, depending only on
$\vep^{-1}$, $\ell$, $p$
and $\omega_0$, such that for all  $n\in
\mathbb{N}$ and all $\bfw_{h_n},\bfz_{h_n}\in \WDGn$,~it~holds
\begin{align*}
\vert \langle B_{n}^{\textup{I}}\bfw_{h_n},\bfz_{h_n}\rangle_{\smash{\WDGn}}\vert
\leq \varepsilon\,\|\bfw_{h_n}\|_{\nabla,p,h_n}^p
+c_{\varepsilon}\,\big
(1+\|\bfw_{h_n}\|_2^q+\|\bfz_{h_n}\|_2^q+\|\bfz_{h_n}\|_{\nabla,p,h_n}^p\big
)\,.
\end{align*}
\item[(ii)] For every $n\in \mathbb{N}$ and
$\bfv_{h_n}\in \Vhnk $, it 
holds  $\langle B_{n}^{\textup{I}} \bfv_{h_n},\bfv_{h_n}\rangle_{\smash{\WDGn}}=0$.
\item[(iii)] The operators $B_{n}^{\textup{I}} \colon\WDGn\to \WDGn^*$, $n\in\mathbb{N}$, are 
non-conforming pseudo-monotone with respect to $(\Vhnk(0))_{n\in \mathbb{N}}$ 
and $B\colon\Vo(0)\to \Vo(0)^*$, for every $\bfv,\bfz\in \Vo(0)$, defined via $\langle B\bfv,\bfz\rangle_{\Vo(0)}
\coloneqq ([\nabla\bfv]\bfv,\bfz)$.
\end{itemize}
\end{proposition}

\begin{proof} We only give a proof for the case $d=3$, since the case $d=2$ 
simplifies due to better~Sobolev~embeddings (cf.~\cite[Proposition 2.6]{kr-pnse-ldg-3}).

\textit{ad (i).} The well-definedness, boundedness and
continuity of the operator $B_{n}^{\textup{I}}$ follows in~the~same~way~as~in \cite[\hspace{-0.1mm}Lemma \hspace{-0.1mm}6.40]{EP12}.
\hspace{-0.5mm}The \hspace{-0.1mm}stability \hspace{-0.1mm}property \hspace{-0.1mm}of \hspace{-0.1mm}$\PiDGn$ \!(cf.~\!\cite[\hspace{-0.1mm}(A.11), \hspace{-0.1mm}(A.18)]{dkrt-ldg})
\hspace{-0.1mm}yields,~\hspace{-0.1mm}for~\hspace{-0.1mm}every~\hspace{-0.1mm}${\bfw_{h_n},\bfz_{h_n}\!\in\! \WDGn} $, that
\begin{align}
\begin{aligned}
&\abs{([\nabla_{h_n}\!\bfw_{h_n}]\bfw_{h_n},\PiDGn\!\bfz_{h_n})}
+\abs{((\mathrm{div}_{h_n}\!\bfw_{h_n})\bfw_{h_n},\PiDGn\!\bfz_{h_n})}
\\
&\leq c\,\|\bfw_{h_n}\|_{2p'}\|\smash{\PiDGn}\!\bfz_{h_n}\|_{2p'}
\|\nabla_{h_n}\!\bfw_{h_n}\|_p\,\\&\leq 
c\,\|\bfw_{h_n}\|_{2p'}\|\bfz_{h_n}\|_{2p'}
\|\bfw_{h_n}\|_{\nabla,p,h_n}\,.
\end{aligned}\label{eq:3.4.1.c}
\end{align}
To treat the remaining two terms in the definition of
$B_{n}^{\textup{I}}$, we use the discrete trace inequality, 
i.e., it holds 
$\smash{h_n \|\bfu_{h_n}\|_{p',\gamma}^{p'}\leq c\, \|\bfu_{h_n}\|_{p',K}^{p'}}$ for all $\bfu_{h_n}\in \Vhnk$
(cf.~\cite[Lemma A.16]{kr-phi-ldg}), H\"older
inequality, and the $L^{2p'}$- and the DG-stability properties of 
$\PiDGn$ (cf.~\cite[(A.11), (A.18)]{dkrt-ldg}), 
to obtain for~every ${\bfw_{h_n},\bfz_{h_n}\in \WDGn}$
\begin{align}\label{eq:3.4.1.h}
\begin{aligned}
&\abs{\langle\jump{\PiDGn\!\bfw_{h_n}\otimes \bfn},
	\avg{\PiDGn\!\bfw_{h_n}}\otimes\avg{\PiDGn\!\bfz_{h_n}}\rangle_{\smash{\Gamma_{h_n}^{i}}}}
+\abs{\langle\tr \jump{\PiDGn\!\bfw_{h_n}\otimes\bfn},
	\avg{\PiDGn\!\bfw_{h_n}\cdot\PiDGn\!\bfz_{h_n}}\rangle_{\smash{\Gamma_{h_n}}}}
\\
&\leq c\,{h_n^{\frac{1}{p}}}\|h_n^{-1}\jump{\PiDGn\!\bfw_{h_n}\otimes \bfn}\|_{p,\Gamma_{h_n}^{i}}
{h_n^{\frac{1}{p'}}}\|\avg{\PiDGn\!\bfw_{h_n}}\otimes
\avg{\PiDGn\!\bfz_{h_n}}\|_{p',\Gamma_{h_n}^{i}}
\\
&\quad+c\,{h_n^{\frac{1}{p}}}\|h_n^{-1}\jump{\PiDGn\!\bfw_{h_n}\otimes \bfn}\|_{p,\Gamma_{h_n}}
{h_n^{\frac{1}{p'}}}\|\avg{\PiDGn\!\bfw_{h_n}\cdot \PiDGn\!\bfz_{h_n}}\|_{p',\Gamma_{h_n}}
\\
&\leq c\,\|\PiDGn\!\bfw_{h_n}\|_{\nabla,p,h_n}
\|\PiDGn\!\bfw_{h_n}\|_{2p'}
\|\PiDGn\!\bfz_{h_n}\|_{2p'}
\\
&\leq c\,\|\bfw_{h_n}\|_{\nabla,p,h_n}
\|\bfw_{h_n}\|_{2p'}		\|\bfz_{h_n}\|_{2p'}\,.
\end{aligned}
\end{align}
Thus, for every $\bfw_{h_n},\bfz_{h_n}\in \WDGn$, we proved that
\begin{align*}
\vert\langle B_{n}^{\textup{I}} \bfw_{h_n},\bfz_{h_n}\rangle_{\smash{\WDGn}}\vert
\le  c\,\|\bfw_{h_n}\|_{\nabla,p,h_n}
\|\bfw_{h_n}\|_{2p'}		\|\bfz_{h_n}\|_{2p'}\,.
\end{align*}
To proceed, we distinguish between the cases $p\ge 3$
and $p<3$:

\textit{Case $p\in [3,\infty)$.} If $p\ge 3$, then
$2p'\leq p$ and $\frac{p-2}{2}\le p$. Thus, we have that  $a^{\frac{p-2}{2}}\leq 1+a^p$ for all
$a\ge 0$ and the validity of  the embedding $\WDGn \vnor
L^{2p'}(\Omega)$ (cf.~\cite[Proposition
2.6]{kr-pnse-ldg-3}) holds.~Using~these~results and the
$\varepsilon$-Young inequality \eqref{ineq:young} with $\psi=\vert
\cdot\vert^{\frac p2}$, we deduce that 
\begin{align}
\begin{aligned}
\|\bfw_{h_n}\|_{2p'}\|\bfz_{h_n}\|_{2p'}
\|\bfw_{h_n}\|_{\nabla,p,h_n}&\leq c\,\|\bfw_{h_n}\|_{\nabla,p,h_n}^2\|\bfz_{h_n}\|_{\nabla,p,h_n}
\\
&\leq {\varepsilon}\,\|\bfw_{h_n}\|_{\nabla,p,h_n}^{p}+c_\varepsilon\,\smash{\|\bfz_{h_n}\|_{\nabla,p,h_n}^{\frac{p-2}{2}}}
\\
&\leq 		\varepsilon\,\|\bfw_{h_n}\|_{\nabla,p,h_n}^p +c_{\varepsilon}\,(1+\|\bfz_{h_n}\|_{\nabla,p,h_n}^p)\,.
\end{aligned}\label{eq:3.4.1.d}
\end{align}

\textit{Case $p\!\in\! (\frac{11}{5},3)$.} \hspace*{-0.1mm}If \hspace*{-0.1mm}$p\!\in\! (\frac{11}{5},3)$, \hspace*{-0.1mm}then, \hspace*{-0.1mm}by \hspace*{-0.1mm}interpolation \hspace*{-0.1mm}with 
\hspace*{-0.1mm}$\frac{1}{\rho}\!=\!\frac{1-\theta}{p^*}+\frac{\theta}{2}$, 
\hspace*{-0.1mm}where \hspace*{-0.1mm}$\rho\!=\!\frac{5p}{3}$,~\hspace*{-0.1mm}${\theta\!=\!\frac{2}{5}}$~\hspace*{-0.1mm}and~\hspace*{-0.1mm}${p^*\!=\!\frac{3p}{3-p}}$,
and \cite[Proposition 2.6]{kr-pnse-ldg-3}, we obtain
\begin{align*}
\|\bfw_{h_n}\|_{\rho}\leq \|\bfw_{h_n}\|_2^\frac{2}{5}\|\bfw_{h_n}\|_{p^*}^{\frac{3}{5}}
\leq  c\,\|\bfw_{h_n}\|_2^\frac{2}{5}\|\bfw_{h_n}\|_{\nabla,p,h_n}^{\frac{3}{5}}\,. 
\end{align*}
Using this, $\rho\ge 2p'$, the $\varepsilon$-Young inequality \eqref{ineq:young} with $\psi=\vert\cdot\vert^{5p/3}$, we conclude that
\begin{align}
\begin{aligned}
\|\bfw_{h_n}\|_{2p'}\|\bfz_{h_n}\|_{2p'}
\|\bfw_{h_n}\|_{\nabla,p,h_n}&\leq c\,\|\bfw_{h_n}\|_2^\frac{2}{5}\|\bfw_{h_n}\|_{\nabla,p,h_n}^{\frac{8}{5}}
\|\bfz_{h_n}\|_2^\frac{2}{5}\|\bfz_{h_n}\|_{\nabla,p,h_n}^{\frac{3}{5}}
\\[-1mm]&\leq c\,\|\bfw_{h_n}\|_2^\frac{2p}{5p-3}\|\bfw_{h_n}\|_{\nabla,p,h_n}^{\frac{8p}{5p-3}}
\|\bfz_{h_n}\|_2^{\frac{2p}{5p-3}}+c\,\|\bfz_{h_n}\|_{\nabla,p,h_n}^{p}\,.
\end{aligned} \label{eq:3.4.1.f}
\end{align}
For $q\coloneqq \smash{2(\frac{5p-3}{8})'\frac{2p}{5p-3}}$, due to $\smash{\frac{8p}{5p-3}}<p$ for $p>\frac{11}{5}$, using
the $\varepsilon$-Young inequality \eqref{ineq:young} with $\psi=\vert\cdot\vert^{(5p-3)/8}$ in \eqref{eq:3.4.1.f}, 
we obtain 
\begin{align}
\begin{aligned}
\|\bfw_{h_n}\|_{2p'}\|\bfz_{h_n}\|_{2p'}
\|\bfw_{h_n}\|_{\nabla,p,h_n}
&\leq c\,\varepsilon\,\|\bfw_{h_n}\|_{\nabla,p,h_n}^{p}
+c_{\varepsilon}\,(1+\|\bfw_{h_n}\|_2^q+\|\bfz\|_2^q
+\|\bfz_{h_n}\|_{\nabla,p,h_n}^{p})\,.
\end{aligned}\label{eq:3.4.1.g}
\end{align}
Thus, we proved that for every
$\varepsilon>0$, there exists $c_\vep>0$ such that for every $n\in \mathbb{N}$ 
and ${\bfw_{h_n},\bfz_{h_n}\!\!\in\! \WDGn}$, it~holds
\begin{align}
\vert\langle B_{n}^{\textup{I}} \bfw_{h_n},\bfz_{h_n}\rangle_{\smash{\WDGn}}\vert
\leq \varepsilon\,\|\bfw_{h_n}\|_{\nabla,p,h_n}^p
+c_{\varepsilon}\,(1+\|\bfw_{h_n}\|_2^q+\|\bfz_{h_n}\|_2^q+\|\bfz_{h_n}\|_{\nabla,p,h_n}^p)\,.\label{eq:3.4.1.j}
\end{align}

\textit{ad (ii).} Here, we refer to \cite[Lemma 6.39]{EP12}.\\[-3mm]

\textit{ad (iii).}  Without loss of generality we
assume that $m_n=n$ for every $n\in \mathbb{N}$ in
Definition \ref{3.4.0}. Therefore, let
$\bfv_{h_n}\in V_{n}^\ell(0)$, $n\in \mathbb{N}$, be a
sequence~such~that
\begin{gather}
\sup_{n\in \mathbb{N}}{\|\bfv_{h_{n}}\|_{\nabla,p,h_n}}<\infty\,,
\qquad\bfv_{h_{n}}\weakto\bfv\quad\text{ in }L^2(\Omega)\quad(n\to \infty)\,,\label{eq:3.4.1.1}\\[-1mm]
\limsup_{n\to\infty}{\langle B_{n}^{\textup{I}} \bfv_{h_{n}},\bfv_{h_{n}}-\bfv\rangle_{W^{1,p}(\mathcal{T}_{h_n})}}\leq 0\,.
\label{eq:3.4.1.2}
\end{gather}
Proposition \ref{3.2} and \eqref{eq:3.4.1.1} gives us $\bfv\in \Vo(0)$. Moreover, due to 
\eqref{eq:3.4.1.1}, the discrete Rellich--Kondrachov theorem 
(cf.~\cite[Theorem 5.6]{EP12}) yields for $q<p^*$, with
$p^*\coloneqq \frac{3p}{3-p}$ if $p<3$ and
$p^*\coloneqq \infty$ if $p\ge 3$, that
\begin{align}\label{eq:3.4.1a}
\bfv_{h_n}\to\bfv\quad\text{ in }L^q(\Omega)\quad(n\to \infty)\,.
\end{align}
Next, we follow \cite[Lemma 6.44]{EP12}, 
where it is proved that for~every~${\bfz\in \Vo(0)}$~and~${n\in \mathbb{N}}$, it holds
\begin{align}
\begin{aligned}
\langle B_{n}^{\textup{I}}\bfv_{h_n},\bfz\rangle_{\WDGn}
&=\big(\big[\boldsymbol{\mathcal{G}}_{h_n,i}^{2\ell}\bfv_{h_n}\big]\bfv_{h_n},\PiDGn\!\bfz\big)
+\tfrac{1}{2}\big((\mathcal{D}\dot{\iota}\nu_h^{2\ell}\bfv_{h_n})\bfv_{h_n},\PiDGn\!\bfz\big)
\\&\quad+\tfrac{1}{4}\langle \jump{ \bfv_{h_n}\otimes\bfn},
\jump{ \bfv_{h_n}}\otimes\jump{ 
	\PiDGn\bfz}\rangle_{\smash{\Gamma_{h_n}^{i}}}\,,
\end{aligned}\label{eq:3.4.1.4}
\end{align}
where $\boldsymbol{\mathcal{G}}_{h_n,i}^{2\ell}\bfv_{h_n}\hspace{-0.1em}
\coloneqq \hspace{-0.1em}\nabla_{h_n}\!\bfv_{h_n}-\smash{\sum_{\gamma\in \Gamma_{h_n}^{i}}{\hspace{-0.5em}\boldsymbol{\mathcal{R}}_{h_n,\gamma}^{2\ell}\bfv_{h_n}}}\hspace{-0.2em}\in\hspace{-0.1em} X^{2\ell}_{h_n}$.
Using \eqref{eq:3.4.1.1}, 
we infer analogously to 
\mbox{Proposition}~\ref{3.3} (or \cite[Lemma 6.45]{EP12}) that
\begin{align}
\begin{aligned}
\begin{aligned}
	\boldsymbol{\mathcal{G}}_{h_n,i}^{2\ell}\bfv_{h_n}
	&\weakto \nabla \bfv&&
	\quad\text{ in }L^{p}(\Omega)&&\quad(n\to\infty)\,,\\
	\mathcal{D}\dot{\iota}\nu_h^{2\ell}\bfv_{h_n}&\weakto 
	\divo\bfv=0&&\quad\text{ in }L^p(\Omega)&&\quad(n\to \infty)\,.
\end{aligned}
\end{aligned}\label{eq:3.4.1.5}
\end{align}
Moreover, the $L^q$-approximation properties of 
$\PiDGn$ (cf.~\cite[Corollary A.8]{kr-phi-ldg}), for every 
$\bfz\in \Vo(0)$ and $q<p^*$, provide 
\begin{align}
\PiDGn\!\bfz\to
\bfz\quad\text{ in }L^q(\Omega)\quad(n\to \infty)\,.\label{eq:3.4.1.6}
\end{align}
Using \eqref{eq:3.4.1a}, \eqref{eq:3.4.1.5} and \eqref{eq:3.4.1.6}, 
for every $\bfz\in \Vo(0)$, we infer that
\begin{align}
\big(\big[\boldsymbol{\mathcal{G}}_{h_n,i}^{2\ell}\bfv_{h_n}\big]\bfv_{h_n},\PiDGn\!\bfz\big)
+\tfrac{1}{2}\big((\mathcal{D}\dot{\iota}\nu_h^{2\ell}\bfv_{h_n})\bfv_{h_n},\PiDGn\bfz\big)
\to([\nabla\bfv]\bfv,\bfz)\quad(n\to\infty)\,.\label{eq:3.4.1.3}
\end{align}
For the remaining term,  we distinguish 
between the cases $p\leq 3$ and $p>3$. 

\textit{Case $p\in (\frac{11}{5},3]$.}
If $p\in (\frac{11}{5},3)$, then, using
$\vert\jump{ \bfz}\vert 
=\vert \jump{\bfz\otimes \bfn}\vert$ on $\Gamma_h$ 
for  ${\bfz\in W^{1,p}(\mathcal{T}_{h_n})}$, 
the inequality $\smash{(\sum_{i=1}^{m} a_i^{2p'})^{\frac{1}{2p'}}\leq (\sum_{i=1}^{m} a_i^{p})^{\frac{1}{p}}}$
for each $(a_i)_{i=1,\dots,m}\subseteq\mathbb{R}_{\ge
0}$ and $m\in \mathbb{N}$,  which just applies for
$p\in (\frac{11}{5},3]$ since $2p'\ge p$, $\|\bfX_{h_n}\|_{2p',\gamma}
\leq c\,h_n^{\frac{1}{p'}-\frac{2}{p}}\|\bfX_{h_n}\|_{p,\gamma}$ for every
$\bfX_{h_n}\in X_h^\ell$ and $\gamma\in \Gamma_{h_n}$,
where $c>0$ is independent~of ${n\in\mathbb{N}}$, (cf.~\cite[Lemma 1.50]{EP12}), 
and $3-\frac{3}{p}+\frac{2}{p'}-\frac{4}{p}>0$ for $p>\frac{9}{5}$,  for~every~$\bfz\in \Vo(0)$, we obtain
\begin{align}
\begin{aligned}
\vert\langle\jump{ \bfv_{h_n}\otimes \bfn},
\jump{\bfv_{h_n}}\otimes\jump{
	\PiDGn\!\bfz}\rangle_{\smash{\Gamma_{h_n}^{i}}}\vert
&\leq h_n^{3-\frac{3}{p}}
\|h_n^{\frac{1}{p}-1}\jump{ \bfv_{h_n}\otimes \bfn}\|_{2p',\Gamma_{h_n}^{i}}^2
\|\PiDGn\bfz\|_{\nabla,p,h_n}\\&
\leq c\,h_n^{3-\frac{3}{p}}h_n^{(\frac{1}{p'}-\frac{2}{p})2}
\|h_n^{\frac{1}{p}-1}\jump{ \bfv_{h_n}\otimes \bfn}
\|_{p,\Gamma_{h_n}^{i}}^{2}\|\nabla \bfz\|_p
\to 0\quad(n\to\infty)\,.
\end{aligned}\hspace*{-5mm}\label{eq:3.4.1.7}
\end{align}

\textit{Case \hspace*{-0.1mm}$p\!\in\! (3,\infty)$.}
\hspace*{-0.1mm}If \hspace*{-0.1mm}$p\!\in\! (3,\infty)$, \hspace*{-0.1mm}then, \hspace*{-0.1mm}using \hspace*{-0.1mm}$h_n\|\bfv_{h_n}\|_{2p',\gamma}^{2p'}
\!\leq\! c\,\|\bfv_{h_n}\|_{2p',K}^{2p'}
\!\leq\!
c\,h^{3\frac{p-3}{p-1}}_n\|\bfv_{h_n}\|_{p,K}^{2p'}$~\hspace*{-0.1mm}for~\hspace*{-0.1mm}${\bfv_{h_n}\!\in\! V_{h_n}^\ell(0)}$, for every ${\bfz\in \Vo(0)}$,~we obtain 
\begin{align}
\begin{aligned}
\vert \big\langle\jump{ \bfv_{h_n}\otimes \bfn},
\jump{\bfv_{h_n}}\otimes\jump{\PiDGn\bfz
}\big\rangle_{\smash{\Gamma_{h_n}^{i}}}\vert&
\leq h_n^{\frac{1}{p'}}\|\jump{ \bfv_{h_n}\otimes \bfn}\|_{2p',\smash{\Gamma_{h_n}^{i}}}^2
\|\PiDGn\bfz\|_{\nabla,p,h_n}
\\
&\leq c\,h_n^{\frac{1}{p'}} (h_n^3\,\# \{\gamma \in \Gamma_{h_n}\})^{\smash{^{\frac{p-3}{p-1}}}}\|\bfv_{h_n}\|_{p}^2\|\nabla\bfz\|_p
\\&\leq c\,h_n^{\frac{1}{p'}}\vert \Omega\vert^{^{\frac{p-3}{p-1}}}\|\bfv_{h_n}\|_{\nabla,p,h_n}^{2}
\|\nabla\bfz\|_p\to 0\quad(n\to\infty)\,,
\end{aligned}\label{eq:3.4.1.8}
\end{align}
Eventually, since also ${\langle B_{n}^{\textup{I}} \bfv_{h_n},\bfv_{h_n}\rangle_{\smash{\WDGn}}
=0=\langle B\bfv,\bfv\rangle_{\Vo(0)}}$ for all $n\in \mathbb{N}$ due to \textit{(ii)}, 
we conclude from \eqref{eq:3.4.1.4}, \eqref{eq:3.4.1.3}, 
\eqref{eq:3.4.1.7} and \eqref{eq:3.4.1.8} for every $\bfz\in \Vo(0)$ that
\begin{align*}
\lim_{n\to\infty}{\langle B_{n}^{\textup{I}}\bfv_{h_n},\bfv_{h_n}-\bfz\rangle_{\smash{\WDGn}}}
=\langle B\bfv,\bfv-\bfz\rangle_{\Vo(0)}\,.
\end{align*}
\end{proof}

\begin{proposition}[Convective term II, (cf.~\cite{kr-pnse-ldg-1})]\label{3.4.2}
Let  $d\in \{2,3\}$, $\ell \in \setN$,  and $p>\frac{3d+2}{d+2}$.
Moreover, for every $n\in \mathbb{N}$, let $B_{n}^{\textup{II}} \colon\WDGn\to \WDGn^*$, for every $\bfw_{h_n},\bfz_{h_n}\in \WDGn$, be defined via
\begin{align*}
\big\langle B_{n}^{\textup{II}} \bfw_{h_n},\bfz_{h_n}\big\rangle_{\smash{\WDGn}}
\coloneqq\tfrac{1}{2}\big( \PiDGn\bfz_{h_n}\otimes \bfw_{h_n},\Ghnk\bfw_{h_n}\big)-\tfrac{1}{2}\big( \bfw_{h_n}\otimes \bfw_{h_n},\Ghnk\bfz_{h_n}\big)\,.
\end{align*}
Then, it holds:
\begin{itemize}
\item[(i)] For every $n\in \mathbb{N}$, the operator $B_{n}^{\textup{II}} \colon\WDGn\to \WDGn^*$ is well-defined, 
bounded~and~continuous. In addition, for every
$p$ there exist $q\in \left(1,\infty\right)$
such that for every $\varepsilon>0$, there
exists $c_\vep>0$, depending only on
$\vep^{-1}$, $\ell$, $p$
and $\omega_0$,  such that for every $n\in
\mathbb{N}$ and $\bfw_{h_n},\bfz_{h_n}\in \WDGn$, it holds
\begin{align*}
\vert \langle B_{n}^{\textup{II}}\bfw_{h_n},\bfz_{h_n}\rangle_{\smash{\WDGn}}\vert
\leq \varepsilon\,\|\bfw_{h_n}\|_{\nabla,p,h_n}^p
+c_{\varepsilon}\big(1+\|\bfw_{h_n}\|_2^q+\|\bfz_{h_n}\|_2^q+\|\bfz_{h_n}\|_{\nabla,p,h_n}^p\big).
\end{align*}
\item[(ii)] For every $n\in \mathbb{N}$ and $\bfv_{h_n}\in \Vhnk$, it
holds $\langle B_{n}^{\textup{II}} \bfv_{h_n},\bfv_{h_n}\rangle_{\smash{\WDGn}}=0$.
\item[(iii)] The operators $B_{n}^{\textup{II}} \colon\WDGn\to \WDGn^*$, $n\in\mathbb{N}$, are 
non-conforming pseudo-monotone with respect to $(\Vhnk(0))_{n\in \mathbb{N}}$ 
and $B\colon\Vo(0)\to \Vo(0)^*$ from Proposition \ref{3.4.1}.
\end{itemize}
\end{proposition}

\begin{proof}
Again, we will only prove the case $d=3$.

\textit{ad (i).}
For every $n\in \mathbb{N}$ and $\bfw_{h_n},\bfz_{h_n}\in \WDGn$, it holds
\begin{align}
\begin{aligned}
&\vert( \PiDG\bfz_{h_n}\otimes \bfw_{h_n},\Ghnk\bfw_{h_n})\vert+\vert( \bfw_{h_n}\otimes \bfw_{h_n},\Ghnk\bfz_{h_n})\vert
\\
&\leq c\,\|\bfz_{h_n}\|_{2p'}\|\bfw_{h_n}\|_{2p'}
\|\bfw_{h_n}\|_{\nabla,p,h_n}+c\,\|\bfw_{h_n}\|_{2p'}^2
\|\bfz_{h_n}\|_{\nabla,p,h_n}\,.
\end{aligned}\label{eq:3.4.2.c}
\end{align}
Next, we distinguish between the cases $p\ge 3$ and $p<3$:

\textit{Case $p\in [3,\infty)$.}
For $p\ge 3$ we have $2p'\ge p$  and $p-1\ge 2$, which implies
$\|\bfw_{h_n}\|_{2p'}\leq c\,\|\bfw_{h_n}\|_{\nabla,p,h_n}$ for all $\bfw_{h_n}\in \WDGn$ and $ n\in \mathbb{N}$, where $c>0$ does not depend on $n\in \mathbb{N}$, 
due to \cite[Proposition 2.6]{kr-pnse-ldg-3}, and  
$a^2\leq c\,(1+a^{p-1})$ for all $a\ge 0$. Using this and the 
$\varepsilon$-Young inequality \eqref{ineq:young}~with~$\psi=\vert \cdot\vert^p$, we obtain  from \eqref{eq:3.4.2.c} for every $\varepsilon>0$, 
$n\in \mathbb{N}$ and $\bfv_{h_n},\bfz_{h_n}\in \WDGn$ that
\begin{align}
\begin{aligned}
\vert( \PiDG\bfz_{h_n}\otimes \bfw_{h_n},\Ghnk\bfw_{h_n})\vert+	\vert( \bfw_{h_n}\otimes\bfw_{h_n},\Ghnk\bfz_{h_n})\vert
&\leq c\,\|\bfw_{h_n}\|_{\nabla,p,h_n}^2\|\bfz_{h_n}\|_{\nabla,p,h_n}
\\&\leq c\,\big(1+\|\bfw_{h_n}\|_{\nabla,p,h_n}^{p-1}\big)
\|\bfz_{h_n}\|_{\nabla,p,h_n}\\&\leq \varepsilon\,c\,\big(1+\|\bfw_{h_n}\|_{\nabla,p,h_n}^{p}\big)+
c_\varepsilon\,\|\bfz_{h_n}\|_{\nabla,p,h_n}^{p}\,.
\end{aligned}\label{eq:3.4.2.d}
\end{align}

\textit{Case $p\!\in\! (\frac{11}{5},3)$.}
If $p\!\in\! (\frac{11}{5},3)$, then, by interpolation with 
$\frac{1}{\rho}\!=\!\frac{1-\theta}{p^*}+\frac{\theta}{2}$, 
where $\rho\!=\!p\frac{5}{3}$,~$\theta\!=\!\frac{2}{5}$~and~${p^*\!=\!\frac{3p}{3-p}}$
and \cite[Propositon 2.6]{kr-pnse-ldg-3}, 
for every $n\in \mathbb{N}$ and $\bfw_{h_n}\in \WDGn$, we obtain 
\begin{align}
\|\bfw_{h_n}\|_{\rho}\leq \|\bfw_{h_n}\|_2^\frac{2}{5}\|\bfw_{h_n}\|_{p^*}^{\frac{3}{5}}
\leq  c\,\|\bfw_{h_n}\|_2^\frac{2}{5}\|\bfw_{h_n}\|_{\nabla,p,h_n}^{\frac{3}{5}}\,.\label{eq:3.4.2.e}
\end{align}
Since $\rho\ge 2p'$, from  \eqref{eq:3.4.2.e} 
in \eqref{eq:3.4.2.c},  for every $n\in \mathbb{N}$ and $\bfw_{h_n},\bfz_{h_n}\in \WDGn$, we further conclude that
\begin{align}
\begin{aligned}
&\vert( \PiDG\bfz_{h_n}\otimes
\bfw_{h_n},\Ghnk\bfw_{h_n})\vert+ \vert(\bfw_{h_n}\otimes
\bfw_{h_n},\Ghnk\bfz_{h_n})\vert \\&\leq
c\,\|\bfw_{h_n}\|_2^\frac{2}{5}\|\bfw_{h_n}\|_{\nabla,p,h_n}^{\frac{8}{5}}
\|\bfz_{h_n}\|_2^\frac{2}{5}\|\bfz_{h_n}\|_{\nabla,p,h_n}^{\frac{3}{5}}
\\&\quad +c\,\|\bfw_{h_n}\|_2^\frac{4}{5}\|\bfw_{h_n}\|_{\nabla,p,h_n}^{\frac{6}{5}}
\|\bfz_{h_n}\|_{\nabla,p,h_n} \\&\leq
c\,\|\bfw_{h_n}\|_2^\frac{2p}{5p-3}\|\bfw_{h_n}\|_{\nabla,p,h_n}^{\frac{8p}{5p-3}}
\|\bfz_{h_n}\|_2^{\frac{2p}{5p-3}}\\&\quad+
c\,\|\bfw_{h_n}\|_2^\frac{4p'}{5}\|\bfw_{h_n}\|_{\nabla,p,h_n}^{\frac{6p'}{5}}+c\,\|\bfz_{h_n}\|_{\nabla,p,h_n}^{p}\,,\label{eq:3.4.2.f}
\end{aligned}
\end{align}
where we used the $\varepsilon$-Young inequality \eqref{ineq:young} with $\psi=\vert \cdot\vert^{\smash{\frac{5p}{3}}}$  for the first term and with $\psi=\vert \cdot\vert^{p}$~for~the second term. 
Since $\smash{s\coloneqq \frac{8p}{5p-3}<p}$  and $\smash{\tilde{s}\coloneqq \frac{6p'}{5}<p}$ for $p>\frac{11}{5}$, 
the $\varepsilon$-Young inequality \eqref{ineq:young} with $\psi=\vert \cdot\vert^{r}$, where $\smash{{r}=\frac{p}{s}}$, and with $\psi=\vert \cdot\vert^{\tilde{r}}$, where $\smash{\tilde{r}=\frac{p}{\tilde{s}}}$ yields  for every $\varepsilon>0$, 
$n\in \mathbb{N}$ and $\bfw_{h_n},\bfz_{h_n}\in \WDGn$ 
\begin{align}
\begin{aligned}
&\vert( \PiDG\bfz_{h_n}\otimes \bfw_{h_n},\Ghnk\bfw_{h_n})\vert+	\vert(\bfw_{h_n}\otimes \bfw_{h_n},\Ghnk\bfz_{h_n})\vert
\\
& \leq \varepsilon\,c\,\|\bfw_{h_n}\|_{\nabla,p,h_n}^{p}
+c_{\varepsilon}\big(1+\|\bfw_{h_n}\|_2^q+\|\bfz_{h_n}\|_2^q
+\|\bfz_{h_n}\|_{\nabla ,p,h_n}^{p}\big)\,,
\end{aligned}\label{eq:3.4.2.g}
\end{align}
where
$\smash{q\coloneqq \max\{2r'\frac{2p}{5p-3},\tilde{r}'\frac{4p'}{5}\}}>0$.

Thus, we proved that for every
$\varepsilon>0$, there exists $c_\vep>0$ such that for every $n\in \mathbb{N}$ 
and $\bfw_{h_n},\bfz_{h_n}\in \WDGn$, it holds 
\begin{align}
\vert\langle B_{n}^{\textup{II}} \bfw_{h_n},\bfz_{h_n}\rangle_{\WDGn}\vert
& \leq \varepsilon\,c\,\|\bfw_{h_n}\|_{\nabla,p,h_n}^{p}
+c_{\varepsilon}\big(1+\|\bfw_{h_n}\|_2^q+\|\bfz_{h_n}\|_2^q
+\|\bfz_{h_n}\|_{\nabla ,p,h_n}^{p}\big)\,.\label{eq:3.4.2.j}
\end{align}

\textit{ad (ii).} Follows right from the definition.

\textit{ad (iii).} See \cite[Lemma 5.3]{kr-pnse-ldg-1}.
\end{proof}

Let us summarize  our setup for the treatment of problem
\eqref{eq:p-NS}.

\begin{asum}\label{asumex-p-NS}
Let $\Omega\subseteq \mathbb{R}^d$, $d\in \{2,3\}$, be  a bounded polygonal
Lipschitz domain, $I\coloneqq \left(0,T\right)$, $T<\infty$,
and $\ell \in \setN$ be given. In addition, we make the following assumptions:
\begin{itemize}
\item[(i)] $\SSS$ satisfies Assumption \ref{assum:extra_stress} for $\delta\ge 0$ and $p>\frac{3d+2}{d+2}$.
\item[(ii)] For $\bfv_0\hspace*{-0.1em}\in\hspace*{-0.1em} \Ho(0)$ let $\bfv_n^0\hspace*{-0.1em}\in\hspace*{-0.1em} V_{h_n}^\ell(0)$, $n\in
\mathbb{N}$ be such that $\bfv_n^0\!\to\! \bfv_0$ in $L^2(\Omega)$
$(n\!\to\! \infty)$, ${\sup_{n\in 
	\mathbb{N}}{\|\bfv_n^0\|_2}\hspace*{-0.1em}\leq\hspace*{-0.1em} \|\bfv_0\|_2}$. 
\item[(iii)] For $\mathbf{g}\in L^{p'}(I, L^{p'}(\Omega))$ and $\mathbf{G}\in
L^{p'}(I,L^{p'}(\Omega))$, we define  $f\in L^{p'}(I,\Vo(0)^*)$,~for~every~$\bfz\in L^p(I,\Vo(0))$  via
\begin{align*}
\int_I{\langle f(t),\bfz(t)\rangle_{\Vo(0)}\,\mathrm{d}t}\coloneqq \int_{I} \int_\Omega{\mathbf{g}\cdot \bfz+\mathbf{G}:\nabla \bfz\,\mathrm{d}t\,\mathrm{d}x}
\end{align*}
and  $f_n\in L^{p'}(I,\WDGn^*)$, $n\in \mathbb{N}$, for every  $\bfz_n\in L^p(I,\WDGn)$ via
\begin{align*}
\int_I{\langle f_n(t),\bfz_n(t)\rangle_{\WDGn}\,\mathrm{d}t}\coloneqq\int_{I}\int_\Omega{\mathbf{g}\cdot \bfz_n+\mathbf{G}:\Ghnk \bfz_n\,\mathrm{d}t\,\mathrm{d}x}\,.
\end{align*}
\item[(iv)] We set $A_n\coloneqq S_n+B_n\colon \WDGn\to \WDGn^*$,
$n\in \mathbb{N}$, where ${S_n\colon \WDGn\to \WDGn^*}$ is either
$S_n^{\textup{\textsf{\textbf{\tiny LDG}}}}\colon \WDGn\to
\WDGn^*$ or $S_n^{\textup{\textsf{\textbf{\tiny SIP}}}}\colon
\WDGn\to \WDGn^*$  and $B_n\colon \WDGn\to \WDGn^*$ is either $B_n^{\textup{I}}\colon \WDGn\to \WDGn^*$ or $B_n^{\textup{II}}\colon \WDGn\to \WDGn^*$.
\end{itemize}
\end{asum}
Moreover, we denote by $e\coloneqq (\textup{id}_{\Vo(0)})^*R_{\Ho(0)}\colon \Vo(0)\to \Vo(0)^*$ the
canonical embedding with respect to the evolution triple\enlargethispage{8mm}
$(\Vo(0),\Ho(0),\textup{id})$.  The following lemma shows that the assumptions on the right-hand side in Assumption \ref{asum} are fulfilled in our setting.

\begin{lemma}\label{lem:rhs}
	Let Assumption \ref{asumex-p-NS} be satisfied. Then, the sequence $f_n\in 
	L^{p'}(I,\WDGn^*)$ satisfies (\hyperlink{BN.1}{BN.1}) and  (\hyperlink{BN.2}{BN.2}) with respect to $f\in L^{p'}(I,\Vo(0)^*)$, i.e., $\sup_{n\in\mathbb{N}}{\| f_n\|_{L^p(I,\WDGn^*)}}<\infty$ and
	for every sequence $\bfv_{h_{m_n}}\in L^p(I,V_{h_{m_n}}^\ell\!(0))\cap L^\infty(I,L^2(\Omega))$, where $(m_n)_{n\in\setN}\subseteq\setN$ with $ m_n\to \infty $ $(n\to\infty)$,  from 
	\begin{align}\label{lem:rhs.1}
		\begin{aligned}
		\sup_{n\in\mathbb{N}}{\| \bfv_{h_{m_n}}\|_{L^p(I,\WDGn)\cap L^\infty(I,L^2(\Omega))}}<\infty\,,\\\qquad\bfv_{h_{m_n}}\overset{\ast}{\rightharpoondown }\bfv\quad\text{ in }L^\infty(I,L^2(\Omega))\quad (n\to\infty)\,,
		\end{aligned}
	\end{align}
	it follows that 
	\begin{align}\label{lem:rhs.2}
			\int_I{\langle f_{m_n}(t),\bfv_{h_{m_n}}(t)\rangle_{\WDGn}\,\mathrm{d}t}\to \int_I{\langle f(t),\bfv(t)\rangle_{\Vo(0)}\,\mathrm{d}t}\quad (n\to \infty)\,.
	\end{align}
\end{lemma}

\begin{proof}
		\textit{ad (\hyperlink{BN.1}{BN.1}).} Let $\bfw_{h_n}\in \WDGn$ be arbitrary. Then, for almost every $t\in I$, we have that 
		\begin{align}
			\big\langle f_n(t),\bfw_{h_n}\big\rangle_{\smash{W^{1,p}(\mathcal{T}_{h_n})}}\leq \|\bfg(t)\|_{p'}\|\bfw_{h_n}\|_p+\|\bfG(t)\|_{p'}\big\|\Ghnk\bfw_{h_n}\big\|_p\,.\label{lem:rhs.3}
		\end{align}
		Appealing to \cite[Lemma A.9]{dkrt-ldg}, it holds $\|\bfw_{h_n}\|_p\leq c\,\|\bfw_{h_n}\|_{\nabla,p,h_n}$, so that~from~\eqref{lem:rhs.3}, using the norm equivalence \eqref{eq:eqiv0}, it follows that ${\|f_n(t)\|_{\smash{\WDGn^*}}\leq c\,\|\bfg(t)\|_{p'}+c\,\|\bfG(t)\|_{p'}}$ for almost every $t\in I$. Therefore, we have that $\sup_{n\in\mathbb{N}}{\|f_n\|_{L^{p'}(I,\smash{\WDGn^*})}}\leq c\,\|\bfg\|_{L^{p'}(I,L^{p'}(\Omega))}+c\,\|\bfG\|_{L^{p'}(I,L^{p'}(\Omega))}$.

		\textit{ad (\hyperlink{BN.2}{BN.2}).}	Let $\bfv_{h_{m_n}}\in L^p(I,V_{h_{m_n}}^\ell\!(0))\cap L^\infty(I,L^2(\Omega))$, $n\in \setN$, where $(m_n)_{n\in\setN}\subseteq\setN$ with $ m_n\to \infty $ $(n\to\infty)$, be a sequence satisfying
		\eqref{lem:rhs.1}. Then, using Proposition \ref{3.3}, in particular, \eqref{eq:3.3.3}, we find that $\bfv\in L^p(I,\Vo(0))\cap L^\infty(I,\Ho(0))$ and 
		\begin{align}
				\boldsymbol{\mathcal{G}}_{h_{m_n}}^\ell\!\!\bfv_{h_{m_n}}\weakto\nabla\bfv\quad\text{ in }L^p(I,L^p(\Omega))\quad(n\to \infty)\,.\label{lem:rhs.4}
		\end{align}
		In addition, resorting to the discrete Poincar\'e inequality \cite[Lemma A.9]{dkrt-ldg}, we obtain
		\begin{align}
				\bfv_{h_{m_n}}\weakto\bfv\quad\text{ in }L^p(I,L^p(\Omega))\quad(n\to \infty)\,.\label{lem:rhs.5}
		\end{align}
		Eventually, from \eqref{lem:rhs.4} and \eqref{lem:rhs.5}, we conclude that \eqref{lem:rhs.2} applies.
\end{proof}

Thus, the quasi non-conforming
Rothe--Galerkin~scheme~in~this~setup~reads:

\begin{alg}\label{eq:p-NSnon}
Let Assumption \ref{asumex-p-NS} be satisfied. For given
$K,n\in \mathbb{N}$, the sequence of iterates
${\bfv_n^k\hspace*{-0.1em}=\hspace*{-0.1em}\bfv_{h_n}^k\hspace*{-0.1em}\in\hspace*{-0.1em} \Vhnk(0)}$,
$k\hspace*{-0.1em}=\hspace*{-0.1em}0,\dots,K$ is given via the implicit Rothe--Galerkin~scheme~for~${\tau\hspace*{-0.1em}=\hspace*{-0.1em}\frac{T}{K}}$~and~${k\hspace*{-0.1em}=\hspace*{-0.1em}1,\dots,K}$
\begin{align}
(d_\tau \bfv_n^k,\bfz_{h_n})_{\Ho(0)}+\langle A_n\bfv_n^k, \bfz_{h_n}\rangle_{\WDGn}=\langle
[f_n]^\tau_k, \bfz_{h_n}\rangle_{\WDGn}\quad\text{ for all
}\bfz_{h_n}\in \Vhnk(0)\,.\label{eq:p-NSnon.1}
\end{align}
\end{alg}

By means of Proposition~\ref{5.1}, Proposition~\ref{apriori}, Theorem~\ref{5.17} and the observations already made in Proposition~\ref{3.3}, Proposition~\ref{ldg}, Proposition \ref{SIP}, Proposition \ref{3.4.1} and Proposition \ref{3.4.2}, we can immediately conclude the following results: 

\begin{theorem}[Well-posedness, stability, and weak convergence of 
\eqref{eq:p-NSnon.1}]\label{thm:appl-p-NS}
Let Assumption \ref{asumex-p-NS} be satisfied. Then, it holds:
\begin{itemize}
\item[(i)] (Well-posedness). For every $K,n\in \mathbb{N}$
there exist iterates $(\bfv_n^k)_{k=0,\dots,K}\subseteq \Vhnk(0)$,
solving \eqref{eq:p-NSnon.1}, without any restrictions on the
step-size. 
\item[(ii)] (Stability). The
corresponding piece-wise constant
interpolants   $\overline{\bfv}_n^\tau\in
\mathcal{P}_0(\mathcal{I}_\tau,\Vhnk(0))$, $K,n\in\mathbb{N}$ with $\tau=\frac{T}{K}$, satisfy
\begin{align*}
\sup_{n\in \mathbb{N},\tau>0}{\|\overline{\bfv}_n^\tau\|_{L^p(I,\WDGn)\cap L^\infty(I,L^2(\Omega))} }<\infty\,.
\end{align*}
\item[(iii)] (Weak convergence). If
$(\overline{\bfv}_n)_{n\in\mathbb{N}}\coloneqq (\overline{\bfv}_{m_n}^{\tau_n})_{n\in\mathbb{N}}$,
where $\tau_n=\frac{T}{K_n}$ and
${K_n,m_n\to \infty }$~${(n\to\infty)}$, is an arbitrary diagonal
sequence of the piece-wise constant interpolants
${\overline{\bfv}_n^\tau\!\in\!
\mathcal{P}_0(\mathcal{I}_\tau,\Vhnk(0))}$,~${K,n\!\in\!\mathbb{N}}$~with
${\tau=\frac{T}{K}}$, then there exists a not relabelled
subsequence and 
${\overline{\bfv}\in L^p(I,\Vo(0))\cap L^\infty(I,\Ho(0))}$~such~that
\begin{align*}
\begin{aligned}
	\overline{\bfv}_n&\;\;\rightharpoonup\;\;\overline{\bfv}&&\quad\text{
		in }L^p(I,L^p(\Omega))
	&&\quad(n\to\infty)\,,\\
	\Dhnk\!\overline{\bfv}_n&\;\;\rightharpoonup\;\;\bfD\overline{\bfv}&&\quad\text{
		in }L^p(I,L^p(\Omega))
	&&\quad(n\to\infty)\,,\\
	\overline{\bfv}_n&\;\;\overset{\ast}{\rightharpoondown}\;\;\overline{\bfv}&&\quad\text{
		in }L^\infty(I,L^2(\Omega))
	&&\quad(n\to\infty)\,.
\end{aligned}
\end{align*}
Furthermore, it follows that $\overline{\bfv}\in
{W}_e^{1,p,p}(I,V,H)\cap L^\infty(I,H)$ satisfies
$\overline{\bfv}(0)=\mathbf{v}_0$ in $\Ho(0)$ and for every
$\bfz\in L^p(I,\Vo(0))$, it holds
\begin{align*}
\int_I{\bigg\langle\frac{d_e\overline{\bfv}}{dt}(t),\bfz(t)\bigg\rangle_{\Vo(0)}\,\mathrm{d}t}+\int_I{\langle
	A(t)(\overline{\bfv}(t)),\bfz(t)\rangle_{\Vo(0)}\,\mathrm{d}t}=\int_I{\langle f(t),\bfz(t)\rangle_{\Vo(0)}\,\mathrm{d}t}\,. 
\end{align*}
\end{itemize}
\end{theorem}

\begin{remark}
If the solution of \eqref{eq:p-NS} is unique, e.g., cf.~\cite{mnrr}, if
$p\ge \frac{d+2}{2}$ (cf. \cite{BKP19} for $p\ge \frac{3d+2}{d+2}$)~and $\SSS\colon \mathbb{R}^{d\times d}\to \mathbb{R}^{d\times d}_{\textup{sym}}$ is strongly monotone, i.e., there exists a constant $\mu_0>0$ such that for every $\bfA,\bfB\in \mathbb{R}^{d\times d}$, it holds 
\begin{align*}
(\SSS(\bfA)-\SSS(\bfB)):(\bfA-\bfB)\ge \mu_0\,\abs{\bfA-\bfB}^2\,,
\end{align*}
then one does not have to pass to a subsequence in Theorem~\ref{thm:appl-p-NS}~(iii).
\end{remark}

\section{Numerical experiments}\label{sec:8}

In this section, we apply the scheme \eqref{eq:p-NSnon.1} (cf.~Algorithm \ref{eq:p-NSnon}) for the case  $S_n\coloneqq S_n^{\tiny \textup{\textsf{\textbf{LDG}}}}$ and $B_n\coloneqq  B_n^{\textup{II}}$, to solve
numerically the system~\eqref{eq:1.1}~with $\SSS\colon \mathbb{R}^{d\times d}\to\mathbb{R}^{d\times d}_{\sym}$, for every $\bfA\in\mathbb{R}^{d\times d}$ defined by  $${\SSS(\bfA) \coloneqq (\delta+\vert \bfA^{\textup{sym}}\vert)^{p-2}\bfA^{\textup{sym}}}\,,$$
where $\delta\coloneqq 1\textrm{e}{-}4$ and $ p\ge \frac{3d+2}{d+2}$. 

The scheme \eqref{eq:p-NSnon.1} (cf.~Algorithm \ref{eq:p-NSnon}) with  $S_n\coloneqq S_n^{\tiny \textup{\textsf{\textbf{LDG}}}}$ and $B_n\coloneqq  B_n^{\textup{II}}$, cf. \cite{kr-pnse-ldg-1,kr-pnse-ldg-2,kr-pnse-ldg-3}, for every $k=1,\dots, K$, $K\in \mathbb{N}$, and $\tau \coloneqq \frac{T}{K}$, $0<T<\infty$, can equivalently be re-written as a saddle point problem seeking for iterates
$(\bfv_{h_n}^{k}, q_{h_n}^{k})^\top\in \Vhnk\times \Qhnkco$
such that for every $(\bfz_{h_n},z_{h_n})^\top\in \Vhnk\times \Qhnkco$, it holds
\begin{align}\label{eq:primal1}
	\begin{aligned}
	&\hskp{d_\tau \bfv_{h_n}^{k} }{\bfz_{h_n}}+
	\bighskp{\SSS(\Dhnk \!\bfv_{h_n}^{k})}{\Dhnk
		\bfz_h}+\alpha \big\langle\SSS_{\smash{\avg{\abs{\Dhnk\!\bfv_h^{k}}}}}(h_n^{-1} \jump{\bfv_{h_n}^{k}\otimes
		\bfn}), \jump{\bfz_{h_n} \otimes
                \bfn}\big\rangle_{\Gamma_{h_n}} \hspace*{-5mm}
	\\&\quad-\tfrac{1}{2}\bighskp{\bfv_{h_n}^{k}\otimes \bfv_{h_n}^{k}}{\Dhnk 	\bfz_{h_n}}+\tfrac{1}{2}\bighskp{\bfz_{h_n}\otimes \bfv_{h_n}^{k}}{\Ghnk 	\bfv_{h_n}^{k}}-\bighskp{q_{h_n}^{k}}{\Divhnk\bfz_{h_n}}
	\\
	&=     \bighskp{[\bfg]_\tau^{k}}{\bfz_{h_n}}-\bighskp{[\bfG]_\tau^{k}}{\Ghnk\bfz_{h_n}}\,,\\
	\bighskp{\Divhk \bfv_{h_n}^{k}}{z_{h_n}}&=0\,.
	\end{aligned}
\end{align}   
For an simple implementation, for every $k=1,\dots ,K$, we replace $[\bfg]_\tau^{k}\in L^{p'}(\Omega)$ and $[\bfG]_\tau^{k}\in L^{p'}(\Omega)$ by $\bfg(t_k)\in L^{p'}(\Omega)$ and $\bfG(t_k)\in L^{p'}(\Omega)$, respectively.
Since the manufactured solutions
below,~cf.~\eqref{eq:manufactured_solutions},~are smooth with respect
to the time variable and, thus, the resulting
right-hand~side,~given~via~${\bfg\hspace*{-0.15em}\in\hspace*{-0.15em}
  L^{p'}(I,L^{p'}(\Omega))}$ and $\bfG\in L^{p'}(I,L^{p'}(\Omega))$,
is at least once continuously differentiable with respect to the time variable, i.e., satisfy at least $\bfg\in W^{1,p'}(I,L^{p'}(\Omega))$ and $\bfG\in W^{1,p'}(I,L^{p'}(\Omega))$,~we~have~that
\begin{align*}
		\sum_{k=1}^K{\int_{I_k^\tau}{\|
  [\bfg]_\tau^{k}-\bfg(t_k)\|_{p'}+\|
  [\bfG]_\tau^{k}-\bfG(t_k)\|_{p'}}}\, \textup{d}t\leq  \tau \, \big(\|\partial_t \bfg\|_{L^{p'}(I,L^{p'}(\Omega))}+\|\partial_t \bfG\|_{L^{p'}(I,L^{p'}(\Omega))}\big)\,.
\end{align*}
Thus, replacing $[\bfg]_\tau^{k}\in  L^{p'}(\Omega)$ and
$[\bfG]_\tau^{k}\in  L^{p'}(\Omega)$ by $\bfg(t_k)\in
L^{p'}(\Omega)$ and ${\bfG(t_k)\in L^{p'}(\Omega)}$, respectively,  $k= 1,\dots ,K$, does not effect the experimental results.

Then,  we approximate the iterates  $(\bfv_{h_n}^{k},q_{h_n}^{k})^\top\in \Vhnk(0)\times \Qo_{h_n}^\ell$, $k=1,\dots,K$, $K\in \mathbb{N}$, solving  \eqref{eq:p-NSnon.1}, deploying the Newton solver from
\mbox{\textsf{PETSc}} (version 3.17.3), cf.~\cite{LW10}, with an
absolute tolerance~of $\tau_{abs}\!=\! 1\textrm{e}{-}8$ and a relative
tolerance of $\tau_{rel}\!=\!1\textrm{e}{-}10$. The linear system
emerging in each Newton step is solved using a sparse direct solver
from \textsf{MUMPS} (version~5.5.0),~cf.~\cite{mumps}. 
We always choose the fixed parameter
$\alpha=2.5$. This choice is in accordance with the choice in
\mbox{\cite[Table~1]{dkrt-ldg}}. In the implementation, the uniqueness
of the discrete pressure is enforced via a zero mean condition.
All experiments were carried out using the finite element software package~\mbox{\textsf{FEniCS}} (version 2019.1.0), cf.~\cite{LW10}. \enlargethispage{6mm}

For our numerical experiments, we choose $\Omega\coloneqq (-1,1)^2$, $I\coloneqq(0,T)$, where $T=0.1$, and~linear~elements, i.e., $ \ell= 1 $. We  choose $\smash{\bfg \in L^{p'}(I,L^{p'}(\Omega))}$, $\smash{\bfG\in L^{p'}(I,L^{p'}(\Omega))}$ and $\smash{\bfv_0\in \Ho(0)}$ such that $\bfv\colon \overline{I\times\Omega}\to \mathbb{R}^2$ and $q\colon I\times\Omega \to \mathbb{R}$, for every $(t,x)^\top\coloneqq (t,x_1,x_2)^\top\in I\times\Omega$ defined by\footnote{The exact solutions
do not satisfy the homogeneous boundary conditions
\eqref{eq:1.1}$_3$. However, the error is mainly concentrated
around the singularity in the origin and hence the small
inconsistency with our theoretical~setup~does not have any influence
in the results.}
\begin{align}\label{eq:manufactured_solutions}
	\begin{aligned}
		\bfv(t,x)&\coloneqq t\,\vert x\vert^\beta(x_2,-x_1)^\top\,,\\  q(t,x)&\coloneqq t^2\,(\vert x\vert^{\gamma}-\langle\,\vert \!\cdot\!\vert^{\gamma}\,\rangle_\Omega)\,,
	\end{aligned}
\end{align}
are solutions of \eqref{eq:1.1}. For
$\rho\hspace*{-0.17em}\in\hspace*{-0.17em} \{0.05,0.1,0.2\}$, we choose
$\beta\hspace*{-0.17em}=\hspace*{-0.17em}\smash{\frac{2(\rho-1)}{p}}$,~which~yields that
${\vert \nabla^{\rho}\bfF(\bfD\bfv)\vert\hspace*{-0.17em} \in \hspace*{-0.17em} L^2(I\hspace*{-0.17em}\times\hspace*{-0.17em}\Omega)}$, and
$\gamma=\rho-\smash{\frac{2}{p'}}$, which
yields~that~$\vert \nabla^{\rho} q\vert\hspace*{-0.15em}\in\hspace*{-0.15em} L^{p'}(I\times\Omega)$. 
Here, $\bfF(\bfA)\coloneqq (\delta+\vert
\bfA^{\textup{sym}}\vert)^{\smash{\frac{p-2}2}}\bfA^{\textup{sym}}$ for all $\bfA\in \mathbb{R}^{2\times 2}$.  
This choice let us expect convergence with a convergence rate $\rho \frac{p'}{2}$ for the error quantities computed below.

We \hspace{-0.1mm}construct \hspace{-0.1mm}a \hspace{-0.1mm}initial \hspace{-0.1mm}triangulation \hspace{-0.1mm}$\mathcal
T_{h_0}$, \hspace{-0.1mm}where \hspace{-0.1mm}$h_0\hspace{-0.2em}=\hspace{-0.2em}\smash{\frac{1}{\sqrt{2}}}$, \hspace{-0.1mm}by \hspace{-0.1mm}subdividing~\hspace{-0.1mm}a~\hspace{-0.1mm}\mbox{rectangular} cartesian grid~into regular triangles with different orientations.  Finer triangulations~$\mathcal T_{h_n}$, where $h_{n+1}=\frac{h_n}{2}$, $n=1,\dots,5$, are 
obtained by
regular subdivision of the previous grid: Each \mbox{triangle} is subdivided
into four equal triangles by connecting the midpoints of the edges, i.e., applying the red-refinement rule, cf. \cite[Definition~4.8~(i)]{Ba16}.
Furthermore, we use the
time step-sizes
$\tau_n\coloneqq 2^{-n-2} $, i.e., $K_n\coloneqq 2^{n+2}$, $n\in \mathbb{N}$. 

As an approximation of the initial condition, for every $n\in \mathbb{N}$, we employed $\bfv_{h_n}^0\coloneqq \PiDG\bfv(0)\in \Vhk(0)$, which satisfies both $\bfv_{h_n}^0\to \bfv(0)$ in $\Ho(0)$  $(n\to \infty)$ and $\sup_{n\in\mathbb{N}}{\|\bfv_{h_n}^0\|_2}\leq  \|
\bfv(0)\|_2$.

Then, for the resulting series of triangulations $\mathcal T_{h_n}$, $n=1,\dots,5$, we apply~the~above Newton scheme to compute the corresponding numerical solutions $(\bfv_{h_n}^{k},q_{h_n}^{k})^\top\in V_{h_n}^\ell\times \Qhnkco$, $n=1,\dots,5$, $i=1,\dots,K_n$,
and the parabolic error quantities 
\begin{align}\label{eq:error_quantities}
	\begin{aligned}
e_{\bfF}^n&\coloneqq\bigg(\sum_{k=0}^{K_n}{\tau_n\|\bfF(\bD\bfv(t_k))-\bfF(\Dhnk\!\bfv_{h_n}^{k})\|^2_2}\bigg)^{\frac{1}{2}}\,,\\
e_{\jump{}}^n&\coloneqq\bigg(\sum_{k=0}^{K_n}{\tau_n\int_{\Gamma_h}{\varphi_{\avg{\abs{\Dhnk\!\bfv_{h_n}^{k}}}}(\bfv(t_k)-\bfv_{h_n}^{k})}\,\mathrm{d}s}\bigg)^{\frac{1}{2}}\,,\\
e_{\bfF^*}^n&\coloneqq\bigg(\sum_{k=0}^{K_n}{\tau_n\|\bfF^*(\SSS(\bD\bfv(t_k)))-\bfF^*(\PiDG\SSS(\Dhnk\!\bfv_{h_n}^{k}))\|^2_2}\bigg)^{\frac{1}{2}}\,,\\
e_{L^2}^n&\coloneqq\max_{0\leq k \leq K_n}{\|\bfv(t_k)-\bfv_{h_n}^{k}\|_2}\,,\\
e_{q}^n&\coloneqq\bigg(\sum_{k=0}^{K_n}{\tau_n\|q(t_k)-q_{h_n}^{k}\|^{p'}_{p'}}\bigg)^{\frac{1}{p'}}\,,
	\end{aligned}
\end{align}
where $\bfF^*(\bfA)\coloneqq (\delta^{p-1}+\vert
\bfA^{\textup{sym}}\vert)^{\smash{\frac{p'-2}2}}\bfA^{\textup{sym}}$ for all $\bfA\in \mathbb{R}^{d\times d}$.   
As an estimation of the convergence rates,  the experimental order of convergence (\texttt{EOC})
\begin{align*}
\texttt{EOC}(e^n)\coloneqq\frac{\log\big(\frac{e^n}{e^{n-1}}\big)}{\log\big(\frac{h_n+\tau_n}{h_{n-1}+\tau_{n-1}}\big)}\,,\quad
n=1,\dots,6\,, 
\end{align*}
where $e^n$, $n=1,\dots,6$, either denote
$e_{\bfF}^n$, $e_{\jump{}}^n$, $e_{\bfF^*}^n$, $e_{L^2}^n$, or $e_q^n$, $n=1,\dots,6$, respectively, is recorded.

For $p\in \{2.0, 2.5, 3.0\}$, $\rho\in \{0.05,0.1,0.2\}$ and a series of triangulations $\mathcal{T}_{h_i}$, $i= 1,\dots,6$, obtained by
regular, global refinement as described above, the EOC is computed and
presented in Table~\ref{tab1}, Table~\ref{tab2}, Table~\ref{tab3}, Table~\ref{tab4}, and Table~\ref{tab5}, respectively.
For the error quantities \eqref{eq:error_quantities}$_{1,2,3,5}$,   we observe the expected
convergence ratios of about
$\texttt{EOC}_n(e_{\bfF}^n)\approx \texttt{EOC}_n(e_{\jump{}}^n)\approx\texttt{EOC}_n(e_{\bfF^*}^n)\approx\texttt{EOC}_n(e_q^n)\approx \rho\frac{p'}{2} $. 
For \eqref{eq:error_quantities}$_4$, we observe the improved convergence ratio of about
 $\texttt{EOC}_n(e_{L^2}^n)\approx 1.0$, $n=3,\dots, 6$. 
 As a whole, the experimental results indicate the convergence of the  scheme \eqref{eq:p-NSnon.1} (cf. Algorithm \ref{eq:p-NSnon}). 

\begin{table}[H]
\setlength\tabcolsep{8pt}
\centering
\begin{tabular}{c |c|c|c|c|c|c|c|c|c|} \cmidrule(){1-10}
\multicolumn{1}{|c||}{\cellcolor{lightgray}$\rho$}	& \multicolumn{3}{c||}{\cellcolor{lightgray}0.05} & \multicolumn{3}{c||}{\cellcolor{lightgray}0.1}    & \multicolumn{3}{c|}{\cellcolor{lightgray}0.2}\\ 
\hline 

\multicolumn{1}{|c||}{\cellcolor{lightgray}\diagbox[height=1.1\line]{\vspace*{-0.6mm}\footnotesize$n$}{\\[-5.5mm]\footnotesize $p$}}
& \cellcolor{lightgray}2.0 & \cellcolor{lightgray}2.5  & \multicolumn{1}{c||}{\cellcolor{lightgray}3.0} & \multicolumn{1}{c|}{\cellcolor{lightgray}2.0}   & \cellcolor{lightgray}2.5  &  \multicolumn{1}{c||}{\cellcolor{lightgray}3.0}     & \multicolumn{1}{c|}{\cellcolor{lightgray}2.0}    & \cellcolor{lightgray}2.5  & \cellcolor{lightgray}3.0 \\ \hline\hline
\multicolumn{1}{|c||}{\cellcolor{lightgray}$3$}   & 0.055 & 0.019 & \multicolumn{1}{c||}{-0.03} & \multicolumn{1}{c|}{0.103} & 0.059 & \multicolumn{1}{c||}{-0.01} & \multicolumn{1}{c|}{0.198} & 0.137 & 0.045 \\ \hline
\multicolumn{1}{|c||}{\cellcolor{lightgray}$4$}   & 0.054 & 0.045 & \multicolumn{1}{c||}{0.031} & \multicolumn{1}{c|}{0.104} & 0.089 & \multicolumn{1}{c||}{0.066} & \multicolumn{1}{c|}{0.202} & 0.176 & 0.137 \\ \hline
\multicolumn{1}{|c||}{\cellcolor{lightgray}$5$}   & 0.053 & 0.049 & \multicolumn{1}{c||}{0.040} & \multicolumn{1}{c|}{0.103} & 0.094 & \multicolumn{1}{c||}{0.077} & \multicolumn{1}{c|}{0.202} & 0.183 & 0.151 \\ \hline
\multicolumn{1}{|c||}{\cellcolor{lightgray}$6$}   & 0.052 & 0.048 & \multicolumn{1}{c||}{0.040} & \multicolumn{1}{c|}{0.102} & 0.094   & \multicolumn{1}{c||}{0.078} & \multicolumn{1}{c|}{0.202} & 0.183 & 0.153 \\ \hline\hline
\multicolumn{1}{|c||}{\cellcolor{lightgray}$\rho\frac{p'}{2}\!$}     & 0.050 & 0.042 & \multicolumn{1}{c||}{0.0375} & \multicolumn{1}{c|}{0.100} & $0.083$ & \multicolumn{1}{c||}{0.075} & \multicolumn{1}{c|}{0.200} & 0.167 & 0.150 \\ \hline
\end{tabular}\vspace{-2mm}	\caption{\small Experimental order of convergence: $\texttt{EOC}_n(e_{\bfF}^n)$,~${n=3,\dots,6}$.}
\label{tab1}
\end{table}\vspace{-5mm}

\begin{table}[H]
\setlength\tabcolsep{8pt}
\centering
\begin{tabular}{c |c|c|c|c|c|c|c|c|c|} \cmidrule(){1-10}
\multicolumn{1}{|c||}{\cellcolor{lightgray}$\rho$}	& \multicolumn{3}{c||}{\cellcolor{lightgray}0.05} & \multicolumn{3}{c||}{\cellcolor{lightgray}0.1}    & \multicolumn{3}{c|}{\cellcolor{lightgray}0.2}\\ 
\hline 

\multicolumn{1}{|c||}{\cellcolor{lightgray}\diagbox[height=1.1\line]{\vspace*{-0.6mm}\footnotesize$n$}{\\[-5.5mm]\footnotesize $p$}}
& \cellcolor{lightgray}2.0 & \cellcolor{lightgray}2.5  & \multicolumn{1}{c||}{\cellcolor{lightgray}3.0} & \multicolumn{1}{c|}{\cellcolor{lightgray}2.0}   & \cellcolor{lightgray}2.5  &  \multicolumn{1}{c||}{\cellcolor{lightgray}3.0}     & \multicolumn{1}{c|}{\cellcolor{lightgray}2.0}    & \cellcolor{lightgray}2.5  & \cellcolor{lightgray}3.0 \\ \hline\hline
\multicolumn{1}{|c||}{\cellcolor{lightgray}$3$}   & 0.084 & 0.061 & \multicolumn{1}{c||}{-0.03} & \multicolumn{1}{c|}{0.135} & 0.109 & \multicolumn{1}{c||}{0.000} & \multicolumn{1}{c|}{0.237} & 0.204 & 0.049 \\ \hline
\multicolumn{1}{|c||}{\cellcolor{lightgray}$4$}   & 0.064 & 0.057 & \multicolumn{1}{c||}{0.033} & \multicolumn{1}{c|}{0.114} & 0.103 & \multicolumn{1}{c||}{0.069} & \multicolumn{1}{c|}{0.215} & 0.196 & 0.139 \\ \hline
\multicolumn{1}{|c||}{\cellcolor{lightgray}$5$}   & 0.056 & 0.052 & \multicolumn{1}{c||}{0.041} & \multicolumn{1}{c|}{0.106} & 0.098 & \multicolumn{1}{c||}{0.078} & \multicolumn{1}{c|}{0.206} & 0.189 & 0.152 \\ \hline
\multicolumn{1}{|c||}{\cellcolor{lightgray}$6$}   & 0.053 & 0.049 & \multicolumn{1}{c||}{0.040} & \multicolumn{1}{c|}{0.103} & 0.095   & \multicolumn{1}{c||}{0.078} & \multicolumn{1}{c|}{0.203} & 0.186 & 0.153 \\ \hline\hline
\multicolumn{1}{|c||}{\cellcolor{lightgray}$\rho\frac{p'}{2}\!$}     & 0.050 & 0.042 & \multicolumn{1}{c||}{0.0375} & \multicolumn{1}{c|}{0.100} & $0.083$ & \multicolumn{1}{c||}{0.075} & \multicolumn{1}{c|}{0.200} & 0.167 & 0.150 \\ \hline
\end{tabular}\vspace{-2mm}	\caption{\small Experimental order of convergence: $\texttt{EOC}_n(e_{\jump{}}^n)$,~${n=3,\dots,6}$.}
\label{tab2}
\end{table}\vspace{-5mm}

\begin{table}[H]
\setlength\tabcolsep{8pt}
\centering
\begin{tabular}{c |c|c|c|c|c|c|c|c|c|} \cmidrule(){1-10}
\multicolumn{1}{|c||}{\cellcolor{lightgray}$\rho$}	& \multicolumn{3}{c||}{\cellcolor{lightgray}0.05} & \multicolumn{3}{c||}{\cellcolor{lightgray}0.1}    & \multicolumn{3}{c|}{\cellcolor{lightgray}0.2}\\ 
\hline 

\multicolumn{1}{|c||}{\cellcolor{lightgray}\diagbox[height=1.1\line]{\vspace*{-0.6mm}\footnotesize$n$}{\\[-5.5mm]\footnotesize $p$}}
& \cellcolor{lightgray}2.0 & \cellcolor{lightgray}2.5  & \multicolumn{1}{c||}{\cellcolor{lightgray}3.0} & \multicolumn{1}{c|}{\cellcolor{lightgray}2.0}   & \cellcolor{lightgray}2.5  &  \multicolumn{1}{c||}{\cellcolor{lightgray}3.0}     & \multicolumn{1}{c|}{\cellcolor{lightgray}2.0}    & \cellcolor{lightgray}2.5  & \cellcolor{lightgray}3.0 \\ \hline\hline
\multicolumn{1}{|c||}{\cellcolor{lightgray}$3$}   & 0.055 & 0.019 & \multicolumn{1}{c||}{-0.03} & \multicolumn{1}{c|}{0.103} & 0.059 & \multicolumn{1}{c||}{0.000} & \multicolumn{1}{c|}{0.198} & 0.137 & 0.050 \\ \hline
\multicolumn{1}{|c||}{\cellcolor{lightgray}$4$}   & 0.054 & 0.045 & \multicolumn{1}{c||}{0.031} & \multicolumn{1}{c|}{0.104} & 0.089 & \multicolumn{1}{c||}{0.067} & \multicolumn{1}{c|}{0.202} & 0.176 & 0.138 \\ \hline
\multicolumn{1}{|c||}{\cellcolor{lightgray}$5$}   & 0.053 & 0.049 & \multicolumn{1}{c||}{0.040} & \multicolumn{1}{c|}{0.103} & 0.094 & \multicolumn{1}{c||}{0.077} & \multicolumn{1}{c|}{0.202} & 0.184 & 0.151 \\ \hline
\multicolumn{1}{|c||}{\cellcolor{lightgray}$6$}   & 0.052 & 0.048 & \multicolumn{1}{c||}{0.040} & \multicolumn{1}{c|}{0.102} & 0.094   & \multicolumn{1}{c||}{0.078} & \multicolumn{1}{c|}{0.202} & 0.184 & 0.153 \\ \hline\hline
\multicolumn{1}{|c||}{\cellcolor{lightgray}$\rho\frac{p'}{2}\!$}     & 0.050 & 0.042 & \multicolumn{1}{c||}{0.0375} & \multicolumn{1}{c|}{0.100} & $0.083$ & \multicolumn{1}{c||}{0.075} & \multicolumn{1}{c|}{0.200} & 0.167 & 0.150 \\ \hline
\end{tabular}\vspace{-2mm}	\caption{\small Experimental order of convergence: $\texttt{EOC}_n(e_{\bfF^*}^n)$,~${n=3,\dots,6}$.}
\label{tab3}
\end{table}\vspace{-5mm}

\begin{table}[H]
\setlength\tabcolsep{8pt}
\centering
\begin{tabular}{c |c|c|c|c|c|c|c|c|c|} \cmidrule(){1-10}
\multicolumn{1}{|c||}{\cellcolor{lightgray}$\rho$}	& \multicolumn{3}{c||}{\cellcolor{lightgray}0.05} & \multicolumn{3}{c||}{\cellcolor{lightgray}0.1}    & \multicolumn{3}{c|}{\cellcolor{lightgray}0.2}\\ 
\hline 

\multicolumn{1}{|c||}{\cellcolor{lightgray}\diagbox[height=1.1\line]{\vspace*{-0.6mm}\footnotesize$n$}{\\[-5.5mm]\footnotesize $p$}}
& \cellcolor{lightgray}2.0 & \cellcolor{lightgray}2.5  & \multicolumn{1}{c||}{\cellcolor{lightgray}3.0} & \multicolumn{1}{c|}{\cellcolor{lightgray}2.0}   & \cellcolor{lightgray}2.5  &  \multicolumn{1}{c||}{\cellcolor{lightgray}3.0}     & \multicolumn{1}{c|}{\cellcolor{lightgray}2.0}    & \cellcolor{lightgray}2.5  & \cellcolor{lightgray}3.0 \\ \hline\hline
\multicolumn{1}{|c||}{\cellcolor{lightgray}$3$}   & 0.885 & 1.040 & \multicolumn{1}{c||}{1.121} & \multicolumn{1}{c|}{0.935} & 1.070 & \multicolumn{1}{c||}{1.149} & \multicolumn{1}{c|}{1.009} & 1.131 & 1.186 \\ \hline
\multicolumn{1}{|c||}{\cellcolor{lightgray}$4$}   & 0.949 & 1.116 & \multicolumn{1}{c||}{1.180} & \multicolumn{1}{c|}{1.002} & 1.151 & \multicolumn{1}{c||}{1.212} & \multicolumn{1}{c|}{1.080} & 1.216 & 1.259 \\ \hline
\multicolumn{1}{|c||}{\cellcolor{lightgray}$5$}   & 0.998 & 1.149 & \multicolumn{1}{c||}{1.240} & \multicolumn{1}{c|}{1.045} & 1.189 & \multicolumn{1}{c||}{1.268} & \multicolumn{1}{c|}{1.137} & 1.265 & 1.321 \\ \hline
\multicolumn{1}{|c||}{\cellcolor{lightgray}$6$}   & 1.025 & 1.149 & \multicolumn{1}{c||}{1.247} & \multicolumn{1}{c|}{1.074} & 1.192   & \multicolumn{1}{c||}{1.281} & \multicolumn{1}{c|}{1.170} & 1.275   & 1.346 \\ \hline\hline
\multicolumn{1}{|c||}{\cellcolor{lightgray}$\rho\frac{p'}{2}\!$}     & 0.050 & 0.042 & \multicolumn{1}{c||}{0.0375} & \multicolumn{1}{c|}{0.100} & $0.083$ & \multicolumn{1}{c||}{0.075} & \multicolumn{1}{c|}{0.200} & 0.167 & 0.150 \\ \hline
\end{tabular}\vspace{-2mm}	\caption{\small Experimental order of convergence: $\texttt{EOC}_n(e_{L^2}^n)$,~${n=3,\dots,6}$.}
\label{tab4}
\end{table}\vspace{-5mm}

\begin{table}[H]
\setlength\tabcolsep{8pt}
\centering
\begin{tabular}{c |c|c|c|c|c|c|c|c|c|} \cmidrule(){1-10}
\multicolumn{1}{|c||}{\cellcolor{lightgray}$\rho$}	& \multicolumn{3}{c||}{\cellcolor{lightgray}0.05} & \multicolumn{3}{c||}{\cellcolor{lightgray}0.1}    & \multicolumn{3}{c|}{\cellcolor{lightgray}0.2}\\ 
\hline 

\multicolumn{1}{|c||}{\cellcolor{lightgray}\diagbox[height=1.1\line]{\vspace*{-0.6mm}\footnotesize$n$}{\\[-5.5mm]\footnotesize $p$}}
& \cellcolor{lightgray}2.0 & \cellcolor{lightgray}2.5  & \multicolumn{1}{c||}{\cellcolor{lightgray}3.0} & \multicolumn{1}{c|}{\cellcolor{lightgray}2.0}   & \cellcolor{lightgray}2.5  &  \multicolumn{1}{c||}{\cellcolor{lightgray}3.0}     & \multicolumn{1}{c|}{\cellcolor{lightgray}2.0}    & \cellcolor{lightgray}2.5  & \cellcolor{lightgray}3.0 \\ \hline\hline
\multicolumn{1}{|c||}{\cellcolor{lightgray}$3$}   & 0.103 & 0.107 &  \multicolumn{1}{c||}{0.165} & \multicolumn{1}{c|}{0.155} & 0.156 & \multicolumn{1}{c||}{0.219} & \multicolumn{1}{c|}{0.261} & 0.255 & 0.326 \\ \hline
\multicolumn{1}{|c||}{\cellcolor{lightgray}$4$}   & 0.071 & 0.072 & \multicolumn{1}{c||}{0.089} & \multicolumn{1}{c|}{0.121} & 0.122 & \multicolumn{1}{c||}{0.140} & \multicolumn{1}{c|}{0.222} & 0.221 & 0.242 \\ \hline
\multicolumn{1}{|c||}{\cellcolor{lightgray}$5$}   & 0.060 & 0.059 & \multicolumn{1}{c||}{0.064} & \multicolumn{1}{c|}{0.109} & 0.109 & \multicolumn{1}{c||}{0.115} & \multicolumn{1}{c|}{0.209} & 0.210 & 0.216 \\ \hline
\multicolumn{1}{|c||}{\cellcolor{lightgray}$6$}   & 0.055 & 0.054 & \multicolumn{1}{c||}{0.056} & \multicolumn{1}{c|}{0.104} & 0.104   & \multicolumn{1}{c||}{0.106} & \multicolumn{1}{c|}{0.204} & 0.205 & 0.207 \\ \hline\hline
\multicolumn{1}{|c||}{\cellcolor{lightgray}$\rho\frac{p'}{2}\!$}     & 0.050 & 0.042 & \multicolumn{1}{c||}{0.0375} & \multicolumn{1}{c|}{0.100} & $0.083$ & \multicolumn{1}{c||}{0.075} & \multicolumn{1}{c|}{0.200} & 0.167 & 0.150 \\ \hline
\end{tabular}\vspace{-2mm}	\caption{\small Experimental order of convergence: $\texttt{EOC}_n(e_q^n)$,~${n=3,\dots,6}$.}
\label{tab5}
\end{table}


\def\cprime{$'$} \def\cprime{$'$} \def\cprime{$'$}
\ifx\undefined\bysame
\newcommand{\bysame}{\leavevmode\hbox to3em{\hrulefill}\,}
\fi

\end{document}